\documentclass[12pt]{amsart}

\usepackage{etex,amsfonts, amsmath, amsthm, amssymb, stmaryrd, epsfig, graphics,
  psfrag, latexsym, color, mathtools, mathrsfs,enumitem, subfig,
  longtable, booktabs, yfonts} \usepackage[all,cmtip]{xy}
\usepackage[percent]{overpic}
\definecolor{darkblue}{rgb}{0,0,0.4} \usepackage[colorlinks=true,
citecolor=darkblue, filecolor=darkblue, linkcolor=darkblue,
urlcolor=darkblue]{hyperref} \usepackage[all]{hypcap}
\usepackage{xr-hyper}

\graphicspath{{draws/}{}}

\numberwithin{equation}{section}
\newtheorem{thm}{Theorem}
\newtheorem{theorem}[thm]{Theorem}
\newtheorem{lem}{Lemma}[section]               
\newtheorem{lemma}[lem]{Lemma}               

\newtheorem{cor}[lem]{Corollary}
\newtheorem{corollary}[lem]{Corollary}               
\newtheorem{prop}[lem]{Proposition}
\newtheorem{proposition}[lem]{Proposition}
\theoremstyle{definition}     

\newtheorem{question}[lem]{Question}
\newtheorem{defn}[lem]{Definition} 
\newtheorem{definition}[lem]{Definition}

\theoremstyle{remark}     
\newtheorem{rem}{Remark}[section]
\newtheorem{remark}[rem]{Remark}

\newtheorem{example}[rem]{Example}
\numberwithin{figure}{section}

\newcommand{\Appendix}[1]{\hyperref[app:#1]{Appendix~\ref*{app:#1}}}
\newcommand{\Section}[1]{\hyperref[sec:#1]{Section~\ref*{sec:#1}}}
\newcommand{\Subsection}[1]{\hyperref[subsec:#1]{Subsection~\ref*{subsec:#1}}}
\newcommand{\Lemma}[1]{\hyperref[lem:#1]{Lemma~\ref*{lem:#1}}}
\newcommand{\Theorem}[1]{\hyperref[thm:#1]{Theorem~\ref*{thm:#1}}}
\newcommand{\Definition}[1]{\hyperref[def:#1]{Definition~\ref*{def:#1}}}
\newcommand{\Remark}[1]{\hyperref[rem:#1]{Remark~\ref*{rem:#1}}}
\newcommand{\Figure}[1]{\hyperref[fig:#1]{Figure~\ref*{fig:#1}}}
\newcommand{\Conjecture}[1]{\hyperref[conj:#1]{Conjecture~\ref*{conj:#1}}}
\newcommand{\Corollary}[1]{\hyperref[cor:#1]{Corollary~\ref*{cor:#1}}}
\newcommand{\Proposition}[1]{\hyperref[prop:#1]{Proposition~\ref*{prop:#1}}}
\newcommand{\Question}[1]{\hyperref[ques:#1]{Question~\ref*{ques:#1}}}
\newcommand{\Example}[1]{\hyperref[exam:#1]{Example~\ref*{exam:#1}}}
\newcommand{\Table}[1]{\hyperref[table:#1]{Table~\ref*{table:#1}}}
\newcommand{\Restric}[1]{\hyperref[restric:#1]{Restriction~\ref*{restric:#1}}}

\newcommand{\Equation}[2][{}]{Equation#1~(\ref{eq:#2})}

\newcommand{\Case}[2][{}]{Case#1~(\ref{case:#2})}

\newcommand{\Item}[1]{(\ref{item:#1})}

\newcommand{\R}{\mathbb{R}}
\newcommand{\Z}{\mathbb{Z}}
\newcommand{\F}{\mathbb{F}}

\newcommand{\N}{\mathbb{N}}

\newcommand{\mf}{\mathfrak}

\newcommand{\wt}{\widetilde}
\newcommand{\ol}{\overline}

\newcommand{\del}{\partial}
\newcommand{\sbs}{\subset}

\newcommand{\sm}{\setminus}

\newcommand{\Si}{\Sigma}
\newcommand{\si}{\sigma}
\newcommand{\al}{\alpha}
\newcommand{\be}{\beta}

\newcommand{\ep}{\epsilon}

\newcommand{\from}{\colon}
\newcommand{\into}{\hookrightarrow}

\newcommand{\set}[2]{\{#1\mid#2\}}
\newcommand{\restrict}[2]{{#1}|_{#2}}

\renewcommand{\th}{^{\text{th}}}

\DeclareMathOperator{\nbd}{nbd}

\DeclareMathOperator{\Id}{Id}
\DeclareMathOperator{\Sq}{Sq}

\DeclareMathOperator{\Hom}{Hom}

\DeclareMathOperator{\sq}{sq}

\newcommand{\Kh}{\mathit{Kh}}

\newcommand{\KhCx}{\mathit{KC}}

\newcommand{\Funky}{\mathscr{F}}

\newcommand{\Realize}[2][{}]{|#2|_{#1}}
\newcommand{\Codim}[1]{\langle#1\rangle}

\newcommand{\CubeFlowCat}{\mathscr{C}_C}
\newcommand{\KhFlowCat}{\mathscr{C}_K}

\newcommand{\Moduli}{\mathcal{M}}

\newcommand{\ob}[1]{\mathbf{#1}}
\newcommand{\gr}{\mathrm{gr}}
\newcommand{\intgr}{\gr_{q}}
\newcommand{\homgr}{\gr_{h}}

\newcommand{\KhSpace}{\mathcal{X}_\mathit{Kh}}

\newcommand{\Cell}[2][{}]{\mathcal{C}_{#1}(#2)}

\newcommand{\Frame}{\varphi}
\newcommand{\Cube}{\mathcal{C}}
\newcommand{\imprec}{\leq_1}

\newcommand{\AssRes}[2]{D_{#1}(#2)}

\newcommand{\diff}{\delta}
\newcommand{\vect}{\overline}

\newcommand{\matching}{\leftrightarrow}

\newcommand{\MatchBdy}{\mathfrak{m}}
\newcommand{\KhGen}{\mathit{KG}}
\newcommand{\PreBdy}[2][{}]{\mathcal{G}_{#1}(#2)}
\newcommand{\BijecMatch}[1][{}]{\mathfrak{b}_{#1}}
\newcommand{\SignMatch}[1][{}]{\mathfrak{s}_{#1}}
\newcommand{\homology}[1]{[#1]}
\newcommand{\InBetween}[2]{\mathcal{G}_{{#1},{#2}}}
\newcommand{\MatchLady}[1][{}]{\mathfrak{l}_{#1}}
\newcommand{\Graph}[2][{}]{\mathfrak{G}_{#1}(#2)}
\newcommand{\graphnumber}[1]{g(#1)}
\newcommand{\framenumber}[1]{f(#1)}
\newcommand{\Steenrod}{\mathcal{A}}
\newcommand{\NonCommExtend}[1]{\{{#1}\}}

\newcommand{\mirror}[1]{m(#1)}

\newcommand{\co}{\colon}
\newcommand{\bdy}{\partial}
\newcommand{\RR}{\R}
\newcommand{\DD}{\mathbb{D}}

\newcommand{\pt}{\mathrm{pt}}

\newcommand{\ZZ}{\mathbb{Z}}

\DeclareMathOperator{\image}{im}

\newcommand{\respectively}{resp.\ }
\newcommand{\RP}{\mathbb{R}\mathrm{P}}
\newcommand{\CP}{\mathbb{C}\mathrm{P}}
\DeclareMathOperator{\SO}{\mathit{SO}}

\newcommand{\cmap}[2][{}]{{\mathfrak{#2}^{#1}}} 

\newcommand{\Path}[1]{\overline{#1}}
\newcommand{\St}{\mathit{St}}
\newcommand{\bma}{\eta}
\newcommand{\bmat}{\tilde\eta}
\newcommand{\PT}{\zeta}

\newcommand{\fr}{\mathrm{fr}}
\newcommand{\rmap}{\mathfrak{r}}
\newcommand{\smap}{\mathfrak{s}}

\newcommand{\topright}{{}^+_+}
\newcommand{\bottomright}{{}^-_+}
\newcommand{\topleft}{{}^+_-}
\newcommand{\bottomleft}{{}^-_-}

\newcommand*{\defeq}{\mathrel{\vcenter{\baselineskip0.5ex \lineskiplimit0pt
                     \hbox{\scriptsize.}\hbox{\scriptsize.}}}%
                     =}

\newcommand{\rel}{relative\ }

\begin{document}

\title[Steenrod Square on Khovanov Homology]{A Steenrod Square on Khovanov Homology}

\author{Robert Lipshitz}
\thanks{RL was supported by an NSF grant number DMS-0905796 and a Sloan Research Fellowship.}
\email{\href{mailto:lipshitz@math.columbia.edu}{lipshitz@math.columbia.edu}}

\author{Sucharit Sarkar}
\thanks{SS was supported by a Clay Mathematics Institute Research Fellowship}
\email{\href{mailto:sucharit@math.columbia.edu}{sucharit@math.columbia.edu}}

\subjclass[2010]{\href{http://www.ams.org/mathscinet/search/mscdoc.html?code=57M25,55P42}{57M25,
    55P42}}

\address{Department of Mathematics, Columbia University, New York, NY 10027}
\keywords{}

\date{\today}

\begin{abstract}
  In a previous paper, we defined a space-level version $\KhSpace(L)$
  of Khovanov homology. This induces an action of the Steenrod algebra
  on Khovanov homology. In this paper, we describe the first
  interesting operation, $\Sq^2\co \Kh^{i,j}(L)\to\Kh^{i+2,j}(L)$. We
  compute this operation for all links up to $11$ crossings; this, in
  turn, determines the stable homotopy type of $\KhSpace(L)$ for all
  such links.
\end{abstract}

\maketitle

\tableofcontents

\section{Introduction}
Khovanov homology, a categorification of the Jones polynomial,
associates a bigraded abelian group $\Kh^{i,j}_{\Z}(L)$ to each link $L\subset
S^3$ \cite{Kho-kh-categorification}.  In~\cite{RS-khovanov} we gave a
space-level version of Khovanov homology. That is, to each link $L$ we
associated stable spaces (finite suspension spectra) $\KhSpace^j(L)$,
well-defined up to stable homotopy equivalence, so that the reduced
cohomology $\wt{H}^i(\KhSpace^j(L))$ of these spaces is the Khovanov
homology $\Kh^{i,j}_{\Z}(L)$ of $L$. Another construction of such
spaces has been given by~\cite{HKK-Kh-htpy}.

The space $\KhSpace^j(L)$ gives algebraic
structures on Khovanov homology which are not (yet) apparent from
other perspectives. Specifically, while the cohomology of a spectrum
does not have a cup product, it does carry stable cohomology
operations. The bulk of this paper is devoted to giving an explicit
description of the Steenrod square
\[
\Sq^2\co \Kh^{i,j}_{\F_2}(L)\to \Kh^{i+2,j}_{\F_2}(L)
\]
induced by the spectrum $\KhSpace^j(L)$.  First we give a
combinatorial definition of this operation $\Sq^2$ in \Section{answer}
and then prove that it agrees with the Steenrod square coming from
$\KhSpace^j(L)$ in \Section{where-answer}. 

The description is suitable for computer computation, and we have
implemented it in Sage. The results for links with $11$ or fewer
crossings are given in \Section{computations}. In particular, the
operation $\Sq^2$ is nontrivial for many links, such as the torus knot
$T_{3,4}$. This implies a nontriviality result for the
Khovanov space:
\begin{theorem}\label{thm:not-moore}
  The Khovanov homotopy type $\KhSpace^{11}(T_{3,4})$ is not a wedge sum
  of Moore spaces.
\end{theorem}

Even simpler than $\Sq^2$ is the operation $\Sq^1\co
\Kh^{i,j}_{\F_2}(L)\to \Kh^{i+1,j}_{\F_2}(L)$. Let
$\be\from\Kh^{i,j}_{\F_2}(L)\to \Kh^{i+1,j}_{\Z}(L)$ be the Bockstein
homomorphism, and $r\from \Kh^{i,j}_{\Z}(L)\to\Kh^{i,j}_{\F_2}(L)$ be
the reduction mod $2$. Then $\Sq^1=r\be$, and is thus determined by
the integral Khovanov homology; see also \Subsection{sq1-describe}. As
we will discuss in \Section{width-two}, the operations $\Sq^1$ and
$\Sq^2$ together determine the Khovanov homotopy type $\KhSpace(L)$
whenever the Khovanov homology of $L$ has a sufficiently simple
form. In particular, they determine $\KhSpace(L)$ for any link $L$ of
$11$ or fewer crossings; these homotopy types are listed in
\Table{computations} (\Section{computations}).

The subalgebra of the Steenrod algebra generated by $\Sq^1$ and $\Sq^2$ is
\[
\Steenrod(1)=\frac{\F_2\NonCommExtend{\Sq^1,\Sq^2}}{(\Sq^1)^2,(\Sq^2)^2+\Sq^1\Sq^2\Sq^1}
\]
where $\F_2\NonCommExtend{\Sq^1,\Sq^2}$ is the non-commuting extension
of $\F_2$ by the variables $\Sq^1$ and $\Sq^2$.  By the Adem
relations, the next Steenrod square $\Sq^3$ is determined by $\Sq^1$
and $\Sq^2$, viz.\ $\Sq^3=\Sq^1\Sq^2$. Therefore, the next interesting
Steenrod square to compute would be $\Sq^4\co \Kh^{i,j}_{\F_2}(L)\to
\Kh^{i+4,j}_{\F_2}(L)$.

The Bockstein $\be$ and the operation $\Sq^2$ are sometimes enough to
compute Khovanov $K$-theory: in the Atiyah-Hirzebruch spectral
sequence for $K$-theory, the $d_2$ differential is zero and the $d_3$
differential is the integral lift $\be\Sq^2 r$ of $\Sq^3$ (see, for
instance~\cite[Proposition 16.6]{Adams-top-book}
or~\cite{Law-top-overflow}). For grading reasons, this operation
vanishes for links with $11$ or fewer crossings; indeed, the
Atiyah-Hirzebruch spectral sequence degenerates in these cases, and
the Khovanov $K$-theory is just the tensor product of the Khovanov
homology and $K^*(\pt)$. In principle, however, the techniques of this
paper could be used to compute Khovanov $K$-theory in some interesting
cases. Similarly, in certain situations, the Adams
spectral sequence may be used to compute the real connective Khovanov
$KO$-theory using merely the module structure of Khovanov homology
$\Kh_{\F_2}(L)$ over $\Steenrod(1)$.

\thinspace\thinspace\thinspace
\noindent
\textbf{Acknowledgements.} We thank C.~Seed and M.~Khovanov for
helpful conversations. We also thank the referee for many helpful
suggestions and corrections.

\section{The answer}\label{sec:answer}

\subsection{Sign and frame assignments on the cube}\label{subsec:cube}
Consider the $n$-dimensional cube $\Cube(n)=[0,1]^n$, equipped with
the natural CW complex structure. For a vertex
$v=(v_1,\dots,v_n)\in\{0,1\}^n$, let $|v|=\sum_i v_i$ denote the
Manhattan norm of $v$. For vertices $u,v$, declare $v\leq u$ if for
all $i$, $v_i\leq u_i$; if $v\leq u$ and $|u-v|=k$, we write $v\leq_k
u$.

For a pair of vertices $v\leq_k u$, let $\Cube_{u,v}=\set{x\in
  [0,1]^n}{\forall i\co v_i\leq x_i\leq u_i}$ denote the corresponding
$k$-cell of $\Cube(n)$. Let $C^*(\Cube(n),\F_2)$ denote the cellular
cochain complex of $\Cube(n)$ over $\F_2$. Let $1_k\in
C^k(\Cube(n),\F_2)$ denote the $k$-cocycle that sends all $k$-cells to $1$.

The \emph{standard sign assignment $s\in C^1(\Cube(n),\F_2)$} (denoted
$s_0$ in \cite[Definition~\ref*{KhSp:def:sign-assignment}]{RS-khovanov}) is the following
$1$-cochain. If 
$u=(\ep_1,\dots,\ep_{i-1},1,\ep_{i+1},\dots,\ep_n)$ and 
$v=(\ep_1,\dots,\ep_{i-1},0,\ep_{i+1},\dots,\ep_n)$, then 
\[
s(\Cube_{u,v})=(\ep_1+\dots+\ep_{i-1})\pmod 2\in\F_2.
\]
It is easy to see that $\diff s=1_2$.

The \emph{standard frame assignment $f\in C^2(\Cube(n),\F_2)$} is the
following $2$-cochain. If
$u=(\ep_1,\dots,\ep_{i-1},1,\ep_{i+1},\dots,\ep_{j-1},1,\ep_{j+1},\dots,\ep_n)$
and
$v=(\ep_1,\dots,\ep_{i-1},0,\ep_{i+1},\dots,\ep_{j-1},0,\ep_{j+1},\dots,\allowbreak\ep_n)$,
then
\[
f(\Cube_{u,v})=(\ep_1+\dots+\ep_{i-1})(\ep_{i+1}+\dots+\ep_{j-1})\pmod
2\in\F_2.
\]
\begin{lem}\label{lem:frame-assignment-sum}
  For any $v\leq_3 u$,
  \[
  (\diff f)(\Cube_{u,v}) = \sum_{w\in\set{w}{v\imprec
      w\leq_2 u}} s(\Cube_{w,v}).
  \]
\end{lem}

\begin{proof}
  Let
  $u=(\ep_1,\dots,\ep_{i-1},1,\ep_{i+1},\dots,\ep_{j-1},1,\ep_{j+1},\dots,\ep_{k-1},1,\ep_{k+1},\dots,\ep_n)$
  and $v=(\ep_1,\dots,\allowbreak\ep_{i-1},\allowbreak
  0,\allowbreak\ep_{i+1},\dots,\ep_{j-1},0,\ep_{j+1},\dots,\ep_{k-1},0,\ep_{k+1},\dots,\ep_n)$. Then,
  \begin{align*}
    \sum_{w\in\set{w}{v\imprec
        w\leq_2 u}} s(\Cube_{w,v})&=
    (\ep_1+\dots+\ep_{i-1})+(\ep_1+\dots+\ep_{j-1})+(\ep_1+\dots+\ep_{k-1})\\
    &=(\ep_1+\dots+\ep_{i-1})+(\ep_{j+1}+\dots+\ep_{k-1}).
  \end{align*}
  On the other hand,
  \begin{align*}
    (\diff
    f)(\Cube_{u,v})&=(\ep_{1}+\dots+\ep_{i-1})(\ep_{i+1}+\dots+\ep_{j-1})+(\ep_{1}+\dots+\ep_{i-1})(\ep_{i+1}+\dots+\ep_{j-1})\\
   &\qquad{}+
    (\ep_{1}+\dots+\ep_{i-1})(\ep_{i+1}+\dots+\ep_{j-1}+0+\ep_{j+1}+\dots+\ep_{k-1})\\
    &\qquad{}+
    (\ep_{1}+\dots+\ep_{i-1})(\ep_{i+1}+\dots+\ep_{j-1}+1+\ep_{j+1}+\dots+\ep_{k-1})\\
    &\qquad{}+
    (\ep_{1}+\dots+\ep_{i-1}+0+\ep_{i+1}+\dots+\ep_{j-1})(\ep_{j+1}+\dots+\ep_{k-1})\\
    &\qquad{}+(\ep_{1}+\dots+\ep_{i-1}+1+\ep_{i+1}+\dots+\ep_{j-1})(\ep_{j+1}+\dots+\ep_{k-1})\\
    &=(\ep_1+\dots+\ep_{i-1})+(\ep_{j+1}+\dots+\ep_{k-1}),
  \end{align*}
  thus completing the proof.
\end{proof}

\subsection{The Khovanov setup}

In this subsection, we recall the definition of the Khovanov chain
complex associated to an oriented link diagram $L$. Assume $L$ has $n$
crossings that have been ordered, and let $n_-$ denote the number of
negative crossings in $L$. In what follows, we will usually work over
$\F_2$, and we will always have a fixed $n$-crossing link diagram $L$
in the background. Hence, we will typically drop both $\F_2$ and $L$
from the notation, writing $\KhCx=\KhCx_{\F_2}(L)$ for the Khovanov
complex of $L$ with $\F_2$-coefficients and $\KhCx_{\Z}$ for the
Khovanov complex of $L$ with $\Z$-coefficients.

Given a vertex $u\in\{0,1\}^n$, let $\AssRes{}{u}$ be the
corresponding complete resolution of the link diagram $L$, where we
take the $0$ resolution at $i\th$ crossing if $u_i=0$, and the
$1$-resolution otherwise. We usually view $\AssRes{}{u}$ as a
\emph{resolution configuration} in the sense of \cite[Definition~\ref*{KhSp:def:res-config}]{RS-khovanov}; that is, we add arcs at the $0$-resolutions to
record the crossings. The
set of circles (\respectively arcs) that appear in $\AssRes{}{u}$ is
denoted $Z(\AssRes{}{u})$ (\respectively $A(\AssRes{}{u})$).

The \emph{Khovanov generators} are of the form $\ob{x}=(\AssRes{}{u},x)$,
where $x$ is a labeling of the circles in $Z(\AssRes{}{u})$ by
elements of $\{x_+,x_-\}$. Each Khovanov generator carries a
\emph{bigrading $(\homgr,\intgr)$}; $\homgr$ is called the homological
grading and $\intgr$ is called the quantum grading. The bigrading is defined by:
\begin{align*}
\homgr(\AssRes{}{u},x)&=-n_-+|u|\\
\intgr(\AssRes{}{u},x)&=n-3n_-+|u|+\#\set{Z\in
  Z(\AssRes{}{u})}{x(Z)=x_+}\\
&\qquad\qquad{}-\#\set{Z\in
  Z(\AssRes{}{u})}{x(Z)=x_-}.
\end{align*}
The set of all Khovanov generators in bigrading $(i,j)$ is denoted
$\KhGen^{i,j}$. There is an obvious map $\Funky\from\KhGen\to\{0,1\}^n$
that sends $(\AssRes{}{u},x)$ to $u$. It is clear that if
$\ob{x}\in\KhGen^{i,j}$, then $|\Funky(\ob{x})|=n_-+i$.

The \emph{Khovanov chain group} in bigrading $(i,j)$, $\KhCx^{i,j}$, is the
$\F_2$ vector space with basis $\KhGen^{i,j}$; for
$\ob{x}\in\KhGen^{i,j}$, and $\ob{c}\in\KhCx^{i,j}$, we say
$\ob{x}\in\ob{c}$ if the coefficient of $\ob{x}$ in $\ob{c}$ is $1$,
and $\ob{x}\notin\ob{c}$ otherwise. 

The \emph{Khovanov differential} $\diff$ maps $\KhCx^{i,j}\to\KhCx^{i+1,j}$,
and is defined as follows. If $\ob{y}=(\AssRes{}{v},y)\in\KhGen^{i,j}$
and $\ob{x}=(\AssRes{}{u},x)\in\KhGen^{i+1,j}$, then
$\ob{x}\in\diff\ob{y}$ if the following hold:
\begin{enumerate}
\item $v\imprec u$, that is, $\AssRes{}{u}$ is obtained from
  $\AssRes{}{v}$ by performing an embedded $1$-surgery along some arc
  $A_1\in A(\AssRes{}{v})$. In particular, either,
  \begin{enumerate}
  \item\label{case:split} the endpoints of $A_1$ lie on the same
    circle, say $Z_1\in\AssRes{}{v}$, which corresponds to two
    circles, say $Z_2,Z_3\in\AssRes{}{u}$; or,
  \item\label{case:merge} The endpoints of $A_1$ lie on two different
    circles, say $Z_1,Z_2\in\AssRes{}{v}$, which correspond to a
    single circle, say $Z_3\in\AssRes{}{u}$.
  \end{enumerate}
\item In \Case{split}, $x$ and $y$ induce the same labeling on
  $\AssRes{}{u}\sm\{Z_2,Z_3\}= \AssRes{}{v}\sm\{Z_1\}$; in
  \Case{merge}, $x$ and $y$ induce the same labeling on
  $\AssRes{}{u}\sm\{Z_3\}= \AssRes{}{v}\sm\{Z_1,Z_2\}$;
\item In \Case{split}, either $y(Z_1)=x(Z_2)=x(Z_3)=x_-$ or
  $y(Z_1)=x_+$ and $\{x(Z_2),x(Z_3)\}=\{x_+,x_-\}$; in \Case{merge},
  either $y(Z_1)=y(Z_2)=x(Z_3)=x_+$ or $\{y(Z_1),y(Z_2)\}=\{x_+,x_-\}$
  and $x(Z_3)=x_-$.
\end{enumerate}
It is clear that if $\ob{x}\in\diff{\ob{y}}$, then
$\Funky(\ob{y})\imprec\Funky(\ob{x})$.  The \emph{Khovanov homology}
is the homology of $(\KhCx,\diff)$; the Khovanov homology
in bigrading $(i,j)$ is denoted $\Kh^{i,j}$. For a cycle
$\ob{c}\in\KhCx^{i,j}$, let $\homology{\ob{c}}\in\Kh^{i,j}$ denote the
corresponding homology element.

\subsection{A first look at the Khovanov space}

The Khovanov chain complex is actually defined over $\Z$, and the
$\F_2$ versions is its mod $2$ reduction. The Khovanov chain group
over $\Z$ in bigrading $(i,j)$, $\KhCx_{\Z}$, is the free $\Z$-module
with basis $\KhGen^{i,j}$. The differential
$\diff_{\Z}\from\KhCx_{\Z}^{i,j}\to \KhCx_{\Z}^{i+1,j}$ is defined by
\begin{equation}\label{eq:integer-kh-diff}
\diff_{\Z}\ob{y}=\sum_{\ob{x}\in\diff\ob{y}}
(-1)^{s(\Cube_{\Funky(\ob{x}),\Funky(\ob{y})})}\ob{x}.
\end{equation}

In \cite[Theorem~\ref*{KhSp:thm:kh-space}]{RS-khovanov}, we construct Khovanov spectra
$\KhSpace^j$ satisfying
$\wt{H}^i(\KhSpace^j)=\Kh^{i,j}_{\Z}$. Moreover, the spectrum
$\KhSpace=\bigvee_j\KhSpace^j$ is defined as the suspension spectrum of a
CW complex $\Realize{\KhFlowCat}$, formally desuspended a few times
\cite[Definition~\ref*{KhSp:def:Kh-space}]{RS-khovanov} (this space is
denoted $Y=\bigvee_jY_j$ in \Section{where-answer}). Furthermore, there is a
bijection between the cells (except the basepoint) of
$\Realize{\KhFlowCat}$ and the Khovanov generators in $\KhGen$, which
induces an isomorphism between $\wt{C}^*(\Realize{\KhFlowCat})$, the
reduced cellular cochain complex, and
$(\KhCx_{\Z},\diff_{\Z})$.

This allows us to associate homotopy invariants to Khovanov
homology. Let $\Steenrod$ be the (graded) Steenrod algebra over
$\F_2$, and let $\Steenrod(1)$ be the subalgebra generated by $\Sq^1$
and $\Sq^2$. The Steenrod algebra $\Steenrod$ acts on the Khovanov
homology $\Kh$, viewed as the (reduced) cohomology of the spectrum
$\KhSpace$. The (stable) homotopy type of $\KhSpace$ is a knot
invariant, and therefore, the action of $\Steenrod$ on $\Kh$ is a knot
invariant as well.

\subsection{The ladybug matching}\label{subsec:ladybug-matching}

Let $\ob{x}\in\KhGen^{i+2,j}$ and $\ob{y}\in\KhGen^{i,j}$ be
Khovanov generators. Consider the set of Khovanov generators between
$\ob{x}$ and $\ob{y}$:
\[
\InBetween{\ob{x}}{\ob{y}}=\set{\ob{z}\in\KhGen^{i+1,j}}{\ob{x}\in\diff\ob{z},\ob{z}\in\diff\ob{y}}.
\]
Since $\diff$ is a differential, for all $\ob{x},\ob{y}$, there are an
even number of elements in $\InBetween{\ob{x}}{\ob{y}}$. It is
well-known that this even number is $0$, $2$ or $4$. Indeed:

\begin{lem}{\cite[Lemma~\ref*{KhSp:lem:ind-2-res-config}]{RS-khovanov}}\label{lem:ladybug-config}
  Let $\ob{x}=(\AssRes{}{u},x)$ and $\ob{y}=(\AssRes{}{v},y)$. The set
  $\InBetween{\ob{x}}{\ob{y}}$ has $4$ elements if and only if the
  following hold.
  \begin{enumerate}
  \item $v\leq_2 u$, that is, $\AssRes{}{u}$ is obtained from
    $\AssRes{}{v}$ by doing embedded $1$-surgeries along two arcs,
    say $A_1,A_2\in A(\AssRes{}{v})$.
  \item The endpoints of $A_1$ and $A_2$ all lie on the same circle,
    say $Z_1\in Z(\AssRes{}{v})$. Furthermore, their endpoints are
    linked on $Z_1$, so $Z_1$ gives rise to a single circle,
    say $Z_2$, in $Z(\AssRes{}{u})$.
  \item $x$ and $y$ agree on
    $Z(\AssRes{}{u})\sm\{Z_2\}=Z(\AssRes{}{v})\sm\{Z_1\}$. 
  \item $y(Z_1)=x_+$ and $x(Z_2)=x_-$.
  \end{enumerate}
\end{lem}

In the construction of the Khovanov space, we made a global
choice. This choice furnishes us with a \emph{ladybug matching
  $\MatchLady$}, which is a collection $\{\MatchLady[\ob{x},\ob{y}]\}$,
for $\ob{x},\ob{y}\in\KhGen$ with $|\Funky(\ob{x})|=|\Funky(\ob{y})|+2$, of
fixed point free involutions $\MatchLady[\ob{x},\ob{y}]\co
\InBetween{\ob{x}}{\ob{y}}\to \InBetween{\ob{x}}{\ob{y}}$.
The ladybug matching is defined as follows.

Fix $\ob{x}=(\AssRes{}{u},x)$ and $\ob{y}=(\AssRes{}{v},y)$ in
$\KhGen$ with $|u|=|v|+2$; we will describe a fixed point free
involution $\MatchLady[\ob{x},\ob{y}]$ of
$\InBetween{\ob{x}}{\ob{y}}$. The only case of interest is when
$\InBetween{\ob{x}}{\ob{y}}$ has $4$ elements; hence assume that we
are in the case described in \Lemma{ladybug-config}. Do an isotopy in
$S^2$ so that $\AssRes{}{v}$ looks like \Figure{ladybug-config}. (In the
figure, we have not shown the circles in $Z(\AssRes{}{v})\sm\{Z_1\}$
and the arcs in $A(\AssRes{}{v})\sm\{A_1,A_2\}$.) \Figure{ladybug-matching}
shows the four generators in $\InBetween{\ob{x}}{\ob{y}}$ and the
ladybug matching $\MatchLady[\ob{x},\ob{y}]$. (Once again, we have not
shown the extra circles and arcs.) It is easy to check (cf.\
\cite[Lemma~\ref*{KhSp:lem:in-out-preserved}]{RS-khovanov}) that this matching is well-defined,
i.e., it is independent of the choice of isotopy and the numbering of
the two arcs in $A(\AssRes{}{v})\sm A(\AssRes{}{u})$ as $\{A_1,A_2\}$.

\captionsetup[subfloat]{width={0.3\textwidth}}
\begin{figure}
  \subfloat[The configuration from \Lemma{ladybug-config}.]{
    \label{fig:ladybug-config}
    \makebox[0.3\textwidth][c]{
      \xymatrix{
        \vcenter{\hbox{
            \psfrag{z}{$x_+$} \psfrag{a1}{$A_1$} \psfrag{a2}{$A_2$}
            \includegraphics[width=0.15\textwidth]{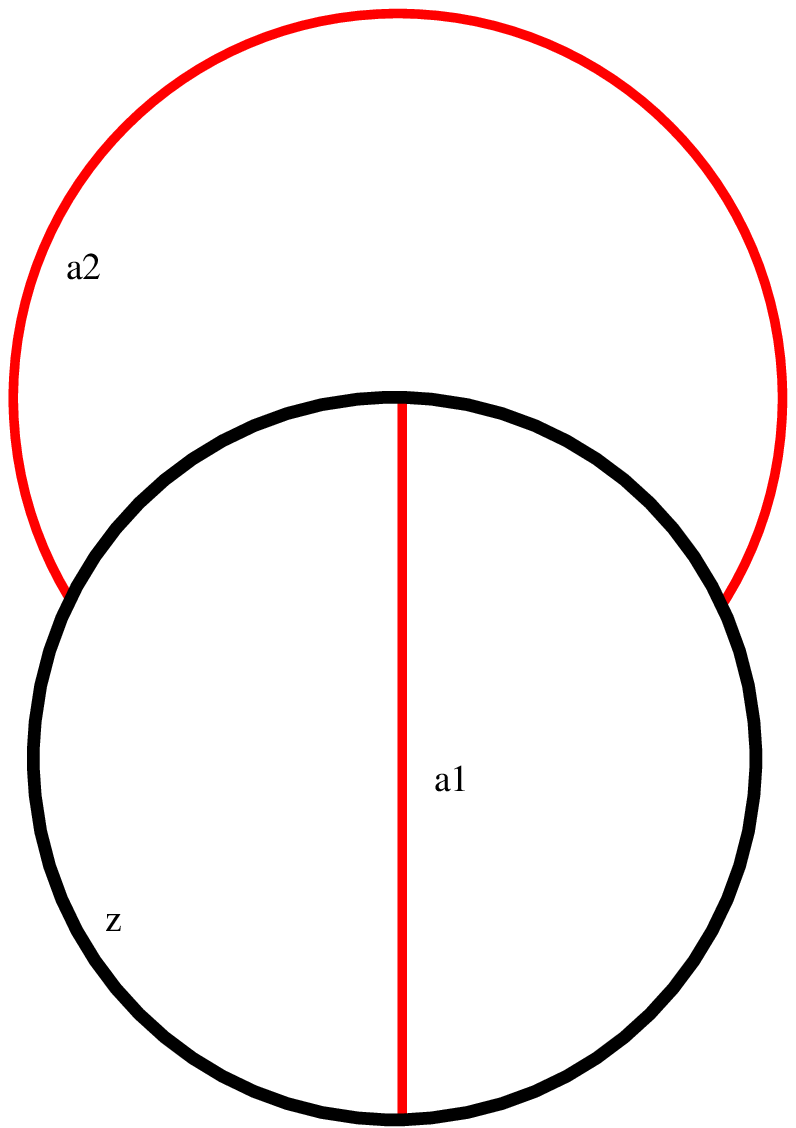}
          }}
      }
    }
  }
  \captionsetup[subfloat]{width={0.5\textwidth}}
  \hspace{0.01\textwidth} 
  \subfloat[{The ladybug matching
    $\MatchLady[\ob{x},\ob{y}]$ on $\InBetween{\ob{x}}{\ob{y}}$.}]{
    \label{fig:ladybug-matching}
    \makebox[0.5\textwidth][c]{
      \xymatrix{
        \vcenter{\hbox{
            \psfrag{a2}{$A_2$} \psfrag{z}{$x_+$}\psfrag{z1}{$x_-$}
            \includegraphics[width=0.15\textwidth]{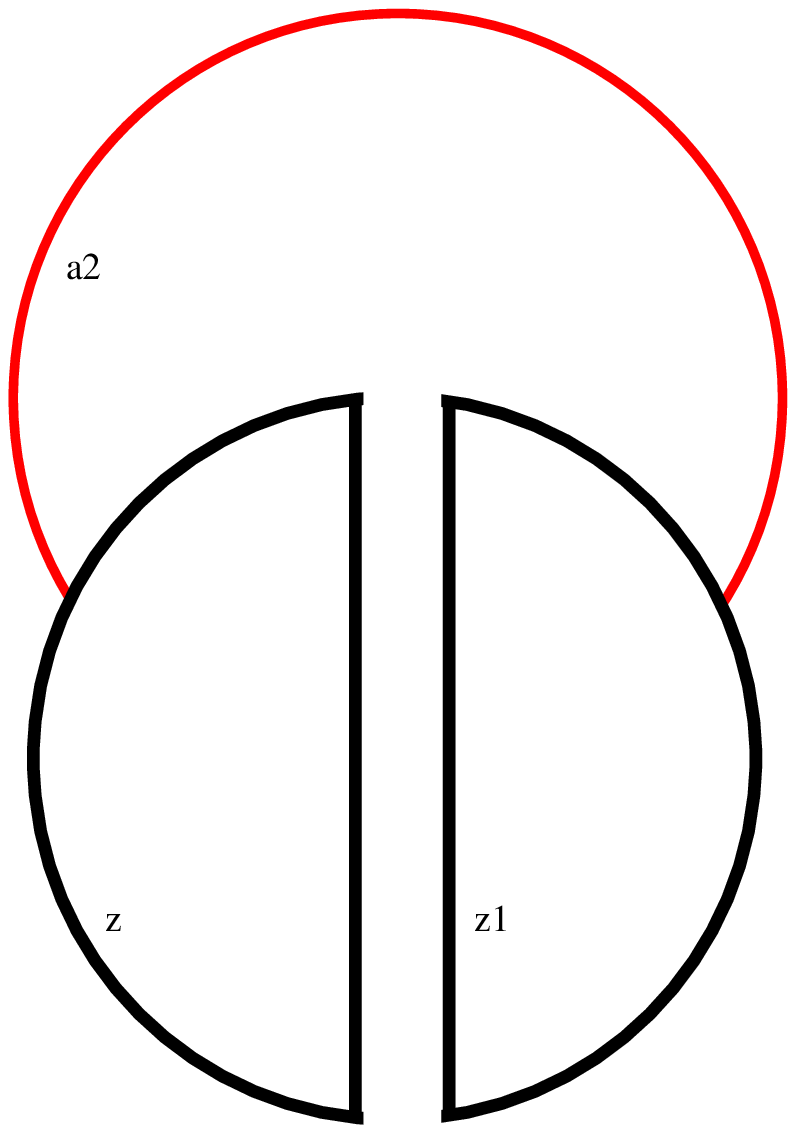}
          }}
        \ar@{<->}[d]&
        \vcenter{\hbox{
            \psfrag{a2}{$A_2$} \psfrag{z}{$x_-$}\psfrag{z1}{$x_+$}
            \includegraphics[width=0.15\textwidth]{LadyBug2}
          }}
        \ar@{<->}[d]\\
        \vcenter{\hbox{
            \psfrag{a1}{$A_1$} \psfrag{z}{$x_+$}\psfrag{z1}{$x_-$}
            \includegraphics[width=0.15\textwidth]{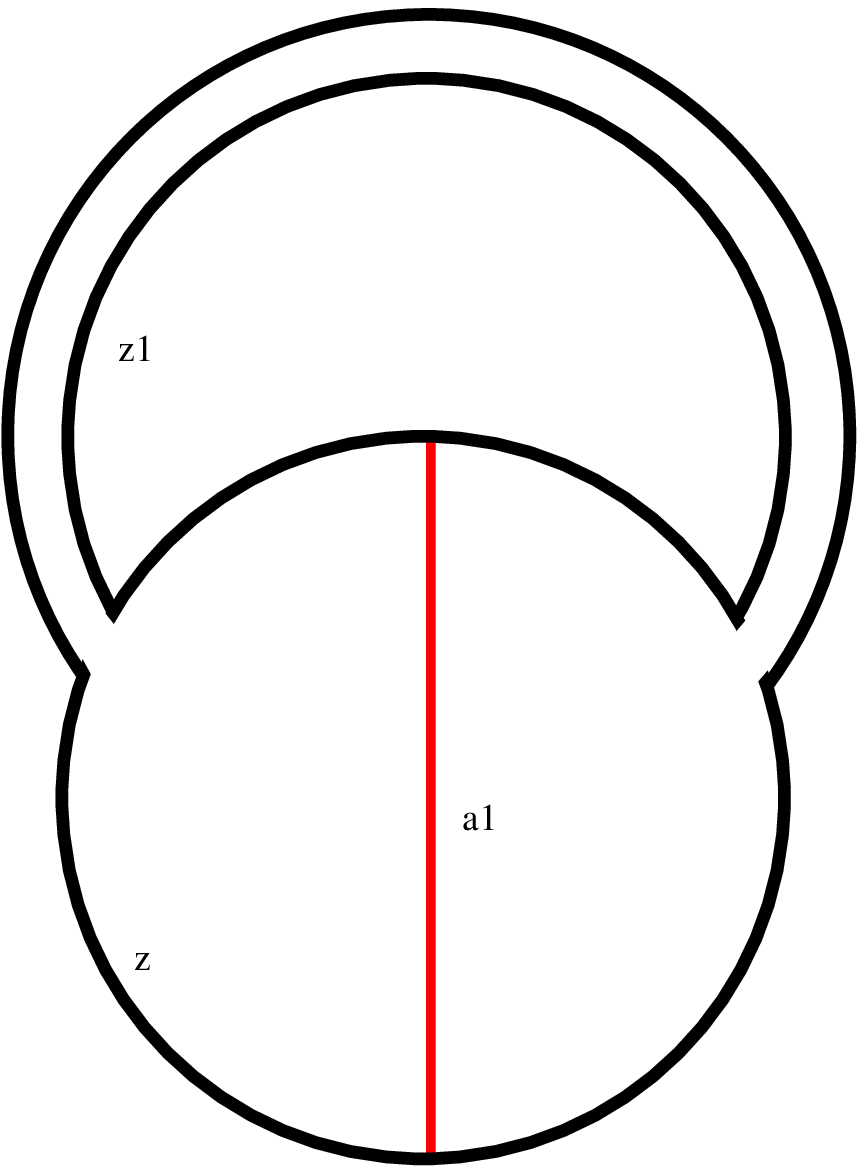}
          }}
        &
        \vcenter{\hbox{
            \psfrag{a1}{$A_1$} \psfrag{z}{$x_-$}\psfrag{z1}{$x_+$}
            \includegraphics[width=0.15\textwidth]{LadyBug3}
          }}
      }
    }
  }
  \caption{\textbf{The ladybug matching.} We have shown the case when
    $\InBetween{\ob{x}}{\ob{y}}$ has $4$ elements.}
\end{figure}

\begin{lem}\label{lem:ladybug-changes-sign}
  Let $\ob{x},\ob{y},\ob{z}$ be Khovanov generators with
  $\ob{z}\in\InBetween{\ob{x}}{\ob{y}}$. Let
  $\ob{z}'=\MatchLady[\ob{x},\ob{y}](\ob{z})$. Then
  \[
  s(\Cube_{\Funky(\ob{x}),\Funky(\ob{z})}) +
  s(\Cube_{\Funky(\ob{z}),\Funky(\ob{y})}) +
  s(\Cube_{\Funky(\ob{x}),\Funky(\ob{z}')}) +
  s(\Cube_{\Funky(\ob{z}'),\Funky(\ob{y})}) =1.
  \]
\end{lem}

\begin{proof}
  Let $u=\Funky(\ob{x})$, $v=\Funky(\ob{y})$, $w=\Funky(\ob{z})$ and
  $w'=\Funky(\ob{z}')$. We have $v\imprec w,w'\imprec u$. 

  It follows from the definition of ladybug matching
  (\Figure{ladybug-matching})
  that $w\neq w'$. Therefore, $u,v,w,w'$ are precisely the four
  vertices that appear in the $2$-cell $\Cube_{u,v}$. Since $\diff
  s=1_2$,
\[
\diff s(\Cube_{u,v}) =
s(\Cube_{u,w})+s(\Cube_{w,v})+s(\Cube_{u,w'}) + s(\Cube_{w',v})=1.\qedhere
\]
\end{proof}

\subsection{The operation \texorpdfstring{$\Sq^1$}{Sq1}}\label{subsec:sq1-describe}

Let $\ob{c}\in\KhCx^{i,j}$ be a cycle in the Khovanov chain complex.
For $\ob{x}\in\KhGen^{i+1,j}$, let
$\PreBdy[\ob{c}]{\ob{x}}=\set{\ob{y}\in\KhGen^{i,j}}{\ob{x}\in\diff\ob{y},
  \ob{y}\in\ob{c}}$. 
\begin{defn}\label{def:boundary-matching}
  A \emph{boundary matching $\MatchBdy$ for $\ob{c}$} is a collection
  of pairs $(\BijecMatch[\ob{x}],\SignMatch[\ob{x}])$, one for each
  $\ob{x}\in\KhGen^{i+1,j}$, where:
  \begin{itemize}
  \item $\BijecMatch[\ob{x}]$ is a fixed point free involution of
    $\PreBdy[\ob{c}]{\ob{x}}$, and
  \item $\SignMatch[\ob{x}]$ is a map
    $\PreBdy[\ob{c}]{\ob{x}}\to\F_2$, such that for all
    $\ob{y}\in\PreBdy[\ob{c}]{\ob{x}}$,
    \[
    \{\SignMatch[\ob{x}](\ob{y}),\SignMatch[\ob{x}](\BijecMatch[\ob{x}](\ob{y}))\}=
    \begin{cases}
      \{0,1\}&\text{if
      }s(\Cube_{\Funky(\ob{x}),\Funky(\ob{y})})=s(\Cube_{\Funky(\ob{x}),\Funky(\BijecMatch[\ob{x}](\ob{y}))})\\
      \{0\}&\text{otherwise.}
    \end{cases}
    \]
  \end{itemize}
\end{defn}
Since $\ob{c}$ is a cycle, for any $\ob{x}$ there are an even number
of elements in $\PreBdy[\ob{c}]{\ob{x}}$. Hence, there exists a
boundary matching $\MatchBdy$ for $\ob{c}$.

\begin{defn}\label{def:sq1}
Let $\ob{c}\in\KhCx^{i,j}$ be a cycle. For any boundary matching
$\MatchBdy=\{(\BijecMatch[\ob{x}],\SignMatch[\ob{x}])\}$ for $\ob{c}$,
define the chain $\sq_{\MatchBdy}^1(\ob{c})\in\KhCx^{i+1,j}$ as
\begin{equation}\label{eq:sq1}
\sq_{\MatchBdy}^1(\ob{c})=\sum_{\ob{x}\in\KhGen^{i+1,j}}\biggl(\sum_{\ob{y}\in\PreBdy[\ob{c}]{\ob{x}}}\SignMatch[\ob{x}](\ob{y})\biggr)\ob{x}.
\end{equation}
\end{defn}

\begin{prop}\label{prop:sq1-agrees}
For any cycle $\ob{c}\in\KhCx^{i,j}$ and any boundary matching $\MatchBdy$ for
$\ob{c}$, $\sq_{\MatchBdy}^1(\ob{c})$ is a cycle. Furthermore,
\[
\homology{\sq_{\MatchBdy}^1(\ob{c})}=\Sq^1(\homology{\ob{c}}).
\]
\end{prop}

\begin{proof}
The first Steenrod square $\Sq^1$ is the Bockstein associated to the
short exact sequence
\[
0\to\Z/2\to\Z/4\to\Z/2\to 0.
\]
Since the differential in $\KhCx_{\Z}$ is given by
\Equation{integer-kh-diff}, a chain representative for
$\Sq^1(\homology{\ob{c}})$ is the following:
\[
\ob{b}=\sum_{\ob{x}\in\KhGen^{i+1,j}}\biggl(\frac{\#\set{\ob{y}\in\PreBdy[\ob{c}]{\ob{x}}}{s(\Cube_{\Funky(\ob{x}),\Funky(\ob{y})})=0}-
\#\set{\ob{y}\in\PreBdy[\ob{c}]{\ob{x}}}{s(\Cube_{\Funky(\ob{x}),\Funky(\ob{y})})=1}}{2}\biggr)\ob{x}.
\]
It is easy to see that
\[
\sum_{\ob{y}\in\PreBdy[\ob{c}]{\ob{x}}}\SignMatch[\ob{x}](\ob{y})=
\biggl(\frac{\#\set{\ob{y}}{s(\Cube_{\Funky(\ob{x}),\Funky(\ob{y})})=0}-
  \#\set{\ob{y}}{s(\Cube_{\Funky(\ob{x}),\Funky(\ob{y})})=1}}{2}\biggr)
\pmod 2,
\]
and hence $\ob{b}=\sq_{\MatchBdy}^1(\ob{c})$.
\end{proof}

\subsection{The operation \texorpdfstring{$\Sq^2$}{Sq2}}

Let $\ob{c}\in\KhCx^{i,j}$ be a cycle. Choose a
boundary matching $\MatchBdy=\{(\BijecMatch[\ob{z}],\SignMatch[\ob{z}])\}$ for
$\ob{c}$. For $\ob{x}\in\KhGen^{i+2,j}$, define
\[
\PreBdy[\ob{c}]{\ob{x}}=\set{(\ob{z},\ob{y})\in\KhGen^{i+1,j}\times\KhGen^{i,j}}{\ob{x}\in\diff\ob{z},
  \ob{z}\in\diff\ob{y},\ob{y}\in\ob{c}}.
\]

Consider the edge-labeled graph $\Graph[\ob{c}]{\ob{x}}$, whose
vertices are the elements of $\PreBdy[\ob{c}]{\ob{x}}$ and whose edges are
the following.
\begin{enumerate}[label=(e-\arabic*),ref=e-\arabic*]
\item \label{item:edge-1}There is an unoriented edge between $(\ob{z},\ob{y})$ and
  $(\ob{z}',\ob{y})$, if the ladybug matching
  $\MatchLady[\ob{x},\ob{y}]$ matches $\ob{z}$ and $\ob{z}'$. This
  edge is labeled by $f(\Cube_{\Funky(\ob{x}),\Funky(\ob{y})})
  \in\F_2$, where $f$ denotes the standard frame assignment (\Subsection{cube}).
\item \label{item:edge-2}There is an edge between $(\ob{z},\ob{y})$ and
  $(\ob{z},\ob{y}')$ if the matching $\BijecMatch[\ob{z}]$ matches
  $\ob{y}$ with $\ob{y}'$. This edge is labeled by $0$. Furthermore,
  if $\SignMatch[\ob{z}](\ob{y})=0$ and
  $\SignMatch[\ob{z}](\ob{y}')=1$, then this edge is oriented from
  $(\ob{z},\ob{y})$ to $(\ob{z},\ob{y}')$; if
  $\SignMatch[\ob{z}](\ob{y})=1$ and $\SignMatch[\ob{z}](\ob{y}')=0$,
  then this edge is oriented from $(\ob{z},\ob{y}')$ to
  $(\ob{z},\ob{y})$; and if
  $\SignMatch[\ob{z}](\ob{y})=\SignMatch[\ob{z}](\ob{y}')$, then the
  edge is unoriented.
\end{enumerate}
\begin{definition}\label{def:graph-f}
  Let $\framenumber{\Graph[\ob{c}]{\ob{x}}}\in\F_2$ be the sum of all
  the edge-labels (of the Type \Item{edge-1} edges) in the graph
  $\Graph[\ob{c}]{\ob{x}}$.
\end{definition}

\begin{lem}\label{lem:even-cyc}
  Each component of $\Graph[\ob{c}]{\ob{x}}$ is an even
  cycle. Furthermore, in each component, the number of oriented edges
  is even.
\end{lem}

\begin{proof}
  Each vertex $(\ob{z},\ob{y})$ of $\Graph[\ob{c}]{\ob{x}}$ belongs to
  exactly two edges: the Type \Item{edge-1} edge joining
  $(\ob{z},\ob{y})$ and $(\MatchLady[\ob{x},\ob{y}](\ob{z}),\ob{y})$;
  and the Type \Item{edge-2} edge joining $(\ob{z},\ob{y})$ and
  $(\ob{z},\BijecMatch[\ob{z}](\ob{y}))$. This implies that each
  component of $\Graph[\ob{c}]{\ob{x}}$ is an even cycle.

  In order to prove the second part, vertex-label the graph as
  follows: To a vertex $(\ob{z},\ob{y})$, assign the number
  $s(\Cube_{\Funky(\ob{x}),\Funky(\ob{z})})+
  s(\Cube_{\Funky(\ob{z}),\Funky(\ob{y})})\in\F_2$.
  \Lemma{ladybug-changes-sign} implies that the Type \Item{edge-1}
  edges join vertices carrying opposite labels; and among the
  Type \Item{edge-2} edges, it is clear from the definition of
  boundary matching (\Subsection{sq1-describe}) that the oriented
  edges join vertices carrying the same label, and the unoriented
  edges join vertices carrying opposite labels. Therefore, each cycle
  must contain an even number of unoriented edges; since there are an
  even number of vertices in each cycle, we are done.
\end{proof}

This observation allows us to associate the following number
$\graphnumber{\Graph[\ob{c}]{\ob{x}}}\in\F_2$ to the graph.
\begin{definition}\label{def:graph-g}Partition
  the oriented edges of $\Graph[\ob{c}]{\ob{x}}$ into two sets, such
  that if two edges from the same cycle are in the same set, they are
  oriented in the same direction. Let
  $\graphnumber{\Graph[\ob{c}]{\ob{x}}}$ be the number modulo $2$ of
  the elements in either set.
\end{definition}

\begin{defn}\label{def:sq2}
  Let $\ob{c}\in\KhCx^{i,j}$ be a cycle. For any boundary matching $\MatchBdy$
  for $\ob{c}$, define the chain
  $\sq_{\MatchBdy}^2(\ob{c})\in\KhCx^{i+1,j}$ as 
  \begin{equation}\label{eq:sq2}
  \sq^2_{\MatchBdy}(\ob{c})=\sum_{\ob{x}\in\KhGen^{i+2,j}}
  \biggl(\#|\Graph[\ob{c}]{\ob{x}}| +
  \framenumber{\Graph[\ob{c}]{\ob{x}}}+\graphnumber{\Graph[\ob{c}]{\ob{x}}}\biggr)\ob{x}.
  \end{equation}
(Here, $\#|\Graph[\ob{c}]{\ob{x}}|$ is the number of components of the
graph, $\framenumber{\Graph[\ob{c}]{\ob{x}}}$ is defined in
\Definition{graph-f}, and $\graphnumber{\Graph[\ob{c}]{\ob{x}}}$ is
defined in \Definition{graph-g}.)
\end{defn}

We devote \Section{where-answer} to proving the following.
\begin{thm}\label{thm:sq2-agrees}
For any cycle $\ob{c}\in\KhCx^{i,j}$ and any boundary matching $\MatchBdy$ for
$\ob{c}$, $\sq_{\MatchBdy}^2(\ob{c})$ is a cycle. Furthermore,
\[
\homology{\sq_{\MatchBdy}^2(\ob{c})}=\Sq^2(\homology{\ob{c}}).
\]
\end{thm}

\begin{cor}\label{cor:commute-isom}
  The operations $\sq_{\MatchBdy}^1$ and $\sq_{\MatchBdy}^2$ induce
  well-defined maps
  \[
  \Sq^1\co \Kh^{i,j}\to \Kh^{i+1,j}\qquad\text{and}\qquad
  \Sq^2\co \Kh^{i,j}\to \Kh^{i+2,j}
  \]
  that are independent of the choices of boundary matchings
  $\MatchBdy$. Furthermore, these maps are link invariants, in the
  following sense: given any two diagrams $L$ and $L'$ representing
  the same link, there are isomorphisms $\phi_{i,j}\co
  \Kh^{i,j}(L)\to\Kh^{i,j}(L')$ making the following diagrams commute:
  \[
  \xymatrix{ \Kh^{i+1,j}(L)\ar[r]^{\phi_{i+1,j}} & \Kh^{i+1,j}(L') &
    \qquad &
    \Kh^{i+2,j}(L)\ar[r]^{\phi_{i+2,j}} & \Kh^{i+2,j}(L')\\
    \Kh^{i,j}(L)\ar[u]_{\Sq^1}\ar[r]_{\phi_{i,j}} &
    \Kh^{i,j}(L')\ar[u]_{\Sq^1} &\qquad &
    \Kh^{i,j}(L)\ar[u]_{\Sq^2}\ar[r]_{\phi_{i,j}} &
    \Kh^{i,j}(L')\ar[u]_{\Sq^2}.  }
  \]
\end{cor}
\begin{proof}
  This is immediate from \Proposition{sq1-agrees},
  \Theorem{sq2-agrees}, and invariance of the Khovanov
  spectrum~\cite[Theorem~\ref*{KhSp:thm:kh-space}]{RS-khovanov}. 
\end{proof}

Indeed, we show in \cite[Theorem~4]{RS-rasmussen} that we can choose
the isomorphisms in \Corollary{commute-isom} to be the canonical ones
induced from an isotopy from $L$ to $L'$.

\subsection{An example}\label{subsec:summary}
For the reader's convenience, we present an artificial example to illustrate \Definition{sq1}
and \Definition{sq2}.
\begin{example}
  Assume $\KhGen^{i,j}=\{\ob{y}_1,\dots,\ob{y}_5\}$,
  $\KhGen^{i+1,j}=\{\ob{z}_1,\dots,\ob{z}_6\}$ and
  $\KhGen^{i+2,j}=\{\ob{x}_1,\ob{x}_2\}$, and the Khovanov
  differential $\diff$ has the following form.
  \[
  \xymatrix{
    &\ob{x}_1&&\ob{x}_2&&\\
    \ob{z}_1\ar[ur]\ar[urrr]&\ob{z}_2\ar[u]&\ob{z}_3\ar[ul]&\ob{z}_4
    \ar[ull]\ar[u] &\ob{z}_5\ar[ul]&\ob{z}_6\ar[ull]\\
    &\ob{y}_1\ar[ul]\ar[u]\ar[ur]\ar[urr]&\ob{y}_2\ar[ull]\ar[u]\ar[urr]
    &\ob{y}_3\ar[ull]\ar[u]\ar[ur]
    &\ob{y}_4\ar[u]\ar[ur]&\ob{y}_5\ar[u]\ar[ul]\\
  }
  \]
  Assume that the sign assignment and the frame assignment are as
  follows.
  \begin{center}
    \begin{tabular}{cc|cc|cc|cc}
      \toprule
      $(\ob{x},\ob{y})$&$s(\Cube_{\Funky(\ob{x}),\Funky(\ob{y})})$&
      $(\ob{x},\ob{y})$&$s(\Cube_{\Funky(\ob{x}),\Funky(\ob{y})})$&
      $(\ob{x},\ob{y})$&$s(\Cube_{\Funky(\ob{x}),\Funky(\ob{y})})$&
      $(\ob{x},\ob{y})$&$f(\Cube_{\Funky(\ob{x}),\Funky(\ob{y})})$\\
      \midrule
      $(\ob{z}_1,\ob{y}_1)$&1&$(\ob{z}_4,\ob{y}_3)$ &1&$(\ob{x}_1,\ob{z}_3)$ &0&$(\ob{x}_1,\ob{y}_1)$ &1\\
      $(\ob{z}_2,\ob{y}_1)$&1&$(\ob{z}_5,\ob{y}_3)$ &0&$(\ob{x}_2,\ob{z}_3)$ &0&$(\ob{x}_2,\ob{y}_1)$ &0\\
      $(\ob{z}_3,\ob{y}_1)$&0&$(\ob{z}_5,\ob{y}_4)$ &0&$(\ob{x}_1,\ob{z}_4)$ &0&$(\ob{x}_1,\ob{y}_2)$ &0\\
      $(\ob{z}_4,\ob{y}_1)$&0&$(\ob{z}_6,\ob{y}_4)$ &0&$(\ob{x}_2,\ob{z}_4)$ &0&$(\ob{x}_2,\ob{y}_2)$ &1\\
      $(\ob{z}_1,\ob{y}_2)$&1&$(\ob{z}_5,\ob{y}_5)$ &0&$(\ob{x}_2,\ob{z}_5)$ &0&$(\ob{x}_1,\ob{y}_3)$ &1\\
      $(\ob{z}_3,\ob{y}_2)$&0&$(\ob{z}_6,\ob{y}_5)$ &0&$(\ob{x}_2,\ob{z}_6)$ &1&$(\ob{x}_2,\ob{y}_3)$ &0\\
      $(\ob{z}_5,\ob{y}_2)$&0&$(\ob{x}_1,\ob{z}_1)$ &0& & &$(\ob{x}_2,\ob{y}_4)$ &1\\
      $(\ob{z}_2,\ob{y}_3)$&0&$(\ob{x}_1,\ob{z}_2)$ &0& & &$(\ob{x}_2,\ob{y}_5)$ &1\\
      \bottomrule
    \end{tabular}
  \end{center}
  Finally, assume that the ladybug matching
  $\MatchLady[\ob{x}_1,\ob{y}_1]$ matches $\ob{z}_1$ with $\ob{z}_4$
  and $\ob{z}_2$ with $\ob{z}_3$.

  Let us start with the cycle $\ob{c}\in\KhCx^{i,j}$ given by
  $\ob{c}=\sum_{i=1}^5 \ob{y}_i$. In order to compute $\Sq^1(\ob{c})$
  and $\Sq^2(\ob{c})$, we need to choose a boundary matching
  $\MatchBdy=\{(\BijecMatch[\ob{z}_j],\SignMatch[\ob{z}_j])\}$ for
  $\ob{c}$. Let us choose the following boundary matching.
  \begin{center}
    \begin{tabular}{ccc|ccc}
      \toprule
      $j$ & $\BijecMatch[\ob{z}_j]$ & $\SignMatch[\ob{z}_j]$ & $j$ &
      $\BijecMatch[\ob{z}_j]$ & $\SignMatch[\ob{z}_j]$ \\
      \midrule
      $1$ & $\ob{y}_1\matching \ob{y}_2$ & $\ob{y}_1\to 0,\ob{y}_2\to 1$ &
      $4$ &$\ob{y}_1\matching \ob{y}_3$ & $\ob{y}_1,\ob{y}_3\to 0$\\ 
      $2$ & $\ob{y}_1\matching \ob{y}_3$ & $\ob{y}_1,\ob{y}_3\to 0$ & $5$ &$\ob{y}_2\matching
      \ob{y}_3, \ob{y}_4\matching \ob{y}_5$ &
      $\ob{y}_2,\ob{y}_4\to 0,\ob{y}_3,\ob{y}_5\to 1$\\ 
      $3$ & $\ob{y}_1\matching \ob{y}_2$ & $\ob{y}_1\to 0,\ob{y}_2\to 1$ &
      $6$ &$\ob{y}_4\matching \ob{y}_5$ & $\ob{y}_4\to 0,\ob{y}_5\to 1$\\       
      \bottomrule
    \end{tabular}
  \end{center}

  Then, the cycle $\sq^1_{\MatchBdy}(\ob{c})$ is given by 
  \begin{align*}
    \sq^1_{\MatchBdy}(\ob{c})&=\sum_{j=1}^6
    \biggl(\sum_{\ob{y}\in\PreBdy[\ob{c}]{\ob{z}_j}}\SignMatch[\ob{z}_j](\ob{y})\biggr)
    \ob{z}_j\\
    &=\bigl(\SignMatch[\ob{z}_1](\ob{y}_1)+\SignMatch[\ob{z}_1](\ob{y}_2)\bigr)
    \ob{z}_1+ \bigl(\SignMatch[\ob{z}_2](\ob{y}_1)+\SignMatch[\ob{z}_2](\ob{y}_3)\bigr)
    \ob{z}_2+ \bigl(\SignMatch[\ob{z}_3](\ob{y}_1)+\SignMatch[\ob{z}_3](\ob{y}_2)\bigr)
    \ob{z}_3\\
    &\qquad{}+ \bigl(\SignMatch[\ob{z}_4](\ob{y}_1)+\SignMatch[\ob{z}_4](\ob{y}_3)\bigr)
    \ob{z}_4+
    \bigl(\SignMatch[\ob{z}_5](\ob{y}_2)+\SignMatch[\ob{z}_5](\ob{y}_3)+
    \SignMatch[\ob{z}_5](\ob{y}_4)+\SignMatch[\ob{z}_5](\ob{y}_5)\bigr)
    \ob{z}_5\\
    &\qquad{}+ \bigl(\SignMatch[\ob{z}_6](\ob{y}_4)+\SignMatch[\ob{z}_6](\ob{y}_5)\bigr)
    \ob{z}_6\\
    &=\ob{z}_1+\ob{z}_3+\ob{z}_6.
  \end{align*}

  In order to compute $\sq^2_{\MatchBdy}(\ob{c})$, we need to study
  the graphs $\Graph[\ob{c}]{\ob{x}_1}$ and
  $\Graph[\ob{c}]{\ob{x}_2}$, which are the following:
  \[
  \xymatrix@C=0pt{
    &(\ob{z}_1,\ob{y}_1)\ar@{->}[rr]\ar@{.}[dl]^-1&&(\ob{z}_1,\ob{y}_2)&&
    &&(\ob{z}_1,\ob{y}_1)\ar@{->}[dr]\ar@{.}[dl]^-0&  &&\\
    (\ob{z}_4,\ob{y}_1)&&&&(\ob{z}_3,\ob{y}_2) \ar@{<-}[d]\ar@{.}[ul]^-0 &          &(\ob{z}_4,\ob{y}_1)&&(\ob{z}_1,\ob{y}_2) &(\ob{z}_6,\ob{y}_5)\ar@{<-}[rr]\ar@{.}[d]^-1&&(\ob{z}_6,\ob{y}_4)\\
    (\ob{z}_4,\ob{y}_3)\ar@{-}[u]\ar@{.}[dr]^-1&&&&(\ob{z}_3,\ob{y}_1)&\text{and}&(\ob{z}_4,\ob{y}_3)\ar@{-}[u]\ar@{.}[dr]^-0
    &&(\ob{z}_5,\ob{y}_2)\ar@{->}[dl]\ar@{.}[u]^-1 &(\ob{z}_5,\ob{y}_5)&&(\ob{z}_5,\ob{y}_4).\ar@{->}[ll]\ar@{.}[u]^-1\\
    &(\ob{z}_2,\ob{y}_3)&&(\ob{z}_2,\ob{y}_1)\ar@{-}[ll]\ar@{.}[ur]^-1&&
    &&(\ob{z}_5,\ob{y}_3)& && }
  \]
  The Type \Item{edge-1} edges are represented by the dotted lines; they are
  unoriented and are labeled by elements of $\F_2$. The
  Type \Item{edge-2} edges are represented by the solid lines; they
  are labeled by $0$ and are sometimes oriented. Therefore, the cycle
  $\sq^2_{\MatchBdy}(\ob{c})$ is given by
  \begin{align*}
    \sq^2_{\MatchBdy}(\ob{c})&=\sum_{j=1}^2
    \biggl(\#|\Graph[\ob{c}]{\ob{x}_j}| +
    \framenumber{\Graph[\ob{c}]{\ob{x}_j}}+
    \graphnumber{\Graph[\ob{c}]{\ob{x}_j}}\biggr)\ob{x}_j\\
    &=(1+1+1)\ob{x}_1+(0+1+1)\ob{x}_2\\
    &=\ob{x}_1.
  \end{align*}
\end{example}

\section{Where the answer comes from}\label{sec:where-answer}
This section is devoted to proving \Theorem{sq2-agrees}. The operation
$\Sq^2$ on a CW complex $Y$ is determined by the sub-quotients
$Y^{(m+2)}/Y^{(m-1)}$. (In \Subsection{sq2-cw-complex} we review an
explicit description in these terms, due to Steenrod.) The space
$\KhSpace(L)$ is a formal de-suspension of a CW complex $Y=Y(L)$. So,
most of the work is in understanding combinatorially how the $m$-,
$(m+1)$- and $(m+2)$-cells of $Y(L)$ are glued together.

The description of $Y(L)$ from~\cite{RS-khovanov} is in terms of a
Pontrjagin-Thom type construction. To understand just
$Y^{(m+2)}/Y^{(m-1)}$ involves studying certain framed points in
$\RR^m$ and framed paths in $\R\times\RR^{m}$. We will draw these framings
from a particular set of choices, described
in \Subsection{frames-R3}. \Subsection{frame-cube-flow-cat} explains
exactly how we assign framings from this set, and shows that these
framings are consistent with the construction
in~\cite{RS-khovanov}. Finally, \Subsection{sq2-on-khspace} discusses
how to go from these choices to $Y^{(m+2)}/Y^{(m-1)}$, and why the
resulting operation $\Sq^2$ agrees with the operation from
\Definition{sq2}.

\subsection{\texorpdfstring{$\Sq^2$}{Sq2} for a CW
  complex}\label{subsec:sq2-cw-complex}
We start by recalling a definition of $\Sq^2$. The discussion in this
section is heavily inspired by~\cite[Section 12]{Ste-top-operations}.

Let $K_m=K(\ZZ/2,m)$ denote the $m\th$ Eilenberg-MacLane space for the
group $\ZZ/2$, so $\pi_m(K_m)=\ZZ/2$ and $\pi_i(K_m)=0$ for $i\neq
m$. Assume that $m$ is sufficiently large, say $m\geq 3$. We start by
discussing a CW
structure for $K_m$. Since $\pi_i(K_m)=0$ for $i<m$, we can choose the
$m$-skeleton $K_m^{(m)}$ to be a single $m$-cell $e^m$ with the entire
boundary $\del e^m$ attached to the basepoint. To arrange that
$\pi_m(K_m)=\ZZ/2$ it suffices to attach a single $(m+1)$-cell via a
degree $2$ map $\bdy e^{m+1}\to K^{(m)}_m=S^m$.

We show that the resulting
$(m+1)$-skeleton $K_m^{(m+1)}$ has $\pi_{m+1}(K_m^{(m+1)})\cong
\ZZ/2$. From the long exact sequence for the pair
$(K_m^{m+1},S^m)$,
\[
\pi_{m+2}(K_m^{(m+1)},S^m)\to \pi_{m+1}(S^m)\to \pi_{m+1}(K_m^{(m+1)})\to
\pi_{m+1}(K_m^{(m+1)},S^m)\to \pi_m(S^m). 
\]
By excision (since $m$ is large),
$\pi_{m+1}(K_m^{(m+1)},S^m)\cong\pi_{m+1}(K_m^{(m+1)}/S^m)=\pi_{m+1}(S^{m+1})=\ZZ$
and $\pi_{m+2}(K_m^{(m+1)},S^m)\cong \pi_{m+2}(S^{m+1})=\ZZ/2$.  The
maps $\pi_{i+1}(K_m^{(m+1)},S^m)=\pi_{i+1}(S^{m+1})\to \pi_i(S^m)$ are
twice the Freudenthal isomorphisms. So, this sequence becomes
\[
\ZZ/2\stackrel{2}{\to} \ZZ/2\to \pi_{m+1}(K_m^{(m+1)})\to
\ZZ\stackrel{2}{\to}\ZZ.
\]
Thus, $\pi_{m+1}(K_m^{(m+1)})\cong\ZZ/2$, represented by the Hopf map
$S^{m+1}\to S^m=K_m^{(m)}\into K_m^{(m+1)}$. 

Let $K_m^{(m+2)}$ be the
result of attaching an $(m+2)$-cell $e^{m+2}$ to kill this
$\ZZ/2$. This attaching map has degree $0$ as a map to the
$(m+1)$-cell $e^{m+1}$ in $K_m^{(m+1)}$, so
the $(m+2)$-skeleton of $K_m$ has cohomology:
\begin{align*}
H^m(K_m^{(m+2)};\ZZ)&=0 & H^{m+1}(K_m^{(m+2)};\Z)&=\F_2
& H^{m+2}(K_m^{(m+2)};\Z)&=\ZZ\\
H^m(K_m^{(m+2)};\F_2)&=\F_2 & H^{m+1}(K_m^{(m+2)};\F_2)&=\F_2
& H^{m+2}(K_m^{(m+2)};\F_2)&=\F_2.
\end{align*}

Therefore, there are fundamental cohomology classes $\iota\in
H^m(K_m;\F_2)$ and $\Sq^2(\iota)\in
H^{m+2}(K_m^{(m+2)};\F_2)$. It turns out that the element
$\Sq^2(\iota)$ survives to $H^{m+2}(K_m;\F_2)$.

Now, consider a CW complex $Y$ and a cohomology class $c\in
H^m(Y;\F_2)$. The element $c$ is classified by a map $\cmap{c}\co Y\to
K_m$, so
that $\cmap{c}^*\iota=c$. We can arrange that the map $\cmap{c}$ is
cellular. So, we have an element $\Sq^2(c)=\cmap{c}^*\Sq^2(\iota)\in
H^{m+2}(Y;\F_2)$.

The element $\Sq^2(c)$ is determined by its restriction to
$H^{m+2}(Y^{(m+2)};\F_2)$. So, to compute $\Sq^2(c)$ it suffices to give a
cellular map $Y^{(m+2)}\to K_m^{(m+2)}$ so that $\iota$ pulls back to
$c$.  Then, $\Sq^2(c)$ is the cochain which sends an $(m+2)$-cell
$f^{m+2}$ of $Y$ to the degree of the map $f^{m+2}/\bdy f^{m+2}\to
e^{m+2}/\bdy e^{m+2}$. Equivalently, $\Sq^2(c)$ sends $f^{m+2}$ to the
element $\cmap{c}|_{\bdy f^{m+2}}\in
\pi_{m+1}(K_m^{(m+1)})=\pi_{m+1}(S^m)=\ZZ/2$. (In other words,
$\Sq^2(c)$ is the obstruction to homotoping $\cmap{c}$ so that it
sends the $(m+2)$-skeleton of $Y$ to the $(m+1)$-skeleton of $K_m$.)
Since $K_m^{(m+2)}$ has no cells of dimension between $0$ and $m$, the
map $Y^{(m+2)}\to K_m^{(m+2)}$ factors through $Y^{(m+2)}/Y^{(m-1)}$.

To understand the operation $\Sq^2$ on Khovanov homology induced by
the Khovanov homotopy type $Y$, 
it remains to explicitly give the map $\cmap{c}$ on
$Y^{(m+2)}/Y^{(m-1)}$. This will be done
in \Subsection{sq2-on-khspace}, after we develop tools to understand
the attaching maps for the $(m+2)$-cells.

\subsection{Frames in
  \texorpdfstring{$\R^3$}{R3}}\label{subsec:frames-R3}
As discussed in \Subsection{sq2-on-khspace}, the sub-quotients
$Y^{(m+2)}/Y^{(m-1)}$ of the Khovanov space $Y=Y(L)$ are defined in
terms of framed points in $\{0\}\times \RR^m$ and framed paths in
$\R\times\RR^{m}$ connecting these points.

A framing of a path $\gamma\co [0,1]\to \RR^{m+1}$ is a tuple
$[v_1(t),\dots, v_m(t)]\in (\RR^{m+1})^m$ of orthonormal vector fields
along $\gamma$, normal to $\gamma$. A collection of $m$ orthonormal
vectors $v_1,\dots, v_m$ in $\RR^{m+1}$ specifies a matrix in $\SO(m+1)$,
whose last $m$ columns are $v_1,\dots,v_m$ and whose first column is
the cross product of $v_1,\dots,v_m$.

Now, suppose that $p,q \in \{0\}\times\RR^m$ and that we are given
trivializations $\Frame_p, \Frame_q$ of the normal bundles in
$\{0\}\times\RR^m$ to $p$, $q$ (i.e., framings of $p$ and $q$). On the
one hand, we can consider the set of isotopy classes of framed paths
from $(p, \Frame_p)$ to $(q,\Frame_q)$. On the other hand, we can
consider the homotopy classes of paths in $\SO(m+1)$ from $\Frame_p$
to $\Frame_q$. There is an obvious map from isotopy classes of framed
paths in $\RR^{m+1}$ to homotopy classes of paths in $\SO(m+1)$, by
considering only the framing. This map is a surjection if $m\geq 2$
and a bijection if $m\geq 3$. In the case that $m\geq 3$, both sets have two
elements.

The upshot is that if we want to specify an isotopy classes of framed
paths with given endpoints, and $m\geq 3$, then it suffices to specify
a homotopy class of paths in $\SO(m+1)$.

The framings on both the endpoints and the paths relevant to
constructing $Y^{(m+2)}/Y^{(m-1)}$ will have a special form: they will
be stabilizations of the $m=2$ case. Specifically, we will write
$m=m_1+m_2$ with $m_i\geq 1$. Let
$[\ol{e},e_{11},\dots,e_{1m_1},e_{21},\dots,e_{2m_2}]$ denote the
standard basis for $\R\times\RR^m$. Then all of the points will have
framings of the form
\[
[v_1,e_{12},\dots, e_{1m_1}, v_2, e_{22},\dots, e_{2m_2}]\in \bigl(\{0\}\times
\RR^{m_1}\times \RR^{m_2}\bigr)^m,
\]
for some $v_1\in\{\pm e_{11}\}$, $v_2\in \{\pm e_{21}\}$. So, to describe isotopy classes of framed paths connecting these
points it suffices to describe paths of the form
\[
[v_1(t),e_{12},\dots, e_{1m_1},v_2(t), e_{22},\dots, e_{2m_2}]\in
\bigl(\RR\times \RR^{m_1}\times \RR^{m_2}\bigr)^{m}.
\]

Therefore, to describe such paths, it suffices to work in $\RR^3$ (i.e., the case
$m=2$). Denote the standard basis for $\RR^3$ by $[\ol{e},e_1,e_2]$.
We will work with the four distinguished frames in $\{0\}\times\RR^2$,
$[e_1,e_2],[-e_1,e_2],[e_1,-e_2],[-e_1,-e_2]$, which we denote by the
symbols $\topright, \topleft, \bottomright, \bottomleft$,
respectively.  By a \emph{coherent system of paths} joining
$\topright, \topleft, \bottomright, \bottomleft$ we mean a choice of a
path $\Path{\Frame_1\Frame_2}$  in
$\SO(3)$ from $\Frame _1$ to $\Frame_2$ for each pair of frames
$\Frame_1,\Frame_2\in\{\topright,
\topleft, \bottomright, \bottomleft\}$, satisfying the following
cocycle conditions: 
\begin{enumerate}
\item For all $\Frame\in\{\topright, \topleft, \bottomright,
  \bottomleft\}$, the loop $\Path{\Frame\Frame}$ is nullhomotopic; and
\item For all $\Frame_1,\Frame_2,\Frame_3\in\{\topright, \topleft,
  \bottomright, \bottomleft\}$, the path
  $\Path{\Frame_1\Frame_2}\cdot\Path{\Frame_2\Frame_3}$ is homotopic
  (\rel endpoints) to the path $\Path{\Frame_1\Frame_3}$.
\end{enumerate}
We make a particular choice of a coherent system of paths, as follows:
\begin{itemize}
\item $\ol{\topright\topleft}, \ol{\topleft\topright},
  \ol{\bottomright\bottomleft}, \ol{\bottomleft\bottomright}$: Rotate
  $180^{\circ}$ around the $e_2$-axis, such that the first vector equals $\ol{e}$
  halfway through.
\item $\ol{\topright\bottomright}, \ol{\bottomright\topright}$: Rotate
  $180^{\circ}$ around the $e_1$-axis, such that the second vector
  equals $\ol{e}$ halfway through.
\item $\ol{\topleft\bottomleft}, \ol{\bottomleft\topleft}$: Rotate
  $180^{\circ}$ around the $e_1$-axis, such that the second vector equals $-\ol{e}$
  halfway through.
\item $\ol{\topright\bottomleft}, \ol{\bottomleft\topright},
  \ol{\topleft\bottomright},\ol{\bottomright\topleft}$: Rotate
  $180^{\circ}$ around the $\ol{e}$-axis, such that the second vector
  equals $-e_1$ halfway through.
\end{itemize}

\begin{lem}
The above choice describes a coherent system of paths.
\end{lem}

\begin{proof}
  We only need to check that each of the loops
  $\ol{\topright\topleft}\cdot\ol{\topleft\bottomleft}\cdot\ol{\bottomleft\topright}$,
  $\ol{\topright\bottomright}\cdot\ol{\bottomright\topleft}\cdot\ol{\topleft\topright}$,
  and
  $\ol{\topright\topleft}\cdot\ol{\topleft\bottomleft}\cdot\ol{\bottomleft\bottomright}\cdot\ol{\bottomright\topright}$
  are null-homotopic. This is best checked with hand
  motions, as we have illustrated for the first loop in
  \Figure{handmotion}.
\end{proof}

\begin{figure}
\includegraphics[width=0.07\textwidth]{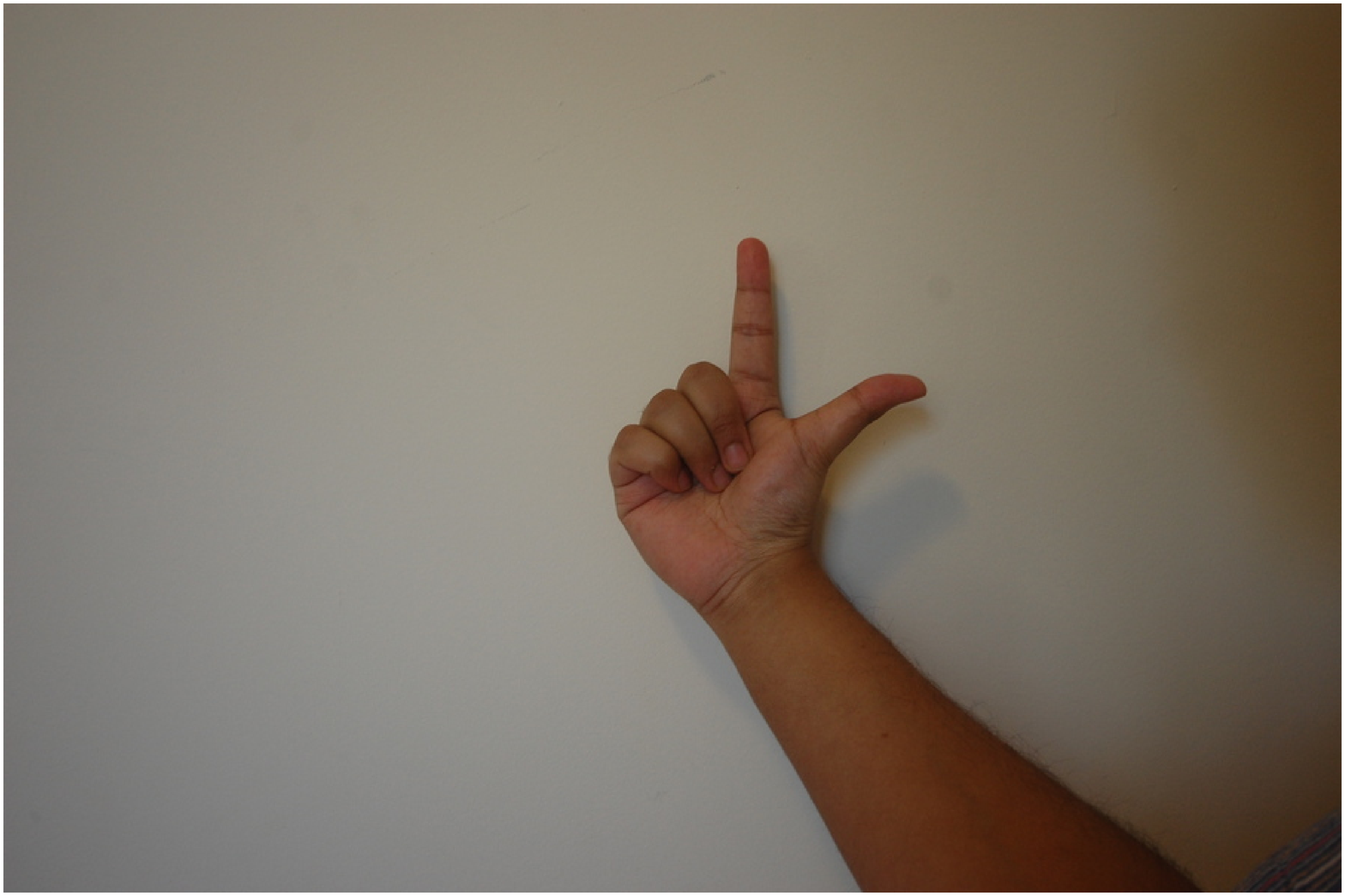}
\includegraphics[width=0.07\textwidth]{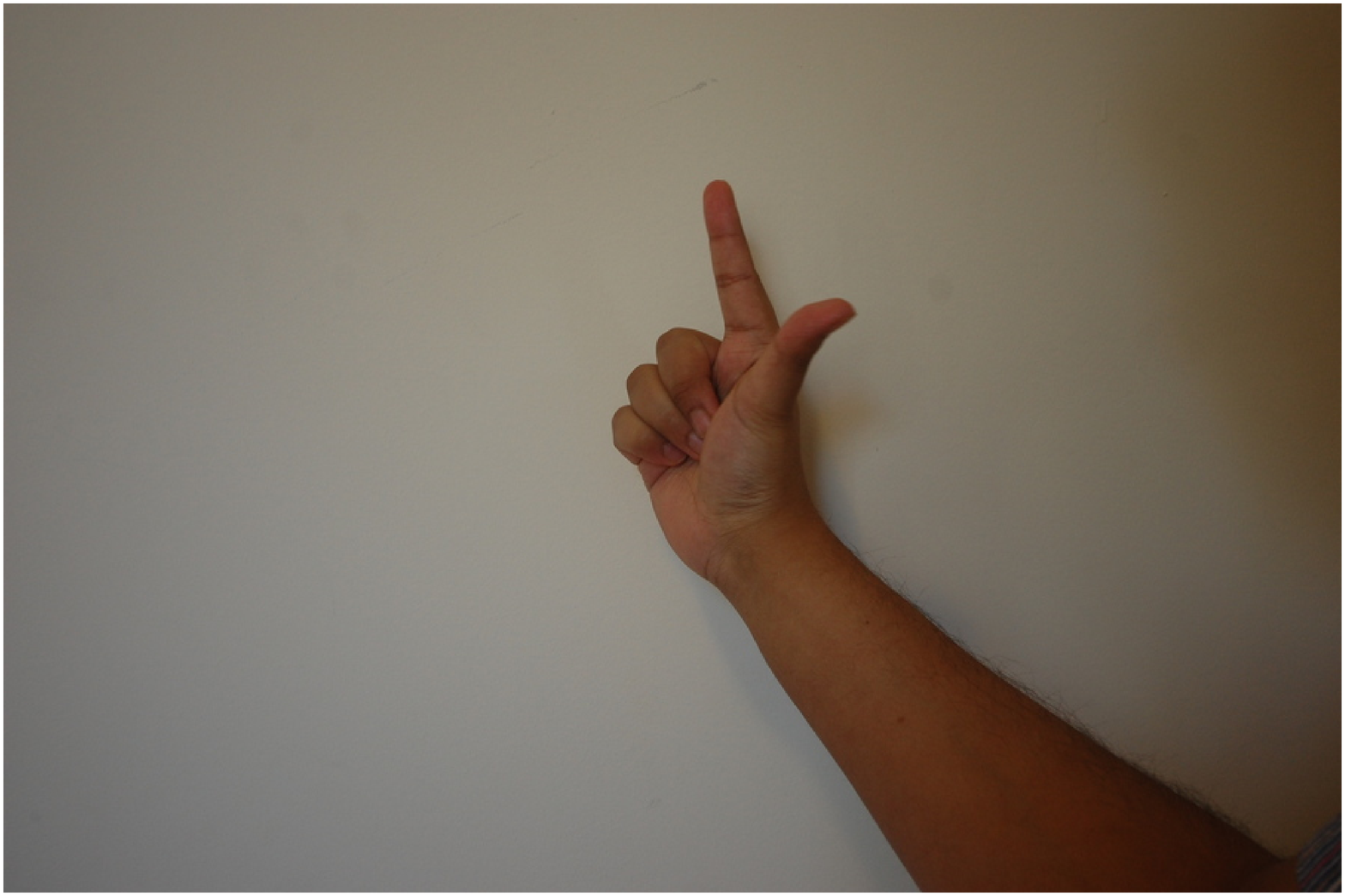}
\includegraphics[width=0.07\textwidth]{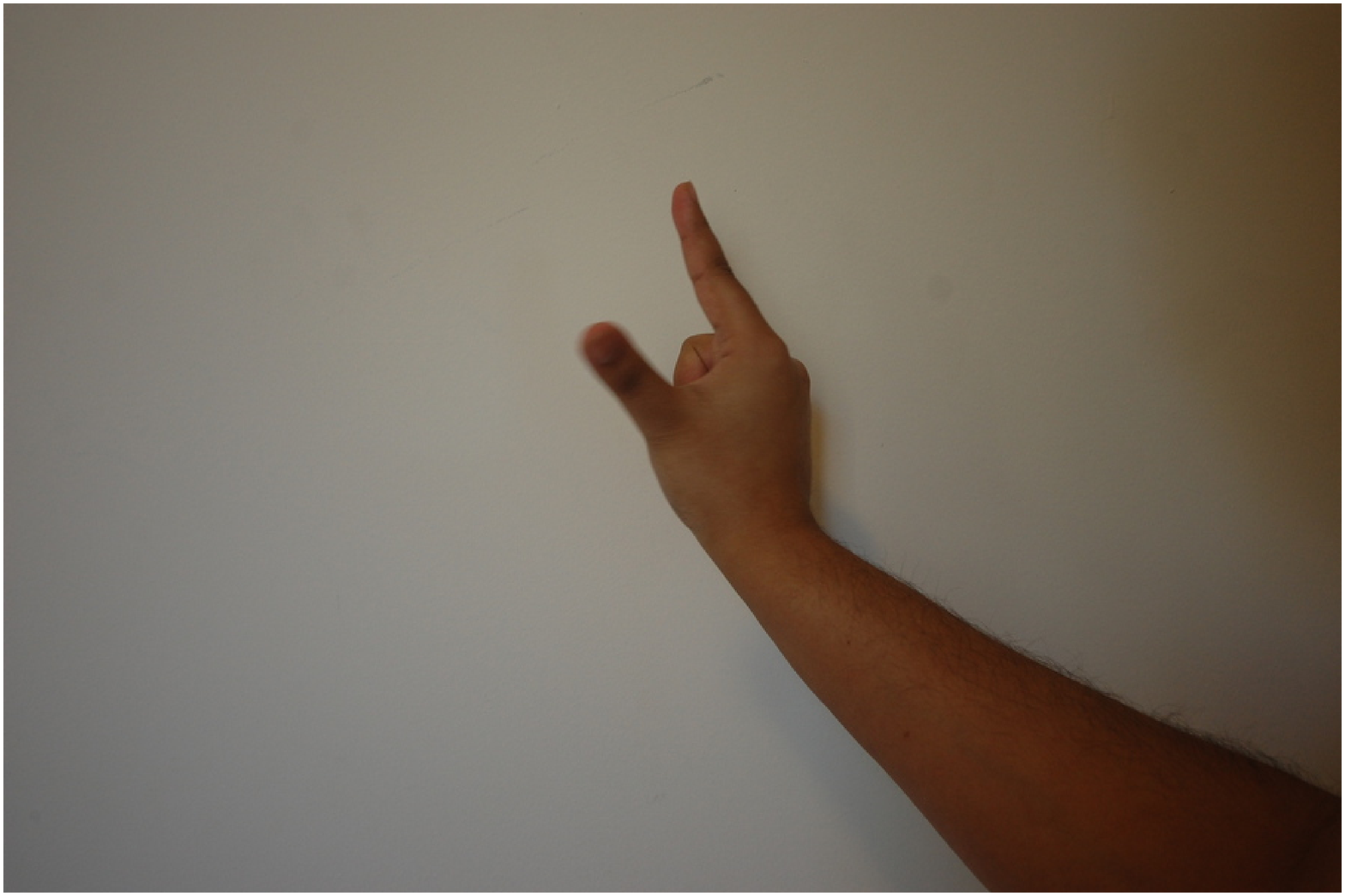}
\includegraphics[width=0.07\textwidth]{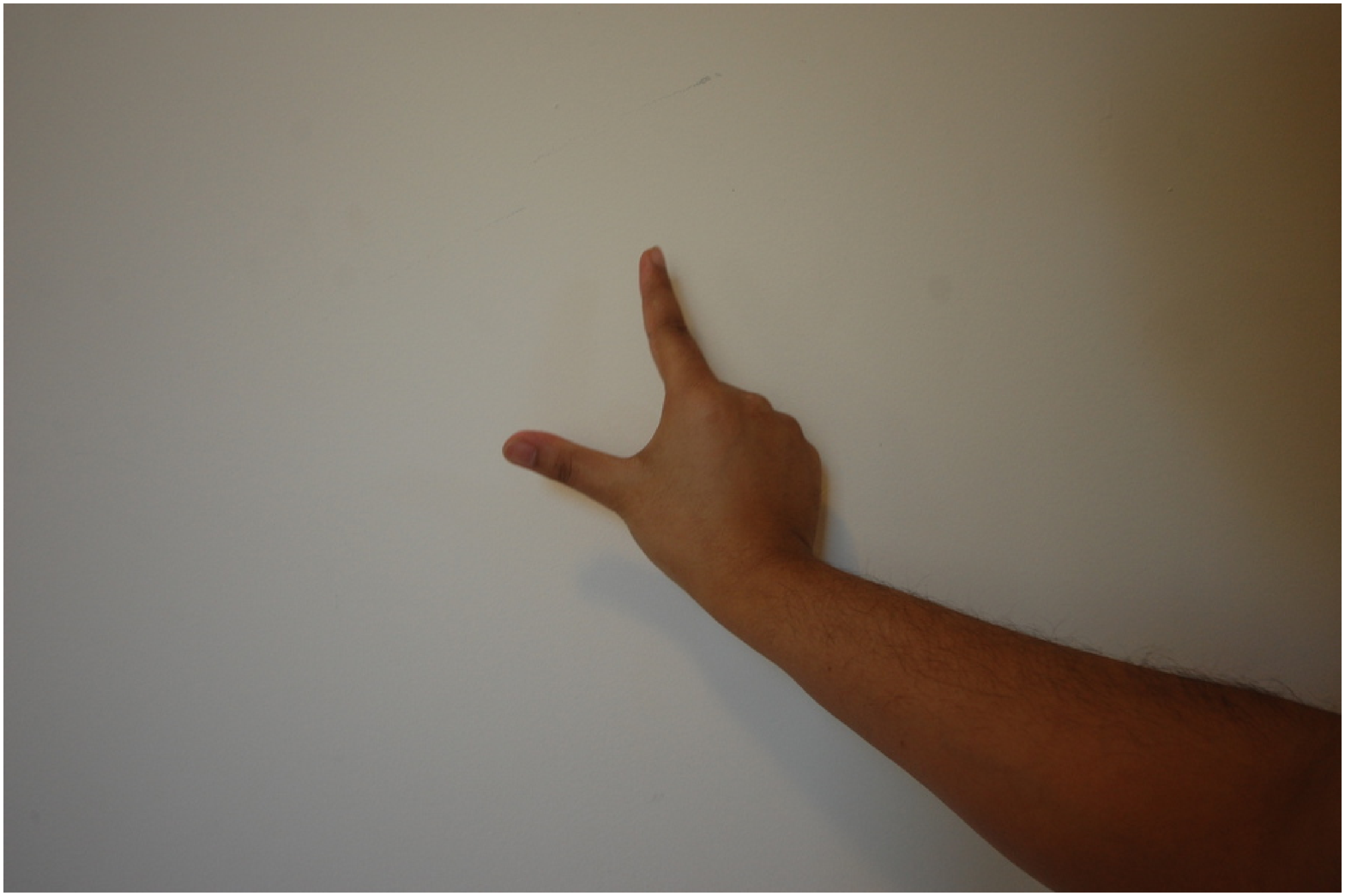}
\includegraphics[width=0.07\textwidth]{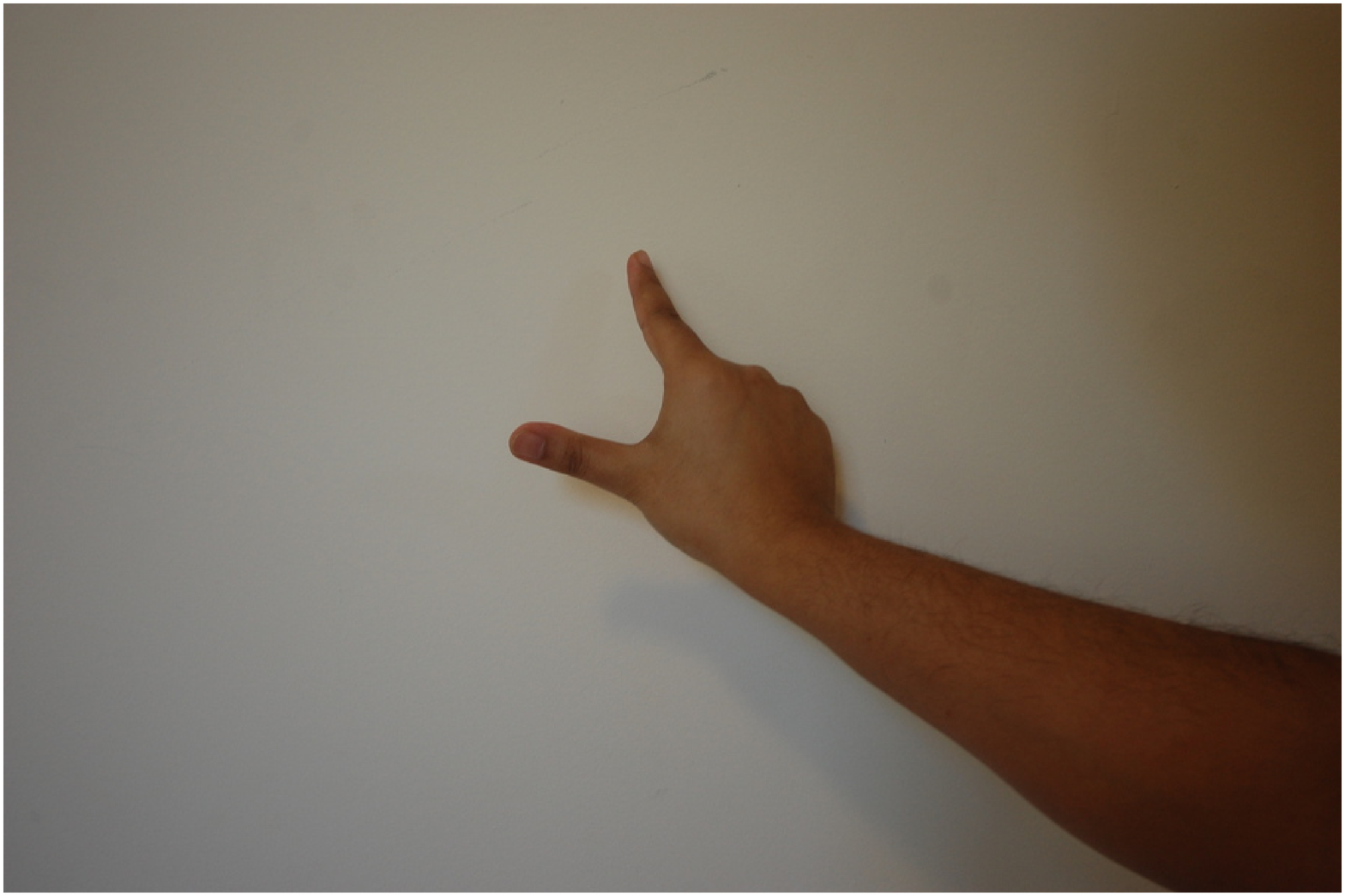}
\includegraphics[width=0.07\textwidth]{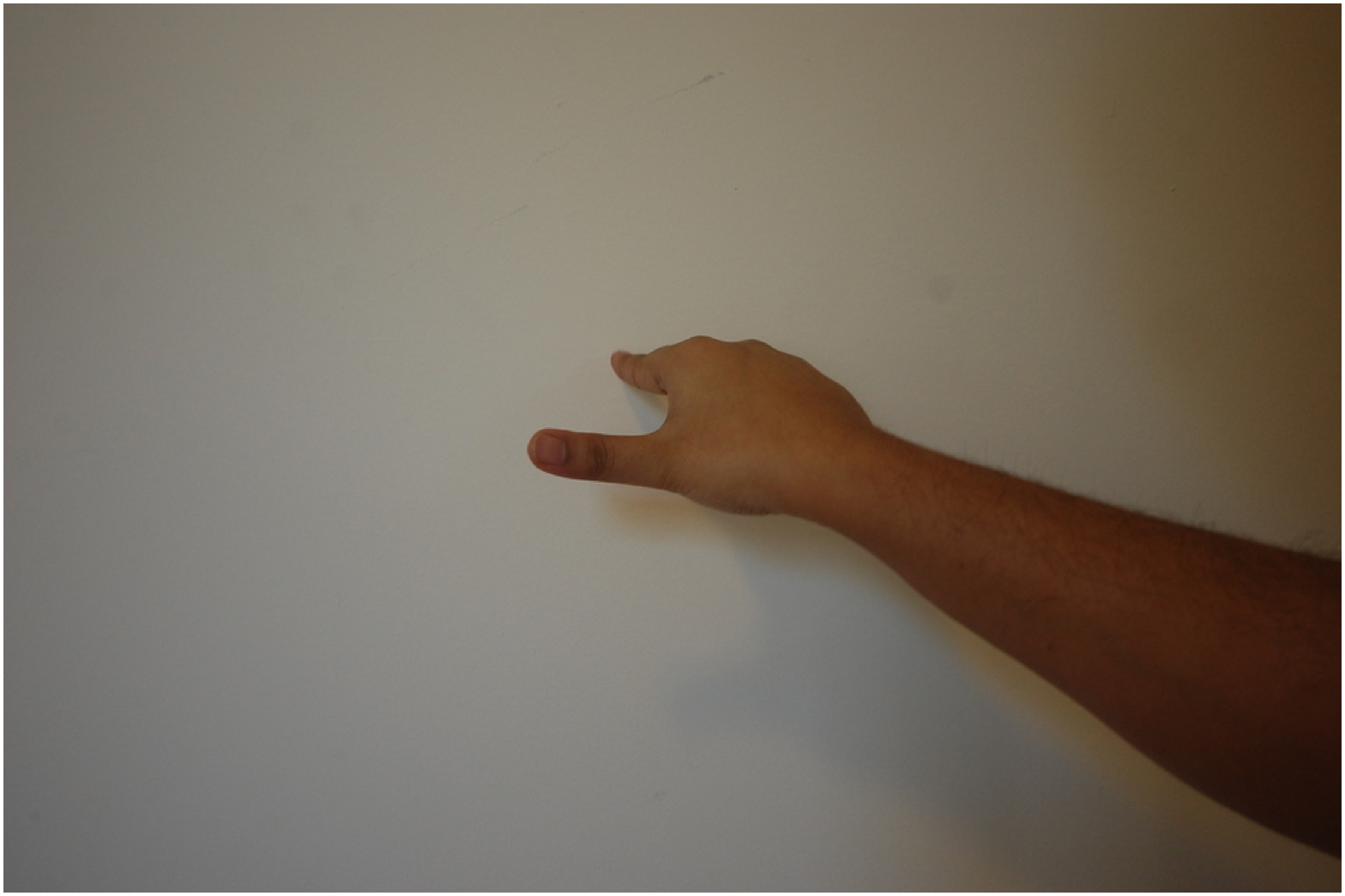}
\includegraphics[width=0.07\textwidth]{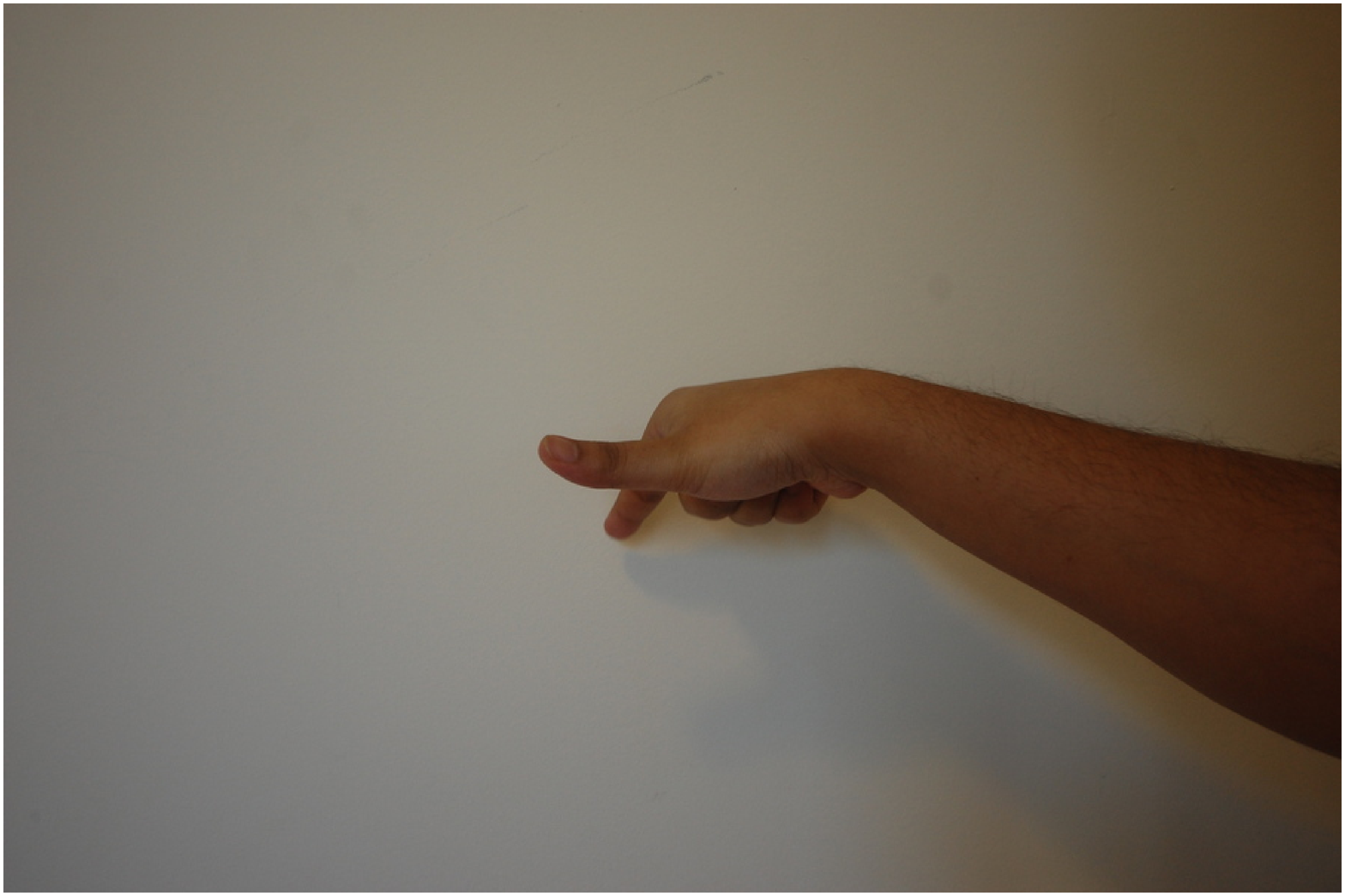}
\includegraphics[width=0.07\textwidth]{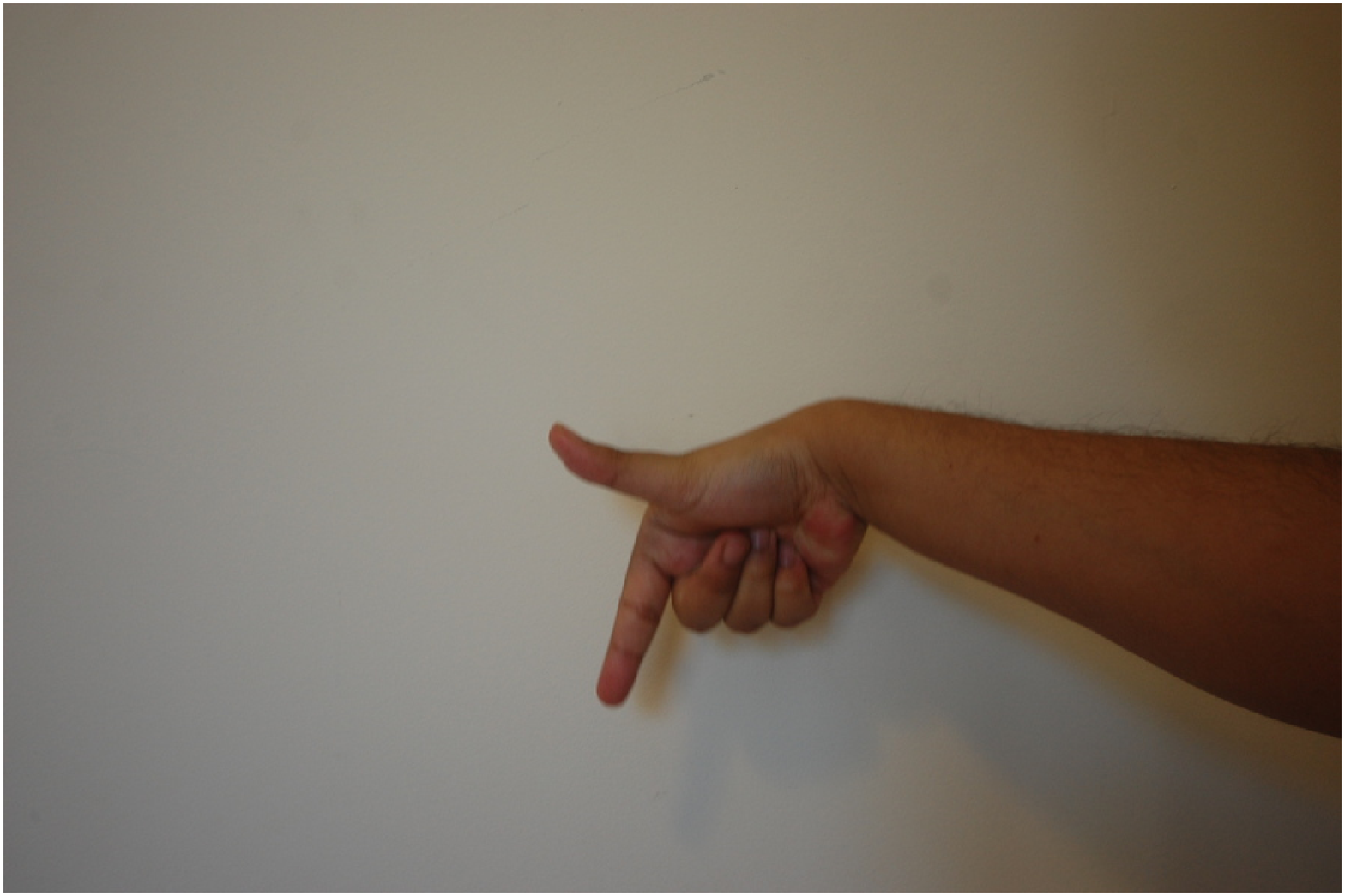}
\includegraphics[width=0.07\textwidth]{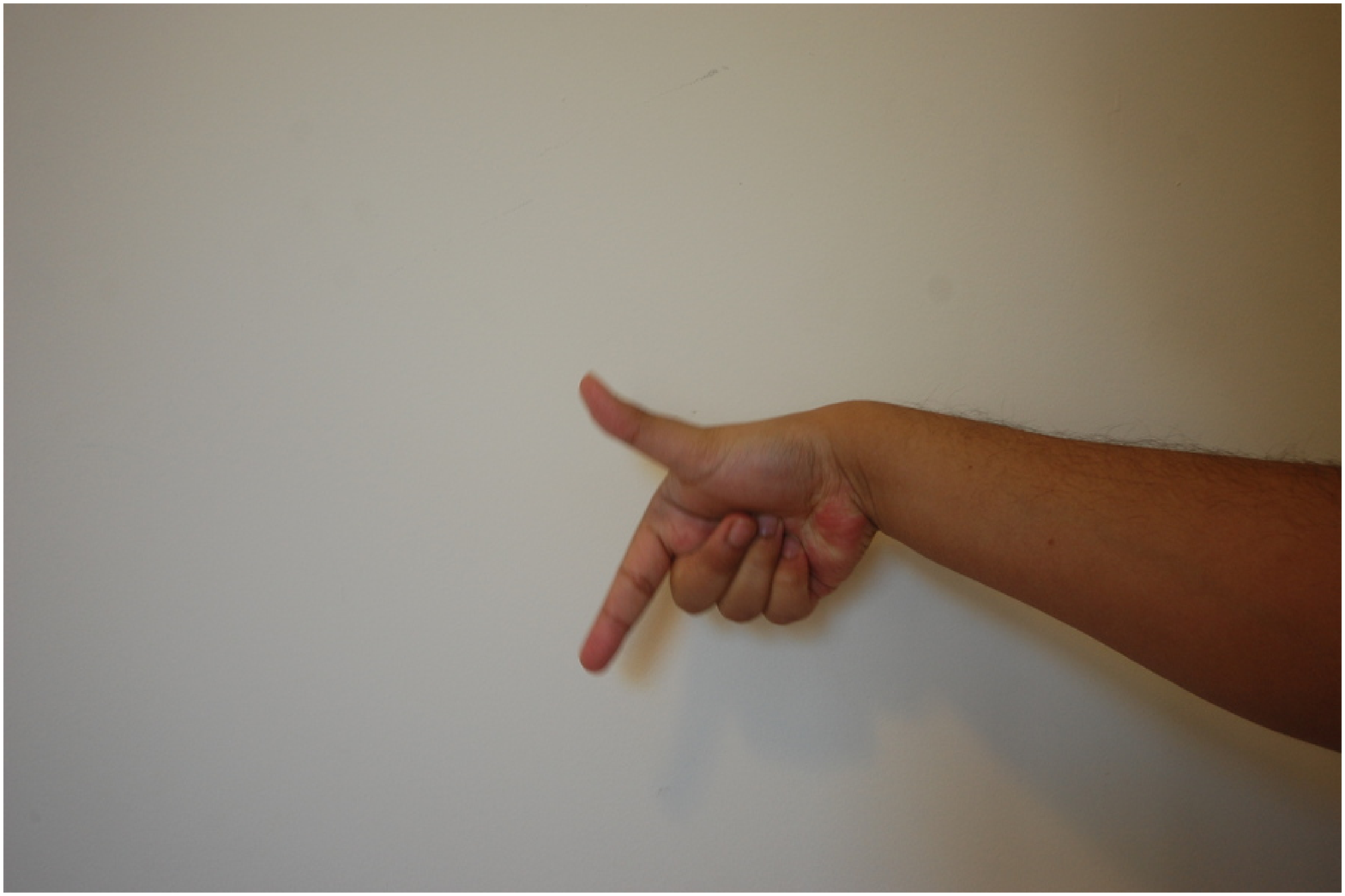}
\includegraphics[width=0.07\textwidth]{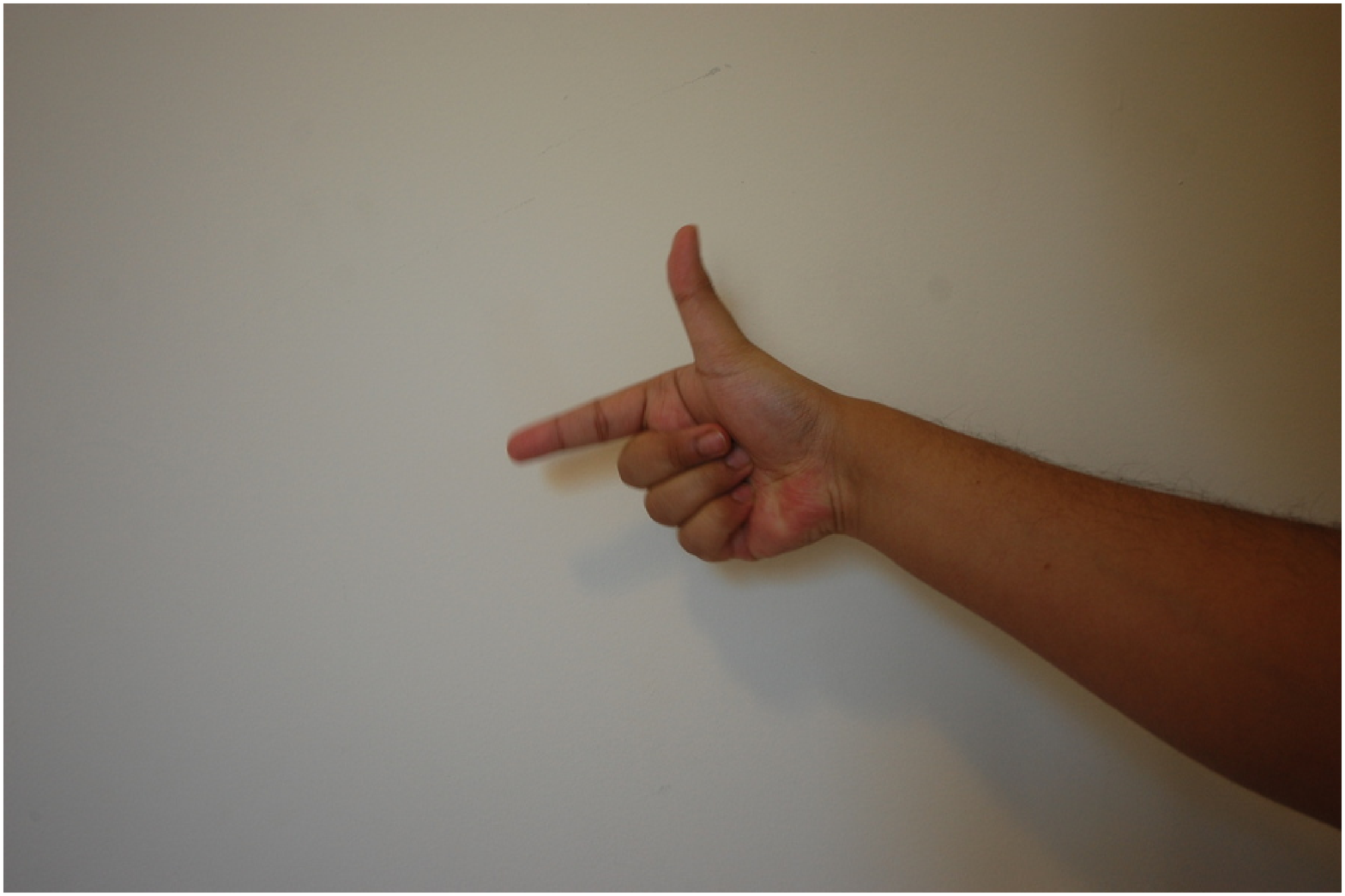}
\includegraphics[width=0.07\textwidth]{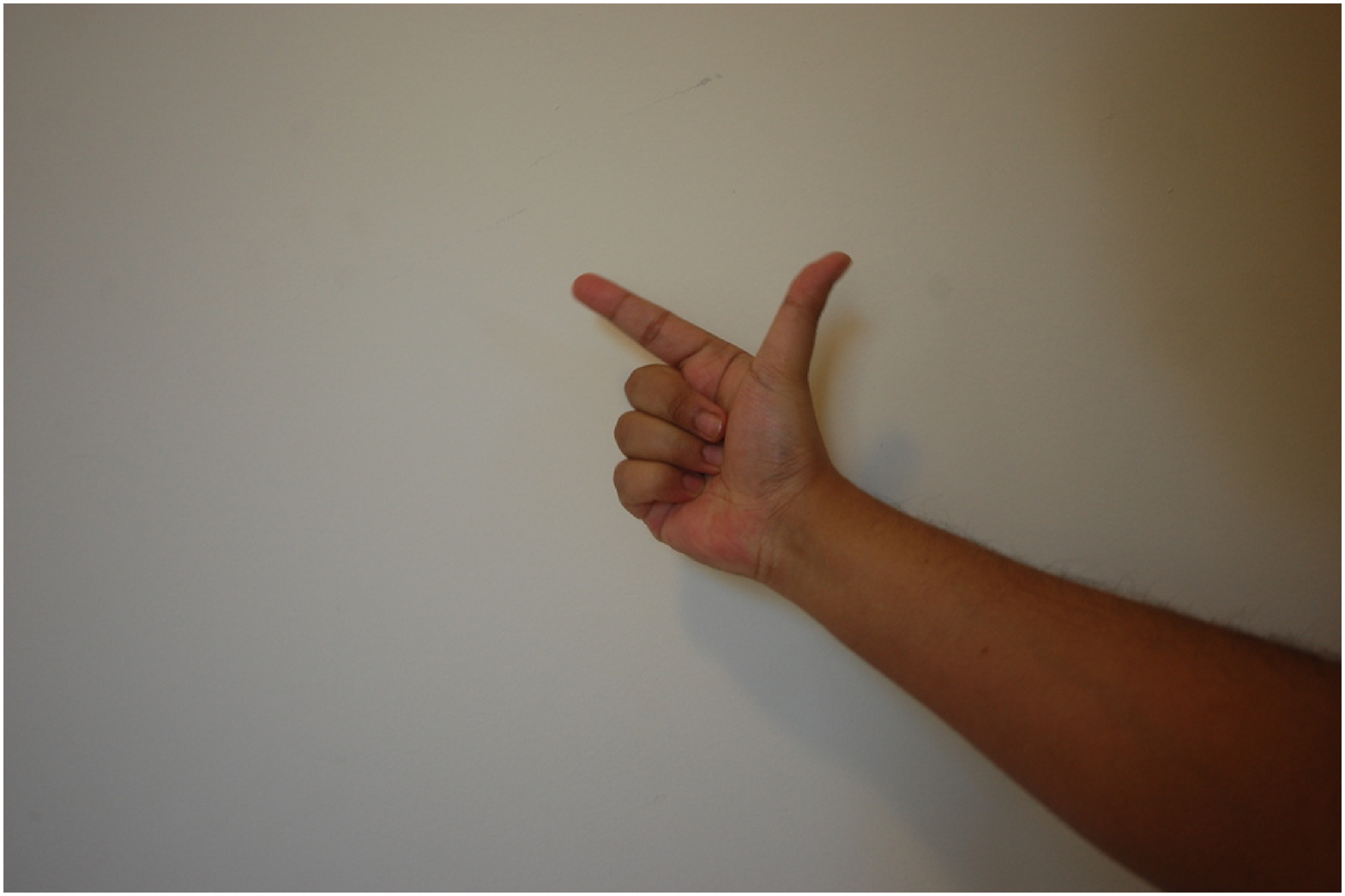}
\includegraphics[width=0.07\textwidth]{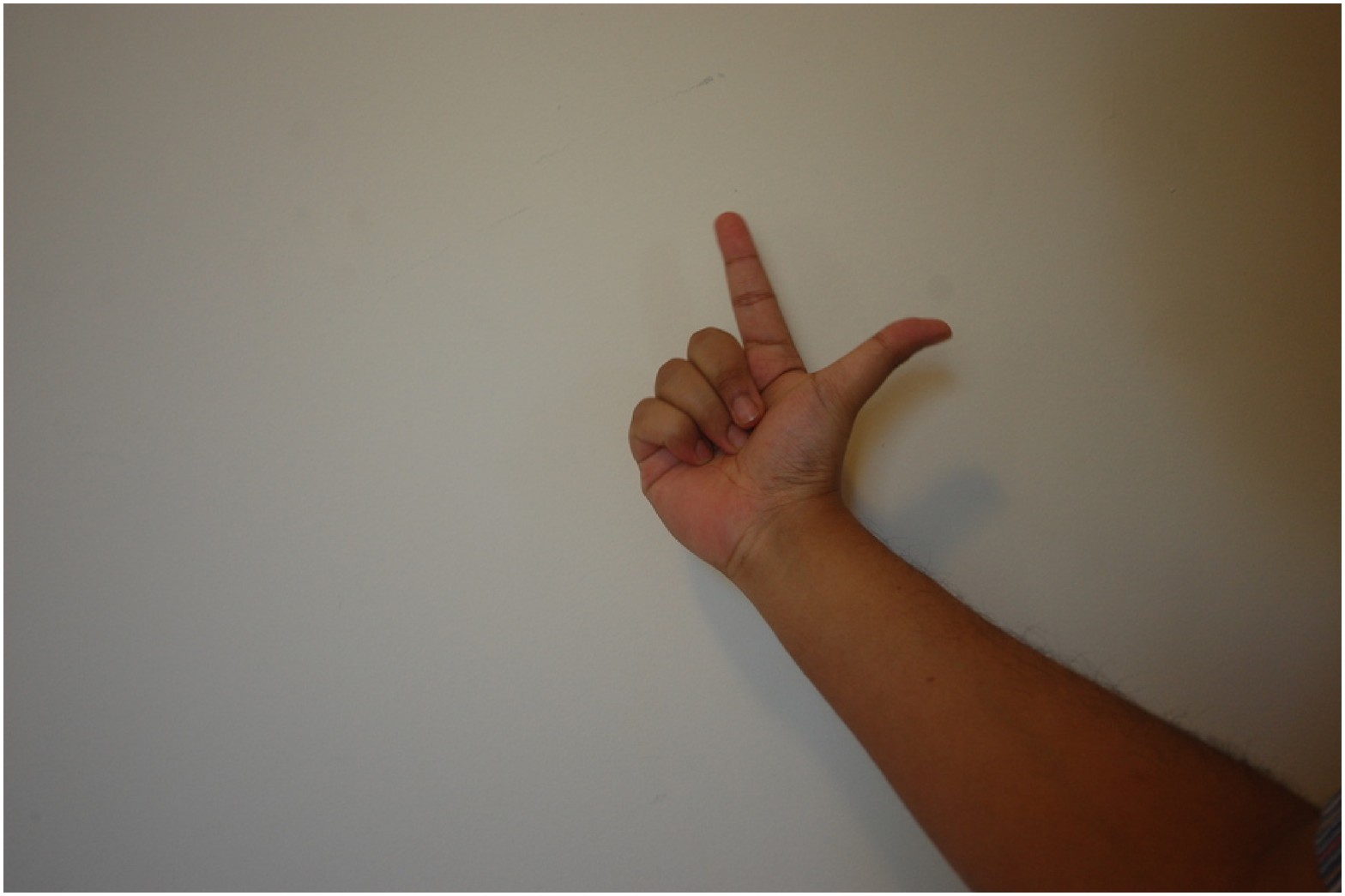}
\caption{\textbf{Null-homotopy of the loop
  $\ol{\topright\topleft}\cdot\ol{\topleft\bottomleft}\cdot\ol{\bottomleft\topright}$
  in $\SO(3)$.} Viewing the arm as $2$-dimensional, spanned by the
tangent vector to the radius and the vector from the radius to the
ulna, it traces out an extension of the map $S^1\to \SO(3)$ to a map
$\DD^2\to \SO(3)$.}\label{fig:handmotion}
\end{figure}

Extending this slightly:
\begin{definition}\label{def:standard-frames}
  Fix $m_1,m_2$, and let $m=m_1+m_2$.  By the four \emph{standard
    frames for $\RR^m=\R^{m_1}\times\R^{m_2}$} we mean the frames
  \[
  [\pm e_{11},e_{12},\dots, e_{1m_1},\pm e_{21}, e_{22},\dots,e_{2m_2}]\in
  \bigl(\{0\}\times\RR^{m_1}\times \RR^{m_2}\bigr)^{m}.
  \]

  Up to homotopy, there are exactly two paths between any pair of
  frames.  By the \emph{standard frame paths in $\R\times\RR^{m}$} we mean
  the one-parameter families of frames obtained by extending the
  coherent system of paths for $SO(3)$ specified above by the identity
  on $\RR^{m-2}=\RR^{m_1-1}\times \RR^{m_2-1}$. Abusing terminology, 
  we will sometimes say that any frame path homotopic (\rel endpoints)
  to a standard frame path is itself a standard frame path. By a
  \emph{non-standard frame path} we mean a frame path which is not
  homotopic (\rel endpoints) to one of the standard frame paths.
\end{definition}

Define
\begin{align}
  \rmap&\co \RR^m\to \RR^m & \rmap(x_1,\dots, x_m)&=(-x_1,x_2,\dots, x_m)\label{eq:rmap}\\
  \smap&\co \RR^{m_1}\times\R^{m_2}\to \RR^m & \smap(x_1,\dots,x_m)&=(x_1,\dots,
  x_{m_1},-x_{m_1+1}, x_{m_1+2},\dots,x_m).\label{eq:smap}
\end{align}
\begin{lemma}\label{lem:r-of-framing}
  Suppose $\Frame_1,\Frame_2$ are oppositely-oriented standard
  frames. Then $\rmap(\ol{\Frame_1\Frame_2})$ is the non-standard frame
  path between $\rmap(\Frame_1)$ and $\rmap(\Frame_2)$. That is,
  $\rmap$ takes
  standard frame paths between oppositely-oriented frames to
  non-standard frame paths.

  The map $\smap$ satisfies
  \[
  \ol{\topright\bottomright} \stackrel{\smap}{\longleftrightarrow}
  \ol{\bottomright\topright}
  \qquad\qquad
  \ol{\topleft\bottomleft}\stackrel{\smap}{\longleftrightarrow}
  \ol{\bottomleft\topleft}.
  \]
  In other words, $\smap$ takes the standard frame path
  $\overline{\begin{smallmatrix}+ & -\\ * & *\end{smallmatrix}}$ to the standard
  frame path $\overline{\begin{smallmatrix}- & +\\ * & *\end{smallmatrix}}$ 
  for either $*\in\{+,-\}$.
\end{lemma}
\begin{proof}
  This is a straightforward verification from the definitions.
\end{proof}

\subsection{The framed cube flow category}\label{subsec:frame-cube-flow-cat}

In this subsection, we describe certain aspects of the \emph{flow
  category $\CubeFlowCat(n)$} associated to the cube $[0,1]^n$. For a
more complete account of the story, see
\cite[Section~\ref*{KhSp:sec:cube-flow}]{RS-khovanov}. The features of
$\CubeFlowCat(n)$ in which we are
interested are the following:

\begin{enumerate}[label=(F-\arabic*),ref=F-\arabic*]
\item To a pair of vertices $u,v\in\{0,1\}^n$ with $v\leq_k u$,
  $\CubeFlowCat(n)$ associates a $(k-1)$-dimensional manifold with
  corners\footnote{It is also a $\Codim{k-1}$-manifold in the sense of \cite{Lau-top-cobordismcorners}.}
called the \emph{moduli
    space} $\Moduli_{\CubeFlowCat(n)}(u,v)$. We drop the subscript if
  it is clear from the context.
\item For vertices $v<w<u$ in $\{0,1\}^n$,
  $\Moduli(w,v)\times\Moduli(u,w)$ is identified with a subspace of
  $\del\Moduli(u,v)$.
\item Fix vertices $v\leq_k u$ in $\{0,1\}^n$; let
  $\vect{0},\vect{1}\in\{0,1\}^k$ be the minimum and the maximum
  vertex, respectively. Then $\Moduli_{\CubeFlowCat(n)}(u,v)$ can be
  identified with $\Moduli_{\CubeFlowCat(k)}(\vect{1},\vect{0})$.
\item $\Moduli_{\CubeFlowCat(n)}(\vect{1},\vect{0})$ is a point, an
  interval and a hexagon, for $n=1,2,3$, respectively, cf.\
  \Figure{hexagon}.
  \begin{figure}
    \begin{center}
      \tiny
      \psfrag{a}{$\Moduli(100,000)\times\Moduli(110,100)\times\Moduli(111,110)$}
      \psfrag{b}{$\Moduli(010,000)\times\Moduli(110,010)\times\Moduli(111,110)$}
      \psfrag{c}{$\Moduli(010,000)\times\Moduli(011,010)\times\Moduli(111,011)$}
      \psfrag{d}{$\Moduli(001,000)\times\Moduli(011,001)\times\Moduli(111,011)$}
      \psfrag{e}{$\Moduli(001,000)\times\Moduli(101,001)\times\Moduli(111,101)$}
      \psfrag{f}{$\Moduli(100,000)\times\Moduli(101,100)\times\Moduli(111,101)$}
      \psfrag{g}{$\Moduli(110,000)\times\Moduli(111,110)$}
      \psfrag{h}{$\Moduli(010,000)\times\Moduli(111,010)$}
      \psfrag{i}{$\Moduli(011,000)\times\Moduli(111,011)$}
      \psfrag{j}{$\Moduli(001,000)\times\Moduli(111,001)$}
      \psfrag{k}{$\Moduli(101,000)\times\Moduli(111,101)$}
      \psfrag{l}{$\Moduli(100,000)\times\Moduli(111,100)$}
      \psfrag{m}{$\Moduli(111,000)$}
      \psfrag{v}{}
      \includegraphics[width=0.9\textwidth]{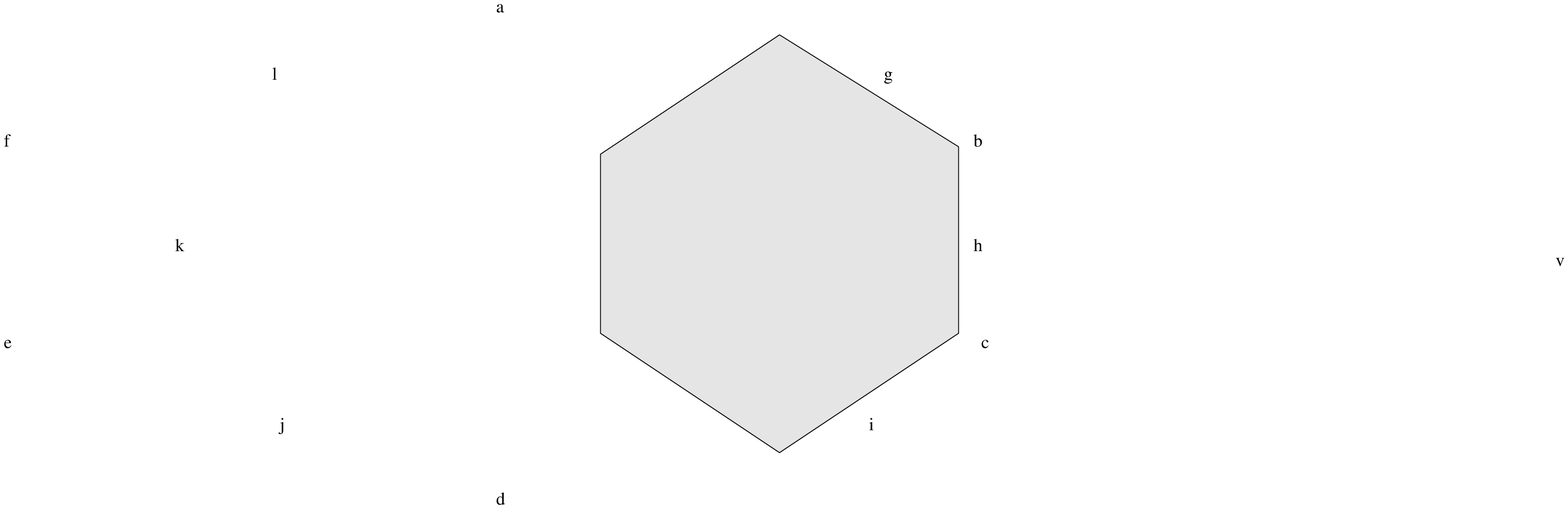}
    \end{center}
    \caption{\textbf{The hexagon
        $\Moduli_{\CubeFlowCat(3)}(111,000)$.} Each face corresponds
      to a product of lower-dimensional moduli spaces, as indicated.}\label{fig:hexagon}
  \end{figure}
\end{enumerate}

The cube flow category is also \emph{framed}.  In order to define
framings, one needs to embed the moduli spaces into Euclidean spaces;
one does so by \emph{neat embeddings} (see~\cite[Definition
2.1.4]{Lau-top-cobordismcorners} or
\cite[Definition~\ref*{KhSp:def:neat-embedding}]{RS-khovanov}). Fix
$d$ sufficiently large; for each $v\leq_k u$, $\Moduli(u,v)$ is neatly
embedded in $\R_+^{k-1}\times\R^{kd}$. These embeddings are coherent
in the sense that for each $v\leq_k w\leq_l u$,
$\Moduli(w,v)\times\Moduli(u,w)\sbs\del\Moduli(u,v)$ is embedded by
the product embedding into
$\R_+^{k-1}\times\R^{kd}\times\R_+^{l-1}\times\R^{ld}=
\R_+^{k-1}\times\{0\}\times\R_+^{l-1}\times\R^{kd+ld}\sbs
\del(\R_+^{k+l-1}\times\R^{(k+l)d})$.  The normal bundle to each of
these moduli spaces is framed. These framings are also coherent in the
sense that the product framing on $\Moduli(w,v)\times\Moduli(u,w)$
agrees with the framing induced from $\Moduli(u,v)$.

The framed cube flow category $\CubeFlowCat(n)$ is needed in the
construction of the Khovanov homotopy type. The cube flow category can
be framed in multiple ways. However, all such framings lead to the
same Khovanov homotopy type \cite[Proposition~\ref*{KhSp:prop:choice-independent}]{RS-khovanov}; hence
it is enough to consider a specific framing. Consider the following
partial framing.

\begin{defn}\label{def:0d-1d-moduli-framing}
  Let $s\in C^1(\Cube(n),\F_2)$ and $f\in C^2(\Cube(n),\F_2)$ be the
  standard sign assignment and the standard frame assignment
  from \Subsection{cube}. Fix $d$ sufficiently large.
  \begin{itemize}
  \item Consider $v\leq_1 u$ in $\{0,1\}^n$. Embed the point
    $\Moduli(u,v)$ in $\R^d$; let $[e_{1},\dots,e_{d}]$ be the
    standard basis in $\R^d$. For framing the point $\Moduli(u,v)$,
    choose the frame $[e_{1},e_{2},\dots,e_{d}]$ if
    $s(\Cube_{u,v})=0$, and choose the frame
    $[-e_{1},e_{2},\dots,e_{d}]$ if $s(\Cube_{u,v})=1$.
  \item Consider $v\leq_2 u$ in $\{0,1\}^n$; let $w_1$ and $w_2$ be
    the two other vertices in $\Cube_{u,v}$. Choose a proper embedding
    of the interval $\Moduli(u,v)$ in $\R_+\times\R^{2d}$; let
    $[\ol{e},e_{11},\dots,e_{1d},e_{21},\dots,e_{2d}]$ be the standard
    basis for $\R\times\R^{2d}$. The two endpoints
    $\Moduli(w_i,v)\times\Moduli(u,w_i)$ of the interval
    $\Moduli(u,v)$ are already framed in $\{0\}\times\R^{2d}$ by the
    product framings, say $\Frame_i$. Since $s$ is a sign assignment,
    the framings of the two endpoints, $\Frame_1$ and $\Frame_2$, are
    opposite, and hence can be extended to a framing on the
    interval. Any such extension can be treated as a path joining
    $\Frame_1$ and $\Frame_2$ in $\SO(2d+1)$,
    cf.\ \Subsection{frames-R3}. If $f(\Cube_{u,v})=0$, choose an
    extension so that the path is a standard frame path; if
    $f(\Cube_{u,v})=1$, choose an extension so that the path is a
    non-standard frame path.
  \end{itemize}
\end{defn}

As we will see in the \Subsection{sq2-on-khspace}, in order to study the
$\Sq^2$ action, one only needs to understand the framings of the
$0$-dimensional and the $1$-dimensional moduli spaces. Therefore, the
information encoded in \Definition{0d-1d-moduli-framing} is all we
need in order to study the $\Sq^2$ action. However, before we proceed
onto the next subsection, we need to check the following.

\begin{lem}\label{lem:extend-framing}
  The partial framing from \Definition{0d-1d-moduli-framing} can be
  extended to a framing of the entire cube flow category
  $\CubeFlowCat(n)$.
\end{lem}

\begin{proof}
  We frame the cube flow category in \cite[Proposition~\ref*{KhSp:prop:cube-can-be-framed}]{RS-khovanov} inductively: We start with coherent framings of
  all moduli spaces of dimension less than $k$; after changing the
  framings in the interior of the $(k-1)$-dimensional moduli spaces if
  necessary, we extend this to a framing of all $k$-dimensional moduli
  spaces.

  Therefore, in order to prove this lemma, we merely need to check
  that the framings of the zero- and one-dimensional moduli spaces from
  \Definition{0d-1d-moduli-framing} can be extended to a framing of
  the two-dimensional moduli spaces.

  Fix $v\leq_3 u$, and fix a neat embedding of the hexagon
  $\Moduli(u,v)$ in $\R_+^2\times\R^{3d}$.  Let
  $[\ol{e}_1,\ol{e}_2,e_{11},\dots,e_{1d},e_{21},\dots,e_{2d},e_{31},\dots,e_{3d}]$
  be the standard basis for $\R_+^2\times\R^{3d}$. The boundary $K$ is
  a framed $6$-gon embedded in $\del(\R_+^2\times\R^{3d})$. Let us
  flatten the corner in $\del(\R_+^2\times\R^{3d})$ so that
  $[\ol{e}_1=-\ol{e}_2,e_{11},\dots,e_{3d}]$ is the standard basis in
  the flattened $\R\times\R^{3d}$. After this flattening operation, we
  can treat $K$ as a framed $1$-manifold in $\R\times\R^{3d}$---see
  \Figure{flattening}---which in turn represents some element
  $\eta_{u,v}\in\pi_4(S^3)=\Z/2$ by the Pontrjagin-Thom
  correspondence.  We want to show that $K$ is null-concordant, i.e.,
  that $\eta_{u,v}=0$.

  \begin{figure}
    \[
    \xymatrix{
      \vcenter{\hbox{\psfrag{rp}{$\R_+$}\psfrag{r3}{$\R^{3d}$}\includegraphics[height=0.3\textwidth]{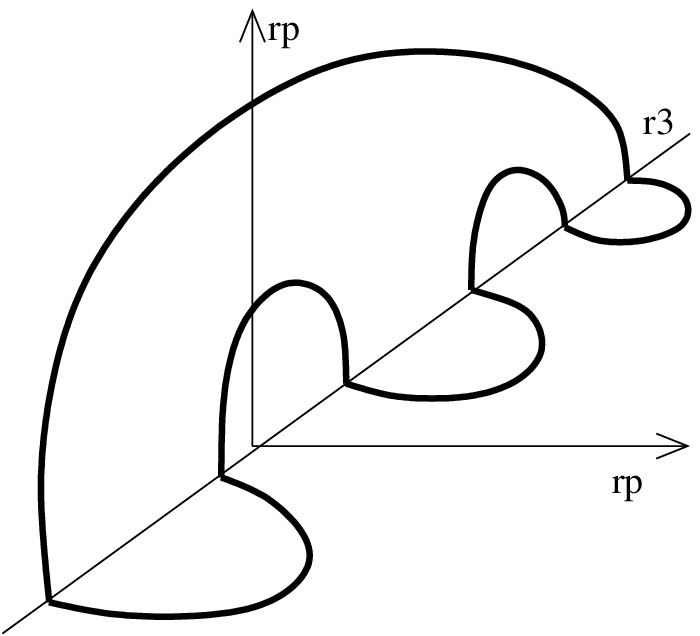}}}\ar[r]&
      \vcenter{\hbox{\psfrag{r}{$\R$}\psfrag{r3}{$\R^{3d}$}\includegraphics[height=0.3\textwidth]{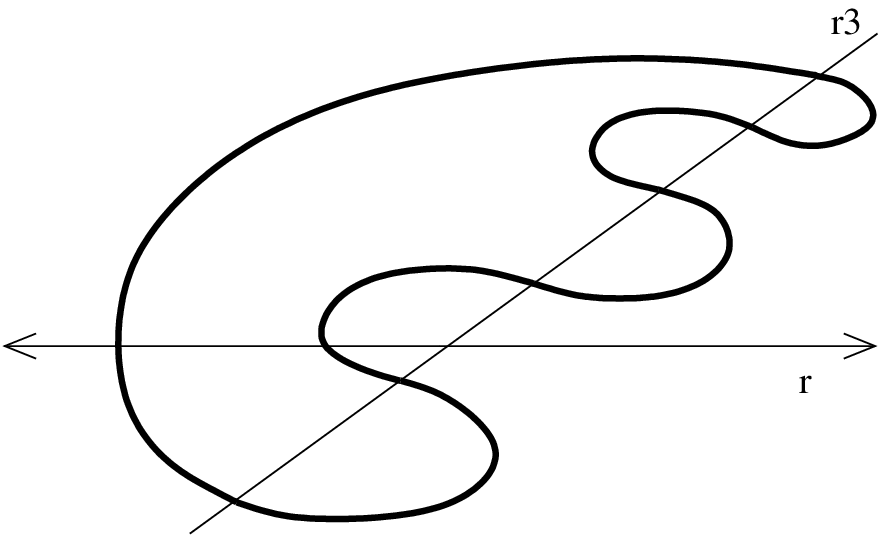}}}
    }
    \]
    \caption{\textbf{The embedding of $\del\Moduli(u,v)$.} Left: the embedding in
      $\del(\R_+^2\times\R^{3d})$. Right: the corresponding embedding in $\R\times\R^3$
      obtained by flattening the corner.}\label{fig:flattening}
  \end{figure}

  As in \Subsection{frames-R3}, $K$ can also be treated as a loop in
  $\SO(3d+1)$,
  and thus represents some element $h_{u,v}\in
  H_1(\SO(3d+1);\Z)=\F_2$. The element $h_{u,v}$ is non-zero if and only
  if $\eta_{u,v}$ is zero; therefore, we want to show $h_{u,v}=1$.

  \begin{figure}
    \makebox[\textwidth][c]{
      \xymatrix@C=10ex@!R=2ex{
        &u\ar@{-}[ddl]|{a_1}\ar@{-}[dd]|{a_2}\ar@{-}[ddr]|{a_3}\ar@{.}[ddddl]|{f_1}\ar@/^2ex/@{.}[dddd]|(0.7){f_2}\ar@{.}[ddddr]|{f_3}&\\
        &&&\\
        t_1\ar@{-}[dd]|{b_1}\ar@{-}[ddr]|(0.4){b_2}
        &t_2\ar@{-}[ddl]|(0.6){b_3}\ar@{-}[ddr]|(0.6){b_4} &
        t_3\ar@{-}[ddl]|(0.4){b_5}\ar@{-}[dd]|{b_6}\\
        &&& \vcenter{\hbox{
            \begin{overpic}[width=0.5\textwidth]{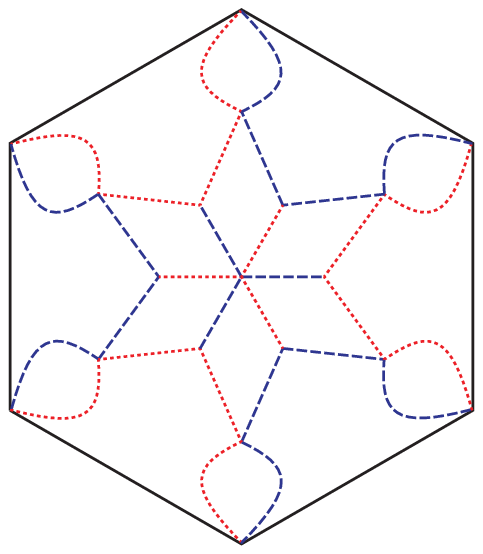}
              \put (41,95) {\tiny $c_1b_1a_1$}
              \put (80,75) {\tiny $c_1b_3a_2$}
              \put (00,75) {\tiny $c_2b_2a_1$}
              \put (00,25) {\tiny $c_2b_5a_3$}
              \put (80,25) {\tiny $c_3b_4a_2$}
              \put (41,05) {\tiny $c_3b_6a_3$}
              \put (47,75) {\tiny $c_10a_1$}
              \put (60,65) {\tiny $c_10a_2$}
              \put (23,65) {\tiny $c_20a_1$}
              \put (23,35) {\tiny $c_20a_3$}
              \put (60,35) {\tiny $c_30a_2$}
              \put (47,25) {\tiny $c_30a_3$}
              \put (52,63) {\tiny $c_100$}
              \put (60,50) {\tiny $00a_2$}
              \put (25,50) {\tiny $c_200$}
              \put (35,37) {\tiny $00a_3$}
              \put (52,37) {\tiny $c_300$}
              \put (35,63) {\tiny $00a_1$}
              \put (43.4,51) {\tiny $000$}
              \put (43,85) {\tiny \txt{$b_1c_1$\\$+b_1$}}
              \put (72,68) {\tiny \txt{$b_3c_1$\\$+b_3$}}
              \put (72,35) {\tiny \txt{$b_4c_3$\\$+b_4$}}
              \put (43,17) {\tiny \txt{$b_6c_3$\\$+b_6$}}
              \put (13,35) {\tiny \txt{$b_5c_2$\\$+b_5$}}
              \put (13,68) {\tiny \txt{$b_2c_2$\\$+b_2$}}
              \put (43,60) {\tiny $a_1c_1$}
              \put (43,40) {\tiny $a_3c_3$}
              \put (53,55) {\tiny $a_2c_1$}
              \put (33,55) {\tiny $a_1c_2$}
              \put (53,45) {\tiny $a_2c_3$}
              \put (33,45) {\tiny $a_3c_2$}
              \put (60,75) {\tiny $f_1$}
              \put (15,50) {\tiny $f_2$}
              \put (60,25) {\tiny $f_3$}
              \put (30,75) {\tiny $g_1$}
              \put (75,50) {\tiny $g_2$}
              \put (30,25) {\tiny $g_3$}
            \end{overpic}
          }}\\
        w_1&w_2&w_3 \\
        &&&\\
        &v\ar@{-}[uul]|{c_1}\ar@{-}[uu]|{c_2}\ar@{-}[uur]|{c_3}\ar@{.}[uuuul]|{g_1}\ar@/^2ex/@{.}[uuuu]|(0.7){g_2}\ar@{.}[uuuur]|{g_3}&
      }}
    \caption{\textbf{Framing on the hexagon.} Left: notation for the
      framings at the vertices and edges of the hexagon. Right: the
      graph $G$.}
    \label{fig:HexagonFlower}
  \end{figure}

  Let $t_1,t_2,t_3,w_1,w_2,w_3$ be the six vertices between $u$ and
  $v$ in the cube, with $w_1\leq_1 t_1,t_2$ and $w_2\leq_1 t_1,t_3$
  and $w_3\leq_1 t_2,t_3$. For $i\in\{1,2,3\}$, let
  $s(\Cube_{u,t_i})=a_i$, $s(\Cube_{w_i,v})=c_i$,
  $f(\Cube_{u,w_i})=f_i$ and $f(\Cube_{t_i,v})=g_i$. Finally let
  $s(\Cube_{t_1,w_1})=b_1$, $s(\Cube_{t_1,w_2})=b_2$,
  $s(\Cube_{t_2,w_1})=b_3$, $s(\Cube_{t_2,w_3})=b_4$,
  $s(\Cube_{t_3,w_2})=b_5$ and $s(\Cube_{t_3,w_3})=b_6$. This
  information is encoded in the first part of \Figure{HexagonFlower}.

  Consider the tri-colored planar graph $G$ in the second part of
  \Figure{HexagonFlower}.  The vertices represent frames in
  $\{0\}\times\R^{3d}$ as follows: if a vertex is labeled $cba$, then
  it represents the frame
  \[
  [(-1)^c e_{11},e_{12},\dots,e_{1d}, (-1)^b
  e_{21},e_{22},\dots,e_{2d},(-1)^a e_{31},e_{32},\dots,e_{3d}].
  \]
  Each edge represents a frame path joining the frames at its
  endpoints as follows.
  \begin{itemize}
  \item If the edge is colored black, i.e., if it is
    at the boundary of the hexagon, then it is one of the edges of the
    framed $6$-cycle $K$.
  \item If the edge is colored blue (dashed), then it represents the image under
    flattening of the following path in
    $\{0\}\times\R_+\times\R^{d}\times\R^d\times\R^d$: it is constant on
    the first $\R^d$ and is a standard frame path on the remaining
    $\R_+\times\R^d\times\R^d$.
  \item If the edge is colored red (dotted), then it represents the image
    under flattening of the following path in
    $\R_+\times\{0\}\times\R^d\times\R^d\times\R^d$: it is constant on
    the last $\R^d$ and is a standard frame path on the remaining
    $\R_+\times\R^d\times\R^d$.
  \end{itemize}
  The element $h_{u,v}\in H_1(\SO(3d+1))$ is represented by the black
  $6$-cycle in $G$. In order to compute $h_{u,v}$, we will
  compute the homology classes of some other cycles in $G$.

  Consider the black-blue $5$-cycle joining $c_1b_1a_1$, $c_10a_1$,
  $c_100$, $c_10a_2$ and $c_1b_3a_2$. Modulo extending by the constant
  map on $\R^d$, the four blue edges represent standard frame paths in
  $\R\times\R^{2d}$ and the black edge is standard if and only if
  $f_1=0$. Therefore, this cycle represents the element $f_1\in
  H_1(\SO(3d+1))$. We denote this by writing $f_1$ in the pentagonal
  region bounded by this $5$-cycle in $G$. The homology classes
  represented by the other $5$-cycles are shown in
  \Figure{HexagonFlower}.

  Next consider the red-blue $4$-cycle connecting $c_10a_1$, $00a_1$,
  $000$ and $c_100$. If $a_1=0$, then the blue edges represent the
  constant paths, and the two red edges represent the same path;
  therefore the cycle is null-homologous. Similarly, if $c_1=0$, the
  cycle is null-homologous as well. Finally, if $a_1=c_1=1$, it is easy
  to check from the definition of standard paths
  (\Subsection{frames-R3}) that the cycle represents the generator of
  $H_1(\SO(3d+1))$. Therefore, the cycle represents the element
  $a_1c_1$. The contributions from such $4$-cycles is also shown in
  \Figure{HexagonFlower}.

  Finally, consider the red-blue $2$-cycle connecting $c_1b_1a_1$ and
  $c_10a_1$. If $b_1=0$, then both the red and the blue edges
  represent the constant paths, and hence the cycle is
  null-homologous. So let us concentrate on the case when $b_1=1$. Let
  \begin{align*}
    \Frame_1&=[(-1)^{c_1}e_{11},\dots,e_{1d},-e_{21},\dots,e_{2d},(-1)^{a_1}e_{31},\dots,e_{3d}]\text{ and}\\
    \Frame_2&=[(-1)^{c_1}e_{11},\dots,e_{1d}, e_{21},\dots,e_{2d},(-1)^{a_1}e_{31},\dots,e_{3d}]
  \end{align*}
  be the two frames in $\{0\}\times\R^{3d}$. The blue edge represents
  the path from $\Frame_1$ to $\Frame_2$ where the $(d+2)\th$ vector
  rotates $180^{\circ}$ in the
  $\langle\ol{e}_2,{e}_{21}\rangle$-plane and equals
  $\ol{e}_2=-\ol{e}_1$ halfway through. The red edge also represents a
  path where the $(d+2)\th$ vector rotates $180^{\circ}$ in the
  $\langle\ol{e}_1,{e}_{21}\rangle$-plane. However, halfway
  through, it equals $\ol{e}_1$ if $c_1=0$, and it equals $-\ol{e}_1$
  if $c_1=1$. Hence, when $b_1=0$, the red-blue $2$-cycle is
  null-homologous if and only if $c_1=1$. Therefore, the cycle
  represents the element $b_1(c_1+1)=b_1c_1+b_1$. These contributions
  are also shown in \Figure{HexagonFlower}.

  We end the proof with a mild exercise in addition.
  \begin{flalign*}
    h_{u,v}&\mathrlap{=(f_1+f_2+f_3+g_1+g_2+g_3)+(b_1+b_2+b_3+b_4+b_5+b_6)}\\
    &\mathrlap{\quad{}+(b_1c_1+b_3c_1+b_4c_3+b_6c_3+b_5c_2+b_2c_2)+(a_1c_1+a_2c_1+a_2c_3+a_3c_3+a_3c_2+a_1c_2)}\\
    &\mathrlap{=(c_1+c_2+c_3)+(b_1+b_2+b_3+b_4+b_5+b_6)}&\text{(by \Lemma{frame-assignment-sum})}\\
    &\mathrlap{\quad{}+(c_1(a_1+a_2+b_1+b_3)+c_2(a_1+a_3+b_2+b_5)+c_3(a_2+a_3+b_4+b_6))}\\
    &\mathrlap{=(c_1+c_2+c_3)+(b_1+b_2+b_3+b_4+b_5+b_6)}\\
    &\mathrlap{\quad{}+(c_1+c_2+c_3)}&\text{(since $s$ is a sign assignment)}\\ 
    &\mathrlap{=(b_1+b_2+c_1+c_2)+(b_3+b_4+c_1+c_3)+(b_5+b_6+c_2+c_3)}\\
    &\mathrlap{=1+1+1}&\text{($s$ is still a sign assignment)}\\
    &\mathrlap{=1.}&\qedhere
  \end{flalign*}
\end{proof}

\subsection{\texorpdfstring{$\Sq^2$}{Sq2} for the Khovanov homotopy type}\label{subsec:sq2-on-khspace}
Fix a link diagram $L$ and an integer $\ell$, and let $\KhSpace^\ell(L)$
denote the Khovanov homotopy type constructed in~\cite{RS-khovanov}.
We want to study the Steenrod square
\begin{align*}
  \Sq^2&\from\wt{H}^{\kappa}(\KhSpace^{\ell}(L))\to\wt{H}^{\kappa+2}(\KhSpace^{\ell}(L)).\\
  \shortintertext{The spectrum $\KhSpace^\ell(L)$ is a formal
    de-suspension $\Sigma^{-N}Y_{\ell}$ of a CW complex $Y_{\ell}$ for some
    sufficiently large $N$. Therefore, we want to understand the
    Steenrod square} \Sq^2&\from H^{N+\kappa}(Y_{\ell};\F_2)\to H^{N+\kappa+2}(Y_{\ell};\F_2).
\end{align*}

Before we get started, we give names to a few maps which will make
regular appearances. Fix $m_1,m_2\geq 2$ and let $m=m_1+m_2$.
First, recall that we have maps $\rmap, \smap\co \RR^m=\R^{m_1}\times\R^{m_2}\to
\RR^m$ given by Formulas~\eqref{eq:rmap} and~\eqref{eq:smap},
respectively. Next, let 
\[
\Pi\co \DD^m\to K_m^{(m+1)}
\]
be the composition of the projection map $\DD^m\to
\DD^m/\bdy\DD^m=S^m$ and the inclusion $S^m\to K_m^{(m+1)}$. Let
\[
\Xi\co [0,1]\times\DD^m \to K_m^{(m+1)}
\]
be the map induced by the identification of $[0,1]\times\DD^m$ with
the $(m+1)$-cell $e^{m+1}$ of $K_m^{(m+1)}$; the map $\Xi$ collapses
$[0,1]\times \bdy \DD^m$ to the basepoint and maps each of
$\{0\}\times\DD^m$ and $\{1\}\times\DD^m$ to the $m$-skeleton $S^m\subset
K_m^{(m+1)}$ by $\Pi$ and $\Pi\circ \rmap$, respectively. The map $\Xi$
factors through a map
\[
\overline{\Xi}\co [0,1]\times S^m\to K_m^{(m+1)}.
\]

We construct a CW complex $X$ as follows.  Choose
real numbers $\epsilon$ and $R$ with $0<\epsilon\ll R$. Then:

\noindent\textbf{Step 1:} Start with a unique $0$-cell $e^0$.

\noindent\textbf{Step 2:} For each Khovanov generator $\ob{x}_i\in \KhGen^{\kappa,\ell}$, $X$ has a
corresponding cell 
\[
f_i^{m}=\{0\}\times\{0\}\times [-\epsilon,\epsilon]^{m_1}\times [-\epsilon,\epsilon]^{m_2}.
\]
The boundary of $f_i^{m}$ is glued to $e^0$.

\noindent\textbf{Step 3:} For each Khovanov generator $\ob{y}_j\in \KhGen^{\kappa+1,\ell}$, $X$
has a corresponding cell
\[
f_j^{m+1}=[0,R]\times\{0\}\times [-R,R]^{m_1}\times [-\epsilon,\epsilon]^{m_2}.
\]

The boundary of $f_j^{m+1}$ is attached to $X^{(m)}$ as follows. 
If $\ob{y}_j$ occurs in $\diff_{\Z} \ob{x}_i$ with sign
$\epsilon_{i,j}\in\{\pm 1\}$ then 
we embed $f_i^m$ in $\bdy f_j^{m+1}$ by a map of the form
\begin{equation}\label{eq:fj-fi}
  \begin{split}
    (0,0,x_1,x_2,\dots,x_{m_1}&,y_1,\dots,y_{m_2})\\
    &\mapsto (0,0,\epsilon_{i,j}x_1+a_1, x_2+a_2,\dots, x_{m_1}+a_{m_1},
    y_1,\dots,y_{m_2})
  \end{split}
\end{equation}
for some vector $(a_1,\dots,a_{m_1})$. Call the
image of this embedding $\Cell[i]{j}$. Then $\Cell[i]{j}$ is mapped
to $f_i^m$ by the specified identification, and $(\bdy
f_j^m)\setminus \bigcup_i \Cell[i]{j}$ is mapped to the basepoint
$e^0$.

We choose the vectors $(a_1,\dots,a_{m_1})$ so that the different
$\Cell[i]{j}$'s are disjoint. Write
$p_{i,j}=(0,a_1,\dots,a_{m_1},0,\dots,0)\in f_j^{m+1}$.

\noindent\textbf{Step 4:} For each Khovanov generator $\ob{z}_k\in \KhGen^{\kappa+2,\ell}$, $X$
has a corresponding cell 
\[
f_k^{m+2}=[0,R]\times[0,R]\times [-R,R]^{m_1}\times [-R,R]^{m_2}.
\]

The boundary of $f_k^{m+2}$ is attached to $X^{(m+1)}$ as
follows. First, we choose some auxiliary data:
\begin{itemize}
\item If $\ob{z}_k$ occurs in $\diff_{\Z} \ob{y}_j$ with sign
  $\epsilon_{j,k}\in\{\pm 1\}$ then we embed $f_j^{m+1}$ in $\bdy
  f_k^{m+2}$ by a map of the form
  \begin{equation}\label{eq:fk-fj}
    \begin{split}
      (x_0,0,x_1,\dots,x_{m_1}&,y_1,\dots,y_{m_2})\\
      &\mapsto (x_0,0,x_1, x_2,\dots, x_{m_1},0,
      \epsilon_{j,k}y_1+b_1,\dots,y_{m_2}+b_{m_2})
    \end{split}
  \end{equation}
  for some vector $(b_1,\dots,b_{m_2})$. Call the image of this
  embedding $\Cell[j]{k}$. Once again, we choose the vectors
  $(b_1,\dots,b_{m_2})$ so that the different $\Cell[j]{k}$'s are
  disjoint.
\item Let $\ob{x}_i\in \KhGen^{\kappa,\ell}$. If the set of generators
  $\InBetween{\ob{z}_k}{\ob{x}_i}$ between $\ob{z}_k$ and
  $\ob{x}_i$ is nonempty then $\InBetween{\ob{z}_k}{\ob{x}_i}$
  consists of $2$ or $4$ points, and these points are identified in
  pairs (via the ladybug matching of \Subsection{ladybug-matching}). Write
  $\InBetween{\ob{z}_k}{\ob{x}_i}=\{\ob{y}_{j_{\be}}\}$. For each
  $j_{\be}$, the cell $\Cell[i]{j_{\be}}$ can be viewed as lying in the
  boundary of $\Cell[j_{\be}]{k}$. Consider the point $p_{i,j_{\be}}$ in the
  interior of $\Cell[i]{j_{\be}}\subset \bdy \Cell[j_{\be}]{k}$. Each of the
  points $p_{i,j_{\be}}$ inherits a framing, i.e., a trivialization of
  the normal bundle to $p_{i,j_{\be}}$ in $\bdy \Cell[j_{\be}]{k}$, from the
  map $f_i^m\to \bdy f_k^{m+2}$,
  \begin{equation*}
    \begin{split}
      \qquad(0&,0,x_1,\dots,x_{m_1},y_1,\dots,y_{m_2})\\
      &\qquad\mapsto (0,0,
      \epsilon_{i,j_{\be}}x_1+a_1, x_2+a_2,\dots, x_{m_1}+a_{m_1},
      \epsilon_{j_{\be},k}y_1+b_1,\dots,y_{m_2}+b_{m_2}).
    \end{split}
  \end{equation*}
  Notice that the framing of $p_{i,j_{\be}}$ is one of the standard
  frames for $\RR^m=\R^{m_1}\times\R^{m_2}$ (\Definition{standard-frames}).
  
  The pair of generators $(\ob{x}_i, \ob{z}_k)$ specifies a
  $2$-dimensional face of the hypercube $\Cube(n)$. Let
  $\Cube_{k,i}$ denote this face. The standard frame assignment $f$
  of \Subsection{cube} assigns an element $f(\Cube_{k,i})\in \F_2$ to
  the face $\Cube_{k,i}$.
  
  The matching of elements of $\InBetween{\ob{z}_k}{\ob{x}_i}$
  matches the points $p_{i,j_a}$ in pairs. Moreover, it follows
  from the definition of the sign assignment that matched pairs of
  points have opposite framings. For each matched pair of points
  choose a properly embedded arc 
  \[
  \PT\subset \{0\}\times[0,R]\times [-R,R]^{m_1}\times
  [-R,R]^{m_2}\sbs \del f_k^{m+2}
  \] 
  connecting the pair of points. The endpoints of $\PT$ are
  framed. Extend this to a framing of the normal bundle to $\PT$ in
  $\bdy f_k^{m+2}$. If $f(\Cube_{k,i})=0$, then choose this framing to
  be isotopic \rel boundary to a standard frame path for
  $\{0\}\times\R\times\RR^m$; if $f(\Cube_{k,i})=1$, then choose this
  framing to be isotopic \rel boundary to a non-standard frame path
  for $\{0\}\times\R\times\RR^m$.
  
  We call these arcs $\PT$ \emph{Pontrjagin-Thom arcs}, and denote
  the set of them by $\{\PT_{i_1,1},\dots,\allowbreak\PT_{i_A,n_A}\}$ where
  the arc $\PT_{i_\alpha,\imath}$ comes from the generator
  $\ob{x}_{i_\alpha}\in\KhGen^{\kappa,\ell}$.
\end{itemize}

The choice of these auxiliary data is illustrated in \Figure{attach}.
Now, the attaching map on $\bdy f_k^{m+2}$ is given as follows:
\begin{itemize}
\item The interior of $\Cell[j]{k}$ is mapped to $f_j^{m+1}$ by
  (the inverse of) the identification in Formula~\eqref{eq:fk-fj}.
\item A tubular neighborhood of each Pontrjagin-Thom arc
  $\nbd(\PT_{i_\alpha,\imath})$ is mapped to $f_{i_\alpha}^m$ as
  follows. The framing identifies 
  \[
  \nbd(\PT_{i_\alpha,\imath})\cong
  \PT_{i_\alpha,\imath}\times [-\epsilon,\epsilon]^{m_1+m_2}\cong \PT_{i_\alpha,\imath}\times f_{i,\alpha}^m.
  \]
  With respect to this identification, the map is the obvious
  projection to $f_{i,\alpha}^m$.
\item The rest of $\bdy f_k^{m+2}$ is mapped to the basepoint $e^0$.
\end{itemize}

\begin{figure}
  \centering
  \begin{overpic}[width=0.9\textwidth]{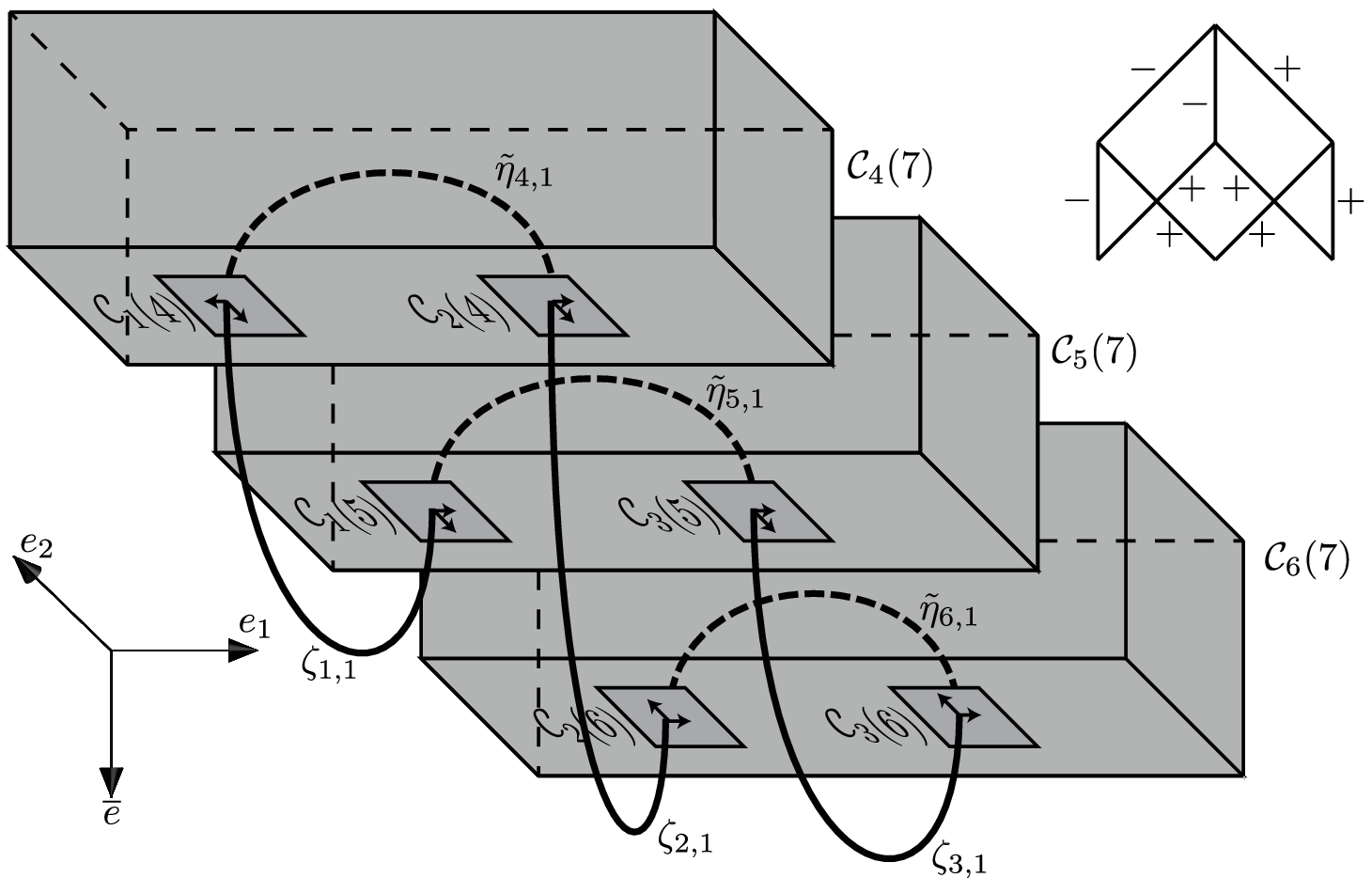}
    \put (88,63) {$\ob{z}_7$}
    \put (98,54) {$\ob{y}_6$}
    \put (89.5,54) {$\ob{y}_5$}
    \put (77,54) {$\ob{y}_4$}
    \put (77.5,44) {$\ob{x}_1$}
    \put (89.5,44) {$\ob{x}_2$}
    \put (98,44) {$\ob{x}_3$}
  \end{overpic}
  \caption{\textbf{The attaching map corresponding to a generator
      $\ob{z}_7$.} The boundary matching arcs are also shown;
    $\bmat_{4,1}$ is boundary-coherent while $\bmat_{5,1}$ and
    $\bmat_{6,1}$ are boundary-incoherent. Here, $m_1=m_2=1$ and we
    have identified 
    $
    \bigl([0,R]\times \{0\}\times [-R,R]\times[-R,R]\bigr)\cup 
    \bigl(\{0\}\times [0,R]\times [-R,R]\times[-R,R]\bigr)\cong 
    [-R,R]\times [-R,R]\times[-R,R]
    $
    via a reflection on the first summand 
    and the identity map on the second summand.}
  \label{fig:attach}
\end{figure}

\begin{prop}
  Let $X$ denote the space constructed above and
  $Y_{\ell}=Y_{\ell}(L)$ the CW complex from~\cite{RS-khovanov}
  associated to $L$ in quantum grading $\ell$. Then
  \[
  \Sigma^{N+\kappa-m} X = Y^{(N+\kappa+2)}_{\ell}/Y^{(N+\kappa-1)}_{\ell}.
  \]
\end{prop}
\begin{proof}
  The construction of $X$ above differs from the construction of $Y_{\ell}$
  in~\cite{RS-khovanov} as follows:
  \begin{enumerate}
  \item We have collapsed all cells of dimension less than $m$ to the
    basepoint, and ignored all cells of dimension bigger than $m+2$.
  \item In the construction above, we have suppressed the
    $0$-dimensional framed moduli spaces, instead speaking directly
    about the embeddings of cells that they induce. The
    $0$-dimensional moduli spaces correspond to the points $(a_1,
    \dots, a_{m_1})$ and $(b_1,\dots,b_{m_2})$ in
    Formulas~\eqref{eq:fj-fi} and~\eqref{eq:fk-fj},
    respectively. Their framings are induced by the maps in
    Formulas~\eqref{eq:fj-fi} and~\eqref{eq:fk-fj}.
  \item In the construction
    of~\cite[Definition~\ref*{KhSp:def:flow-gives-space}]{RS-khovanov},
    each of the cells above were multiplied by
    $[0,R]^{p_1}\times[-R,R]^{p_2}\times[-\epsilon,\epsilon]^{p_3}$
    for some $p_1,p_2,p_3\in\N$ with $p_1+p_2+p_3=N+\kappa-m$; and the
    various multiplicands were ordered differently. This has the
    effect of suspending the space $X$ by $(N+\kappa-m)$-many
    times.
  \item The framings in~\cite{RS-khovanov} were given by an
    obstruction-theory
    argument~\cite[Proposition~\ref*{KhSp:prop:cube-can-be-framed}]{RS-khovanov},
    while the framings here are given explicitly by the standard sign
    assignment and standard frame assignment. This is justified by
    \Lemma{extend-framing}.
  \end{enumerate}
  Thus, up to stabilizing, the two constructions give the same space.
\end{proof}

Therefore, it is enough to study the Steenrod square $\Sq^2\from
H^m(X;\F_2)\to H^{m+2}(X;\F_2)$. Fix a cohomology class $[c]\in
H^{m}(X;\F_2)$. Let
\[
c=\sum_{f_i^{m}} c_i(f_i^{m})^*
\]
be a cocycle representing $[c]$. Here, the $c_i$ are elements of
$\{0,1\}$. 

We want to understand the map $\cmap{c}\co X\to K_m^{(m+2)}$
corresponding to $c$. We start with $\cmap[(m)]{c}\co X^{(m)}\to
K_m^{(m)}$: on $f_i^m$, this map is defined as follows:
\begin{itemize}
\item the projection $\Pi$ composed with the identification
  $[-\epsilon,\epsilon]^{m}=\DD^{m}$ if $c_j=1$.
\item the constant map to the basepoint of $K_{m}^{(m)}$ if $c_j=0$.
\end{itemize}

To extend $\cmap{c}$ to $X^{(m+1)}$ we need to make one more
auxiliary choice:
\begin{definition}
  A \emph{topological boundary matching} for $c$ consists of the
  following data
 for each $(m+1)$-cell $f_j^{m+1}$: a collection of
  disjoint, embedded, framed arcs $\bma_{j,\jmath}$ in $f_j^{m+1}$
  connecting the points
  \[
  \coprod_{i\mid c_i=1} p_{i,j}\subset \bdy f_j^{m+1}
  \]
  in pairs, together with framings of the normal bundles to the
  $\bma_{j,\jmath}$.

  The normal bundle in $f_j^{m+1}$ to each of the points $p_{i,j}$
  inherits a framing from Formula~\eqref{eq:fj-fi}.  Call an arc
  $\bma_{j,\jmath}$ \emph{boundary-coherent} if the points $p_{i_1,j}$
  and $p_{i_2,j}$ in $\bdy\bma_{j,\jmath}$ have opposite framings, i.e.,
  if $\epsilon_{i_1,j}=-\epsilon_{i_2,j}$, and
  \emph{boundary-incoherent} otherwise. We require the following
  conditions on the framings for the $\bma_{j,\jmath}$:
  \begin{itemize}
  \item Trivialize $Tf_j^{m+1}$ using the following inclusion
    \begin{align*}
      \qquad f_j^{m+1}=[0,R]\times\{0\}\times[-R,R]^{m_1}\times[-\ep,\ep]^{m_2}&\hookrightarrow
      \RR^{m+1}\\
      (t,0,x_1,\dots,x_{m_1},y_1,\dots,y_{m_2})&\mapsto(-t,x_1,\dots,x_{m_1},y_1,\dots,y_{m_2}).
    \end{align*}
    We require the framing of $\bma_{j,\jmath}$ to be isotopic \rel
    boundary to one of the standard frame paths for $\R\times \R^{m}$.
  \item If $\bma_{j,\jmath}$ is boundary-coherent then the framing
    of $\bma_{j,\jmath}$ is compatible with the framing of its boundary.
  \item If $\bma_{j,\jmath}$ is boundary-incoherent then the framing of
    one end of $\bma_{j,\jmath}$ agrees with the framing of the
    corresponding $p_{i_1,j}$ while the framing of the other end of
    $\bma_{j,\jmath}$ differs from the framing of $p_{i_2,j}$ by the
    reflection $\rmap\co \RR^m\to\RR^m$.
  \end{itemize}
\end{definition}
Each boundary-incoherent arc in a topological boundary matching
inherits an orientation: it is oriented from the endpoint $p_{i,j}$ at which
the framings agree to the endpoint at which the framings disagree.

\begin{lemma}
  A topological boundary matching for $c$ exists.
\end{lemma}
\begin{proof}
  Since $c$ is a cocycle, for each $(m+1)$-cell $f_j^{m+1}$ we have
  \[
  \sum_i \epsilon_{i,j}c_i = c(\sum_i \epsilon_{i,j}f_i^m) = c(\bdy f_j^{m+1})\equiv
  0\pmod{2}.
  \]
  Together with our condition on $m$, this ensures that a topological
  boundary matching for $c$ exists.
\end{proof}

The map $\cmap[(m+1)]{c}\co X^{(m+1)}\to K_{m}^{(m+1)}$ is defined
using the topological boundary matching as follows. On $f_j^{m+1}$:
\begin{itemize}
\item The map $\cmap[(m+1)]{c}$ sends the complement of a neighborhood
  of the $\bma_{j,\jmath}$ to the basepoint.
\item If $\bma_{j,\jmath}$ is boundary-coherent then the framing of
  $\bma_{j,\jmath}$ identifies a neighborhood of the arc
  $\bma_{j,\jmath}$ with $\bma_{j,\jmath}\times \DD^m$. With respect to
  this identification, the map $\cmap[(m+1)]{c}$ is projection
  $\bma_{j,\jmath}\times\DD^m\to \DD^m\stackrel{\Pi}{\longrightarrow}
  K_m^{(m+1)}.$
  
  Note that $\cmap[(m)]{c}$ induces a map $(\bdy
  \bma_{j,\jmath})\times\DD^m$, and that the compatibility condition
  of the framing of $\bma_{j,\jmath}$ with the framing of $\bdy \bma_{j,\jmath}$
  implies that the map $\cmap[(m+1)]{c}$ extends the map
  $\cmap[(m)]{c}$.
\item If $\bma_{j,\jmath}$ is boundary-incoherent then the orientation
  and framing identify a neighborhood of $\bma_{j,\jmath}$ with
  $[0,1]\times\DD^m$. With respect to this identification,
  $\cmap[(m+1)]{c}$ is given by the map $\Xi$.
  
  Again, $\cmap[(m)]{c}$ induces a map $\{0,1\}\times\DD^m$, and 
  the compatibility condition of the framing of $\bma_{j,\jmath}$
  with the framing of $\bdy \bma_{j,\jmath}$ implies that the map
  $\cmap[(m+1)]{c}$ extends the map $\cmap[(m)]{c}$.
\end{itemize}

Now, fix an $(m+2)$-cell $f^{m+2}$. We want
to compute the element 
\[
\cmap{c}|_{\bdy f^{m+2}}\in \pi_{m+1}(K_{m}^{(m+1)})=\ZZ/2.
\]

As described above, 
\begin{align*}
\bdy f^{m+2}&= 
\bigl(\{0\}\times [0,R]\times [-R,R]^{m}\bigr)
\cup
\bigl(\{R\}\times [0,R]\times [-R,R]^{m}\bigr)\\
&\qquad{}\cup
\bigl([0,R]\times \{0\}\times [-R,R]^{m}\bigr)
\cup
\bigl([0,R]\times \{R\}\times [-R,R]^{m}\bigr)\\
&\qquad{}\cup
\bigl([0,R]\times [0,1]\times \bdy([-R,R]^m)\bigr)
\end{align*}
has corners.  The map $\cmap{c}|_{\bdy f^{m+2}}$ will send 
\[
\bigl(\{R\}\times[0,R]\times [-R,R]^{m}\bigr)\cup 
\bigl([0,R]\times \{R\}\times [-R,R]^{m}\bigr)\cup
\bigl([0,R]\times [0,1]\times \bdy([-R,R]^m)\bigr)
\]
to the basepoint. We straighten the corner between the other two parts
of $\bdy f^{m+2}$ via the map 
\begin{align*}
\bigl(\{0\}\times [0,R]\times [-R,R]^{m}\bigr)\cup
\bigl([0,R]\times \{0\}\times[-R,R]^{m}\bigr)
&\to [-R,R]\times [-R,R]^{m}\\
(0,t,x_1,\dots, x_{m_1},y_1,\dots,y_{m_2})&\mapsto (t, x_1,\dots,
x_{m_1},y_1,\dots, y_{m_2})\\
(t,0,x_1,\dots, x_{m_1},y_1,\dots,y_{m_2})&\mapsto (-t, x_1,\dots,
x_{m_1},y_1,\dots, y_{m_2}).
\end{align*}
We will suppress this straightening from the notation in the rest of
the section.

Let $\PT_1,\dots, \PT_k\subset S^{m+1}=\bdy f^{m+2}$ be the
Pontrjagin-Thom arcs corresponding to $f$. Let $\{\bmat_{j,\jmath}\}$ be the
preimages in $S^{m+1}=\bdy f^{m+2}$ of the topological boundary
matching. The union
\[
\bigcup_{j,\jmath} \bmat_{j,\jmath}\cup \bigcup_i \PT_i
\]
is a one-manifold in $S^{m+1}$. Each of the arcs $\PT_i\subset \bdy
f^{m+2}$ comes with a framing. Each of the arcs $\bmat_{j,\jmath}\subset
\bdy f^{m+2}$ also inherits a framing: the pushforward of the framing
of $\bma_{j,\jmath}$ under the map of Formula~\eqref{eq:fk-fj}. The map
$\cmap{c}|_{\bdy f^{m+2}}\co S^{m+1}\to K_m^{(m+1)}$ is induced from
these framed arcs as follows:
\begin{itemize}
\item A tubular neighborhood of each
  Pontrjagin-Thom arc $\PT_i$ is mapped via
  \[
  \nbd(\PT_i)\cong \PT_i\times\DD^m\to
  \DD^m\stackrel{\Pi}{\longrightarrow} S^m,
  \]
  where the first isomorphism is induced by the framing.
\item A tubular neighborhood of each boundary-coherent
  $\bmat_{j,\jmath}$ is mapped via
  \[
  \nbd(\bmat_{j,\jmath})\cong \bmat_{j,\jmath}\times\DD^m\to \DD^m\stackrel{\Pi}{\longrightarrow}S^m,
  \]
  where the first isomorphism is induced by the framing.
\item A tubular neighborhood of each boundary-incoherent
  $\bmat_{j,\jmath}$ is mapped via
  \[
  \nbd(\bmat_{j,\jmath})\cong [0,1]\times \DD^m\stackrel{\Xi}{\longrightarrow} K_m^{(m+1)}.
  \]
  where the first isomorphism is induced by the orientation and
  framing.
\item The map $\cmap{c}$ takes the rest of $S^{m+1}$ to the basepoint
  of $K_m^{(m+1)}$.
\end{itemize}

Let $K$ be a component of $\bigcup_i \PT_i\cup \bigcup_{j,\jmath}
\bmat_{j,\jmath}$. Relabeling, let $p_1,\dots, p_{2k}$ be the points
$p_{i,j}$ on $K$, $\bmat_1,\dots,\bmat_k$ the sub-arcs of $K$ coming
from the topological boundary matching and $\PT_1,\dots,\PT_k$ the
sub-arcs of $K$ coming from the Pontrjagin-Thom data. Order these so
that $\bdy \PT_i=\{p_{2i-1},p_{2i}\}$ and
$\bdy\bmat_i=\{p_{2i},p_{2i+1}\}$.

We define an isomorphism $\Phi\co \nbd(K)\to K\times \DD^m$ as
follows. First, the framing of $\PT_1$ induces an identification of
the normal bundle $N_{p_1}K$ with $\DD^m$. Second, the framing of each
arc $\gamma\in \{\PT_i, \bmat_i\}$ induces a trivialization of the
normal bundle $N\gamma$. Suppose that the framing of $K$ has already
been defined at the endpoint $p_i$ of $\gamma$. Then the
trivialization of $N\gamma$ allows us to transport the framing of
$p_i$ along $\gamma$. This transported framing is the framing of $K$
along $\gamma$.

Note that the framing of $K$ along $\gamma$ may not agree with the
original framing of $\gamma$; but the two either agree or differ by
the map 
\[
K\times\DD^m\to K\times\DD^m\quad (x,v)\mapsto (x,\rmap(v)),
\]
depending on the parity of the number of boundary-incoherent arcs
traversed from $p_1$.
In particular, it is not \emph{a priori} obvious that the framing we
have defined is continuous at $p_1$; but this follows from the
following lemma:

\begin{lemma}\label{lem:even-cyc-top} 
  An even number of the arcs in $K$ are boundary-incoherent.
\end{lemma}
\begin{proof}
  The proof is essentially the same as the proof of the second half of
  \Lemma{even-cyc}, and is left to the reader.
\end{proof}

Call an arc $\gamma$ in $K$ \emph{$\rmap$-colored} if the original framing
of $\gamma$ disagrees with the framing of $K$, and
\emph{$\Id$-colored} if the original framing of $\gamma$ agrees with
the framing of $K$.
Write $\Psi=\cmap{c}\circ \Phi^{-1}\co K\times \DD^m\to
K_m^{(m+1)}$. Explicitly, the map $\Psi$ is given as follows:
\begin{itemize}
\item If $\gamma_i$ is one of the Pontrjagin-Thom arcs or is a
  boundary-coherent topological boundary matching arc then a
  neighborhood of $\gamma_i$ in $K\times\DD^m$ is mapped to
  $S^m\subset K_m^{(m+1)}$ by the map
  \begin{align*}
    \gamma_i\times\DD^m\ni (x,v)&\mapsto \Pi(v)\in K_m^{(m+1)} &
    &\text{if $\gamma_i$ is $\Id$-colored}\\
    \gamma_i\times\DD^m\ni (x,v)&\mapsto \Pi(\rmap(v)) \in K_m^{(m+1)} &
    &\text{if $\gamma_i$ is $\rmap$-colored}.
  \end{align*}
\item If $\bmat_i$ is boundary-incoherent then the framing of $K$
  and the orientation of $\bmat_i$ induce an identification
  $\nbd(\bmat_i)\cong [0,1]\times\DD^m$. With respect to this
  identification, $\nbd(\bmat_i)$ is mapped to $K_m^{(m+1)}$ by the map
  \begin{align*}
    (t,v)&\mapsto \Xi(t,v) &
    &\text{if $\bmat_i$ is $\Id$-colored}\\
    (t,v)&\mapsto \Xi(t,\rmap(v)) &
    &\text{if $\bmat_i$ is $\rmap$-colored}.
  \end{align*}
\end{itemize}

Let $\Psi'$ be the projection $K\times\DD^m\to
\DD^m\stackrel{\Pi}{\longrightarrow} S^m$.
These maps are summarized in the following diagram:
\begin{equation}\label{eq:nbd-K-maps}
\xymatrix{
\nbd(K)\ar[r]^{\Phi}_\cong\ar@/^2pc/[rr]^{\cmap{c}} & K\times \DD^m\ar[r]^\Psi\ar[dr]_{\Psi'} & K_m^{(m+1)}\\
& & S^m\ar@{^(->}[u]_\iota
}
\end{equation}
It is immediate from the definitions that the top triangle commutes. Our
next goal is to show that the other triangle commutes up to homotopy:
\begin{prop}\label{prop:PsiPsiprime}
  The map $\Psi$ is homotopic \rel $(K\times\bdy \DD^m)\cup
  (\{p_1\}\times\DD^m)$ to $\iota\circ\Psi'$, i.e., the bottom
  triangle of Diagram~(\ref{eq:nbd-K-maps}) commutes up to homotopy
  \rel $(K\times\bdy \DD^m)\cup (\{p_1\}\times\DD^m)$.
\end{prop}

The proof of \Proposition{PsiPsiprime} uses a model
computation. Consider the map $\overline{\Xi}\co [0,1]\times S^m\to
K_m^{(m+1)}$.  Concatenation in $[0,1]$ endows $\Hom([0,1]\times S^m,
K_m^{(m+1)})$ with a multiplication, which we denote by $*$. Let
$\mf{t}\from[0,1]\to[0,1]$ be the reflection
$\overline{f}(t,x)=f(1-t,x)$. Using $\mf{t}$, we obtain a map
$\overline{\Xi}\circ (\mf{t}\times \Id)\co [0,1]\times S^m\to
K_m^{(m+1)}$. Finally, using the map $\rmap$ we obtain a
map $\overline{\Xi}\circ (\Id\times \rmap)\co [0,1]\times S^m\to
K_m^{(m+1)}$.

\begin{lemma}\label{lem:frame-cancel}
  Assume that $m\geq 3$. Then both
  $\ol{\Xi}*[\overline{\Xi}\circ(\mf{t}\times \Id)]$ and $\ol{\Xi}*
  [\ol{\Xi}\circ (\Id\times \rmap)]$ are homotopic (\rel boundary) to the map
  $[0,1]\times S^m\to S^m \subset K_m^{(m+1)}$ given by $(t,x)\mapsto
  x$ (i.e., the constant path in $O(m+1)$ with value $\Id$).
\end{lemma}
\begin{proof}
  The statement about $\ol{\Xi}*[\overline{\Xi}\circ(\mf{t}\times
  \Id)]$ is obvious.  For the statement about $\ol{\Xi}*
  [\ol{\Xi}\circ (\Id\times \rmap)]$, let $H_m=\ol{\Xi}*
  [\ol{\Xi}\circ (\Id\times \rmap)]\co [0,1]\times S^m\to
  K_m^{(m+1)}$. We can view $H_m$ as an element of $\pi_1(\Omega^m
  K_m^{(m+1)})\cong \pi_{m+1}(K_m^{(m+1)})$.  Moreover, the map $H_m$
  is the $(m-1)$-fold suspension of the map $H_1\co [0,1]\times S^1\to
  K_1^{(2)}=\RR P^2$. But the suspension map
  \[
  \Sigma^i \co \pi_2(\RR P^2)\to \pi_{i+2}(\Sigma^i\RR P^2)
  \]
  is nullhomotopic for $i\geq 2$; see, for instance,~\cite[Proposition
  6.5 and discussion before Proposition 6.11]{Wu-top-proj-plane}. So,
  it follows from our assumption on $m$ that $H_1\in \pi_1(\Omega^m
  K_m^{(m+1)})$ is homotopically trivial.  Keeping in mind that our
  loops are based at the identity map $S^m\to S^m\subset K_m^{(m+1)}$,
  this proves the result.
\end{proof}

\begin{proof}[Proof of \Proposition{PsiPsiprime}]
  As in the proof of \Lemma{frame-cancel}, we can view $\Psi$ as an element
  of $\pi_1(\Omega^m K_m^{(m+1)}, \iota)$, i.e., a loop of maps
  $S^m\to K_m^{(m+1)}$ based at the map $\iota\co S^m\to
  K_m^{(m+1)}$. From its definition, $\Psi$ decomposes as a product of
  paths,
  \[
  \Psi=\Psi_{\gamma_1}*\dots*\Psi_{\gamma_{2k}},
  \]
  one for each arc $\gamma_i$ in $K$. Here, $\Psi_{\gamma_i}$ is an
  element of the fundamental groupoid of $\Omega^m K_m^{(m+1)}$, with
  endpoints in $\{\iota, \iota\circ \rmap\}$. The path $\Psi_{\gamma_i}$
  is:
  \begin{itemize}
  \item The constant path based at either $\iota$ or $\iota\circ \rmap$ if
    $\gamma_i$ is one of the Pontrjagin-Thom arcs or is a
    boundary-coherent topological boundary matching arc.
  \item One of the paths $\Xi$, $\Xi\circ (\Id\times \rmap)$,
    $\Xi\circ(\mf{t}\times\Id)$ or $\Xi\circ(\mf{t}\times\Id)\circ
    (\Id\times \rmap)$ if $\gamma_i$ is a boundary-incoherent
    topological boundary matching arc.
  \end{itemize}
  Contracting the constant paths, $\Psi$ can be expressed as
  \[
  \Psi = \Psi_{\eta_{i_1}}*\Psi_{\eta_{i_2}}*\dots*\Psi_{\eta_{i_A}}
  \]
  where the $\eta_{i_{\al}}$ are boundary-incoherent. By
  \Lemma{even-cyc-top}, $A$ is even. Moreover, $\Psi_{\eta_{i_{\al}}}$
  is either $\Xi$ or $\Xi\circ(\mf{t}\times\Id)\circ (\Id\times
  \rmap)$ if $\al$ is odd, and is either $\Xi\circ(\mf{t}\times\Id)$ or $\Xi\circ
  (\Id\times \rmap)$ if $\al$ is even. So, by \Lemma{frame-cancel},
  the concatenation $\Psi_{\eta_{i_{2\al-1}}}*\Psi_{\eta_{i_{2\al}}}$
  is homotopic to the constant path $\iota$. The result follows.
\end{proof}

The pair $(K,\Phi)$ specifies a framed cobordism class $[K,\Phi]\in
\Omega_1^{\fr}=\ZZ/2$.
\begin{prop}\label{prop:KPhi}
  The element $[K,\Phi]\in\ZZ/2$ is given by the sum of:
  \begin{enumerate}
  \item\label{item:one} $1$.
  \item\label{item:framing} The number of Pontrjagin-Thom arcs in $K$
    with the non-standard framing.
  \item\label{item:arrows} The number of arrows on $K$ which point in
    a given direction.
  \end{enumerate}
\end{prop}
\begin{proof}
  First, exchanging the standard and non-standard framings on an arc
  changes the overall framing of $K$ by $1$. So, it suffices to prove
  the proposition in the case that all of the Pontrjagin-Thom arcs in
  $K$ have the standard framing.  

  Second, the framing on each boundary-matching arc is standard if the
  corresponding $(m+1)$-cell occurs positively in $\bdy f^{m+2}$, and
  differs from the standard framing by the map $\smap$
  of \Subsection{frames-R3} if the $(m+1)$-cell occurs negatively in
  $\bdy f^{m+2}$. In the notation
  of \Subsection{frames-R3}, the framings of the boundary matching
  arcs are among $\bigl\{\ol{\topright\bottomright},
  \ol{\topleft\bottomleft}, \ol{\bottomright\topright},
  \ol{\bottomleft\topleft}\bigr\}$. So, by \Lemma{r-of-framing}, $\smap$
  takes the standard frame path on a boundary-matching arc to a
  standard frame path. In sum, each of the arcs $\bmat_i$ is framed by
  a standard frame path.

  So, by \Lemma{r-of-framing}, the framing of $K$ at each arc
  $\gamma_i$ is standard if $\gamma_i$ is $\Id$-colored, and
  non-standard if $\gamma_i$ is $\rmap$-colored. Thus, it suffices to show
  that the number of $\rmap$-colored arcs agrees modulo $2$ with the
  number of arrows on $K$ which point in a given direction. 

  Let $\bmat_{i_1},\dots, \bmat_{i_A}$ be the boundary-incoherent
  boundary matching arcs in $K$. Then:
  \begin{itemize}
  \item There are an odd number of arcs strictly between $\bmat_{i_{\al}}$ and
    $\bmat_{i_{\al+1}}$. Moreover:
    \begin{itemize}
    \item These arcs are all $\rmap$-colored if $\al$ is odd.
    \item These arcs are all $\Id$-colored if $\al$ is even.
    \end{itemize}
  \item If $\bmat_{i_\al}$ and $\bmat_{i_{\al+1}}$ are oriented in the same direction
    then exactly one of $\bmat_{i_\al}$ and $\bmat_{i_{\al+1}}$ is $\rmap$-colored.
  \item If $\bmat_{i_\al}$ and $\bmat_{i_{\al+1}}$ are oriented in
    opposite directions then either both of $\bmat_{i_\al}$ and
    $\bmat_{i_{\al+1}}$ are $\rmap$-colored or both of $\bmat_{i_\al}$ and
    $\bmat_{i_{\al+1}}$ are $\Id$-colored.
  \end{itemize}
  It follows that there are an even (respectively odd) number of
  $\rmap$-colored arcs in the interval $[\bmat_{i_{2\al-1}},\bmat_{i_{2\al}}]$
  if $\bmat_{i_{2\al-1}}$ and $\bmat_{i_{2\al}}$ are oriented in the same
  direction (respectively opposite directions); and all of the arcs in
  the interval $(\bmat_{i_{2\al}},\bmat_{i_{2\al+1}})$ are
  $\Id$-colored. So, the number of $\rmap$-colored arcs agrees with the
  number of arcs which point in a given direction.

  Finally, the contribution $1$ comes from the fact that the constant
  loop in $\SO(m)$ corresponds to the nontrivial element of $\Omega_1^{\fr}$.
\end{proof}

\begin{proof}[Proof of \Theorem{sq2-agrees}]
  Let $\Phi_K$ and $\Psi_K$ be the maps associated above to each
  component $K$ of $\bigcup_i \PT_i\cup \bigcup_{j,\jmath}
  \bmat_{j,\jmath}$. Then $\Psi_K\circ \Phi_K$ induces an element
  $[\Psi_K\circ \Phi_K]$ of $\pi_{m+1}(K_m^{m+1})\cong\ZZ/2$ (by collapsing
  everything outside a neighborhood of $K$ to the basepoint). Let $\ob{z}$ be
  the generator of the Khovanov complex corresponding to the cell
  $f^{m+2}$ and $\ob{c}$ the element of the Khovanov chain group
  corresponding to the cocycle $c$.  Then it suffices to show that the
  sum
  \[
  \sum_K [\Psi_K\circ \Phi_K]
  \]
  agrees with the expression
  \begin{equation}\label{eq:sq2-term}
  \biggl(\#|\Graph[\ob{c}]{\ob{z}}| +
  \framenumber{\Graph[\ob{c}]{\ob{z}}}+\graphnumber{\Graph[\ob{c}]{\ob{z}}}\biggr)\ob{z},
  \end{equation}
  from Formula~\eqref{eq:sq2}. 

  By \Proposition{PsiPsiprime}, 
  \[
  \xymatrix@R=0ex@C=1ex{
    [K,\Phi] \ar@{|->}[rrrr] &&&&[\Psi_K\circ\Phi_K]\\
\rotatebox[origin=c]{270}{$\in$}&&&&\rotatebox[origin=c]{270}{$\in$}\\
    \Omega_1^{\fr} \ar@{=}[r]& \pi_{m+1}(S^m) \ar[rrr]_-{\cong}^-{\iota_*} &&& \pi_{m+1}(K_m^{(m+1)}).
  }
  \]
  The
  element $[K,\Phi_K]\in\Z/2$ is computed in \Proposition{KPhi}, and it remains
  to match the terms in that proposition with the terms in
  Formula~\eqref{eq:sq2-term}.

  By construction, the graph $\Graph[\ob{c}]{\ob{z}}$ is exactly
  $\bigcup_i \PT_i\cup \bigcup_{j,\jmath} \bmat_{j,\jmath}$, and the
  orientations of the oriented edges of $\Graph[\ob{c}]{\ob{z}}$ match
  up with the orientations of the boundary-incoherent $\bmat_i$.  So,
  the first term in Formula~\eqref{eq:sq2-term} corresponds to
  part~(\ref{item:one}) of \Proposition{KPhi}, and the third term in
  Formula~\eqref{eq:sq2-term} corresponds to part~(\ref{item:arrows})
  of \Proposition{KPhi}. Finally, since the framings of the Pontrjagin-Thom
  arcs differ from the standard frame paths by the standard frame
  assignment $f$, the second term of Formula~\eqref{eq:sq2-term}
  corresponds to part~(\ref{item:framing}) of \Proposition{KPhi}. This
  completes the proof.
\end{proof}

\section{The Khovanov homotopy type of width three knots}\label{sec:width-two}
It is immediate from Whitehead's theorem that if $\widetilde{H}_i(X)$
is trivial for all $i\neq m$ then $X$ is a Moore space; in particular,
the homotopy type of $X$ is determined by the homology of $X$ in this
case. This result can be extended to spaces with nontrivial homology
in several gradings, if one also keeps track of the action of the
Steenrod algebra. To determine the Khovanov homotopy types of links up
to $11$ crossings we will use such an extension due to
Whitehead~\cite{Whitehead-top-exact} and
Chang~\cite{Chang-top-homotopy}, which we review here. (For further
discussion along these lines, as well as the next larger case,
see~\cite[Section 11]{Baues-top-handbook}.)

\begin{proposition}\label{prop:characterise-quiver}
  Consider quivers of the form
  \[
  \xymatrix{
    A\ar[r]_{f}\ar@/^{1.5pc}/[rr]^{s} & B \ar[r]_{g}& C
  }
  \]
  where $A$, $B$ and $C$ are $\F_2$-vector spaces, and $gf=0$. Such a
  quiver uniquely decomposes as a direct sum of the following quivers:
  \[ 
  \xymatrix@C=1ex{
    \text{(S-1)}& \F_2&& 0 && 0&&
    \text{(S-2)}& 0 && \F_2 && 0&&
    \text{(S-3)}& 0 && 0 && \F_2\\
    \text{(P-1)}& \F_2\ar[rr]_-{\Id} && \F_2 && 0&&
    \text{(P-2)}& 0 && \F_2 \ar[rr]_-{\Id}&& \F_2&&
    \text{(X-1)}& \F_2\ar@/^{1.5pc}/[rrrr]^{\Id} && 0&& \F_2\\ 
    \text{(X-2)}& \F_2\ar[rr]_-{\Id}\ar@/^{1.5pc}/[rrrr]^{\Id} && \F_2 && \F_2&&
    \text{(X-3)}& \F_2\ar@/^{1.5pc}/[rrrr]^{\Id} && \F_2 \ar[rr]_-{\Id}&& \F_2&&
    \text{(X-4)}& \F_2\ar[rr]_-{\left(\begin{smallmatrix}\Id\\0\end{smallmatrix}\right)}\ar@/^{1.5pc}/[rrrr]^{\Id} && \F_2\oplus\F_2 \ar[rr]_-{(0,\Id)}&& \F_2.
  }
  \]
\end{proposition}

\begin{proof}
  We start with uniqueness. In such a decomposition, let $s_i$ be the
  number of (S-i) summands, $p_i$ be the number of (P-i) summands, and
  $x_i$ be the number of (X-i) summands. Consider the following nine
  pieces of data:
  \begin{itemize}
  \item The dimensions of the $\F_2$-vector spaces $A$, $B$ and $C$,
    say $d_1,d_2,d_3$, respectively;
  \item The ranks of the maps $f$ and $g$, say $r_f,r_g$, respectively;
  \item The dimensions of the $\F_2$-vector spaces 
    $\image(s)$,
    $\image(\restrict{s}{\ker(f)})$,
    $\image(g)\cap\image(s)$ and
    $\image(g)\cap\image(\restrict{s}{\ker(f)})$, say
    $r_1,r_2,r_3,r_4$, respectively.
  \end{itemize}
  We have
  \begin{align}
    \begin{pmatrix}
      d_1\\
      d_2\\
      d_3\\
      r_f\\
      r_g\\
      r_1\\
      r_2\\
      r_3\\
      r_4
    \end{pmatrix}
    &=
    \begin{pmatrix}
     1&.&.&1&.&1&1&1&1\\
     .&1&.&1&1&.&1&1&2\\
     .&.&1&.&1&1&1&1&1\\
     .&.&.&1&.&.&1&.&1\\
     .&.&.&.&1&.&.&1&1\\
     .&.&.&.&.&1&1&1&1\\
     .&.&.&.&.&1&.&1&.\\
     .&.&.&.&.&.&.&1&1\\
     .&.&.&.&.&.&.&1&.
    \end{pmatrix}
    \begin{pmatrix}
      s_1\\
      s_2\\
      s_3\\
      p_1\\
      p_2\\
      x_1\\
      x_2\\
      x_3\\
      x_4
    \end{pmatrix}
    \notag\\
    \shortintertext{and therefore, the numbers $s_i$, $p_i$ and $x_i$ are determined as
      follows:}
    \begin{pmatrix}
      s_1\\
      s_2\\
      s_3\\
      p_1\\
      p_2\\
      x_1\\
      x_2\\
      x_3\\
      x_4
    \end{pmatrix}
    &=
    \begin{pmatrix*}[r]
  1& .& .&-1& .& .&-1& .& .\\
  .& 1& .&-1&-1& .& .& .& .\\
  .& .& 1& .&-1&-1& .& 1& .\\
  .& .& .& 1& .&-1& 1& .& .\\
  .& .& .& .& 1& .& .&-1& .\\
  .& .& .& .& .& .& 1& .&-1\\
  .& .& .& .& .& 1&-1&-1& 1\\
  .& .& .& .& .& .& .& .& 1\\
  .& .& .& .& .& .& .& 1&-1\\
    \end{pmatrix*}
    \begin{pmatrix}
      d_1\\
      d_2\\
      d_3\\
      r_f\\
      r_g\\
      r_1\\
      r_2\\
      r_3\\
      r_4
    \end{pmatrix}
    \label{eq:coeff-of-decomposition}.
  \end{align}

  For existence of such a decomposition, we carry out a standard
  change-of-basis argument. Choose generators for $A$, $B$ and $C$,
  and construct the following graph. There are three types of
  vertices, $A$-vertices, $B$-vertices and $C$-vertices, corresponding
  to generators of $A$, $B$ and $C$ respectively. There are three
  types of edges, $f$-edges, $g$-edges and $s$-edges, corresponding to
  the maps $f$, $g$ and $s$ as follows: for $a$ an $A$-vertex and $b$
  a $B$-vertex, if $b$ appears in $f(a)$ then there is an $f$-edge
  joining $a$ and $b$; the $g$-edges and $s$-edges are defined
  similarly.

  We will do a change of basis, which will change the graph, so that
  in the final graph, each vertex is incident to at most one edge of
  each type. This will produce the required decomposition of the
  quiver.

  We carry out the change of basis in the following sequence of
  steps. Each step accomplishes a specific simplification of the
  graph; it can be checked that the later steps do not undo the
  earlier simplifications.

  \begin{enumerate}
  \item\label{step:simplify-1} We ensure that no two $f$-edges share a
    common vertex. Fix an
    $f$-edge joining an $A$-vertex $a$ to a $B$-vertex $b$. Let
    $\{a_i\}$ be the other $A$-vertices that are $f$-adjacent to $b$
    and $\{b_j\}$ be the other $B$-vertices that are $f$-adjacent to
    $a$. Then change basis by replacing each $a_i$ by $a_i+a$, and by
    replacing $b$ with $b+\sum_j b_j$.
  \item By the same procedure as Step~(\ref{step:simplify-1}), we ensure
    that no two $g$-edges share a common
    vertex. 

    Since $gf=0$, this ensures that no $B$-vertex is adjacent
    to both an $f$-edge and a $g$-edge.  Call an $A$-vertex an
    $A_1$-vertex (\respectively $A_2$-vertex) if it is adjacent
    (\respectively non-adjacent) to an $f$-edge; similarly, call a
    $C$-vertex a $C_1$-vertex (\respectively $C_2$-vertex) if it is
    adjacent (\respectively non-adjacent) to a $g$-vertex.
  \item Next, we isolate the $s$-edges that connect $A_2$-vertices to
    $C_2$-vertices. Fix an $s$-edge joining an $A_2$-vertex $a$ to a
    $C_2$-vertex $c$. If $\{a_i\}$ (\respectively $\{c_j\}$) are the
    other $A$-vertices (\respectively $C$-vertices) that are
    $s$-adjacent to $c$ (\respectively $a$), then change basis by
    replacing each $a_i$ by $a_i+a$ and by replacing $c$ with
    $c+\sum_j c_j$.
  \item The next step is to isolate the $s$-edges that connect
    $A_1$-vertices to $C_2$-vertices. Once again, fix an $s$-edge
    joining an $A_1$-vertex $a$ to a $C_2$-vertex $c$. Let $\{a_i\}$
    (\respectively $\{c_j\}$) be the other $A$-vertices
    (\respectively $C$-vertices) that are $s$-adjacent to $c$
    (\respectively $a$). Let $b_i$ be the $B$-vertex that is
    $f$-adjacent to $a_i$ (observe, each $a_i$ is an $A_1$-vertex),
    and let $b$ be the $B$-vertex that is $f$-adjacent to $a$. Then
    change basis by replacing each $a_i$ by $a_i+a$, by replacing each
    $b_i$ by $b_i+b$ and by replacing $c$ with $c+\sum_j c_j$.
  \item Similarly, we can isolate the $s$-edges that connect
    $A_2$-vertices to $C_1$-vertices. As before, fix an $s$-edge
    joining an $A_2$-vertex $a$ to a $C_1$-vertex $c$. Let $\{a_i\}$
    (\respectively $\{c_j\}$) be the other $A$-vertices (\respectively
    $C$-vertices) that are $s$-adjacent to $c$ (\respectively
    $a$). Let $b_j$ be the $B$-vertex that is $g$-adjacent to $c_j$
    and let $b$ be the $B$-vertex that is $g$-adjacent to $c$. Then
    change basis by replacing each $a_i$ by $a_i+a$, by replacing $b$
    with $b+\sum_j b_j$ and by replacing $c$ with $c+\sum_j c_j$.
  \item Finally, we have to isolate the $s$-edges that connect
    $A_1$-vertices to $C_1$-vertices. This can be accomplished by a
    combination of the previous two steps. \qedhere
  \end{enumerate}
\end{proof}

We are interested in stable spaces $X$ satisfying the following
conditions:
\begin{itemize}
\item The only torsion in $H^*(X;\Z)$ is $2$-torsion, and
\item $\wt{H}^i(X;\F_2)=0$ if $i\neq 0,1,2$.
\end{itemize}
Then the quiver
\[
\xymatrix{
  \wt{H}^0(X;\F_2)\ar[r]_{\Sq^1}\ar@/^2pc/[rr]^{\Sq^2} & \wt{H}^{1}(X;\F_2) \ar[r]_{\Sq^1}& \wt{H}^{2}(X;\F_2)
}
\]
is of the form described in \Proposition{characterise-quiver}. In
Examples~\ref{exam:first-subthin}--\ref{exam:last-subthin}, we will
describe nine such spaces whose associated quivers are the nine
irreducible ones of \Proposition{characterise-quiver}.

\begin{example}\label{exam:first-subthin}
The associated quivers of $S^0$, $S^1$ and $S^2$ are (S-1), (S-2) and
(S-3), respectively. The associated quivers of $\Sigma^{-1}\RP^2$ and
$\RP^2$ are (P-1) and (P-2), respectively.
\end{example}

\begin{example}
  The space $\CP^2$ has cohomology
  \[
  \xymatrix{
    \wt{H}^4(\CP^2;\ZZ) & \ZZ &\qquad & \wt{H}^4(\CP^2;\F_2) & \F_2\\
    \wt{H}^3(\CP^2;\ZZ) & 0 & & \wt{H}^3(\CP^2;\F_2) & 0\\
    \wt{H}^2(\CP^2;\ZZ) & \ZZ & & \wt{H}^2(\CP^2;\F_2) & \F_2\ar@/_2pc/[uu]_{\Sq^2}.
  }    
  \]
  (The fact that $\Sq^2$ has this form follows from the fact that
  for $x\in H^n$, $\Sq^n(x)=x\cup x$.) Therefore, the stable space
  $X_1\defeq\Sigma^{-2}\CP^2$ has (X-1) as its associated quiver.
\end{example}

\begin{example}
  The space $\RP^5/\RP^2$ has cohomology
  \[
  \xymatrix{
    \wt{H}^5(\RP^5/\RP^2;\ZZ) & \ZZ &\qquad & \wt{H}^5(\RP^5/\RP^2;\F_2) & \F_2\\
    \wt{H}^4(\RP^5/\RP^2;\ZZ) & \F_2 & & \wt{H}^4(\RP^5/\RP^2;\F_2) & \F_2\\
    \wt{H}^3(\RP^5/\RP^2;\ZZ) & 0 & & \wt{H}^3(\RP^5/\RP^2;\F_2) & \F_2\ar@/_2pc/[uu]_{\Sq^2}\ar[u]^{\Sq^1}.
  }  
  \]
  To see that $\Sq^2$ has the stated form, consider the inclusion map
  $\RP^5/\RP^2\to \RP^6/\RP^2$. The map $\Sq^3\co H^3(\RP^6/\RP^2)\to
  H^6(\RP^6/\RP^2)$ is an isomorphism (since it is just the cup
  square). By the Adem relations, $\Sq^3=\Sq^1\Sq^2$, so $\Sq^2\co
  H^3(\RP^6/\RP^2)\to H^5(\RP^6/\RP^2)$ is nontrivial. So, the
  corresponding statement for $\RP^5/\RP^2$ follows from
  naturality. Therefore, the stable space
  $X_2\defeq\Sigma^{-3}(\RP^5/\RP^2)$ has (X-2) as its associated quiver.
\end{example}

\begin{example}
  The space $\RP^4/\RP^1$ has cohomology
  \[
  \xymatrix{
    \wt{H}^4(\RP^4/\RP^1;\ZZ) & \F_2 &\qquad & \wt{H}^4(\RP^4/\RP^1;\F_2) & \F_2\\
    \wt{H}^3(\RP^4/\RP^1;\ZZ) & 0 & & \wt{H}^3(\RP^4/\RP^1;\F_2) & \F_2\ar[u]^{\Sq^1}\\
    \wt{H}^2(\RP^4/\RP^1;\ZZ) & \ZZ & & \wt{H}^2(\RP^4/\RP^1;\F_2) & \F_2\ar@/_2pc/[uu]_{\Sq^2}.
  }
  \]
  (The answer for $\Sq^2$ again follows from the fact that it is the
  cup square.) Therefore, the stable space
  $X_3\defeq\Sigma^{-2}(\RP^4/\RP^1)$ has (X-3) as its associated quiver.
\end{example}

\begin{example}\label{exam:last-subthin}
  The space $\RP^2\wedge\RP^2$ has cohomology
  \[
  \xymatrix{
    \wt{H}^4(\RP^2\wedge\RP^2;\ZZ) & \F_2 &\qquad & \wt{H}^4(\RP^2\wedge\RP^2;\F_2) & \F_2\ar@{<-}[d]!R(.5)^{\Sq^1}\ar@{<-}[d]!L(.5)_{\Sq^1}\\
    \wt{H}^3(\RP^2\wedge\RP^2;\ZZ) & \F_2 & & \wt{H}^3(\RP^2\wedge\RP^2;\F_2) & \F_2\oplus\F_2\\
    \wt{H}^2(\RP^2\wedge\RP^2;\ZZ) & 0 & & \wt{H}^2(\RP^2\wedge\RP^2;\F_2) & \F_2\ar[u]!R(.5)_{\Sq^1}\ar[u]!L(.5)^{\Sq^1}\ar@/_3pc/[uu]_{\Sq^2}.
  }  
  \]
  (The answer for $\Sq^2$ follows from the product formula:
  $\Sq^2(a\wedge
  b)=a\wedge\Sq^2(b)+\Sq^1(a)\wedge\Sq^1(b)+\Sq^2(a)\wedge b$.)
  Therefore, the quiver associated to the stable space
  $X_4\defeq\Sigma^{-2}(\RP^2\wedge\RP^2)$ is isomorphic to (X-4).
\end{example}

The following is a classification theorem from \cite[Theorems 11.2 and
11.7]{Baues-top-handbook}.
\begin{proposition}\label{prop:space-classify}
  Let $X$ be a simply connected CW complex such that:
  \begin{itemize}
  \item The only torsion in the cohomology of $X$ is $2$-torsion.
  \item There exists $m$ sufficiently large so that the reduced
    cohomology $\widetilde{H}^i(X;\F_2)$ is trivial for $i\neq m,
    m+1,m+2$.
  \end{itemize}
  Then the homotopy type of $X$ is determined by the isomorphism class of the quiver
  \[
    \xymatrix{
      H^m(X;\F_2)\ar[r]_{\Sq^1}\ar@/^2pc/[rr]^{\Sq^2} & H^{m+1}(X;\F_2) \ar[r]_{\Sq^1}& H^{m+2}(X;\F_2)
    }
  \]
  as follows: Decompose the quiver as in
  \Proposition{characterise-quiver}; let $s_i$ be the number of (S-i)
  summands, $1\leq i\leq 3$; let $p_i$ be the number of (P-i)
  summands, $1\leq i\leq 2$; and let $x_i$ be the number of (X-i)
  summands, $1\leq i\leq 4$. Then $X$ is homotopy equivalent to 
  \[
  Y\defeq  (\bigvee_{i=1}^3\bigvee_{j=1}^{s_i}S^{m+i-1})\vee
  (\bigvee_{i=1}^2\bigvee_{j=1}^{p_i}\Sigma^{m+i-2}\RP^2)\vee(\bigvee_{i=1}^4\bigvee_{j=1}^{x_i}\Sigma^mX_i).
  \]
\end{proposition}

In light of \Corollary{width-3-determined}, the following seems a
natural link invariant.
\begin{defn}\label{def:st}
  For any link $L$, the function $\St(L)\from\Z^2\to\N^4$ is defined
  as follows: Fix $(i,j)\in\Z^2$; for $k\in\{i,i+1\}$, let
  $\Sq^1_{(k)}$ denote the map
  $\Sq^1\from\Kh^{k,j}(L)\to\Kh^{k+1,j}(L)$.
  
  Let $r_1$ be the rank of the map $\Sq^2\from
  \Kh^{i,j}(L)\to\Kh^{i+2,j}(L)$; let $r_2$ be the rank of the map
  $\restrict{\Sq^2}{\ker \Sq^1_{(i)}}$; let $r_3$ be the dimension of
  the $\F_2$-vector space $\image\Sq^1_{(i+1)}\cap\image\Sq^2$; and
  let $r_4$ be the dimension of the $\F_2$-vector space
  $\image\Sq^1_{(i+1)}\cap\image(\restrict{\Sq^2}{\ker\Sq^1_{(i)}})$.
  Then,
  \[
  \St(i,j)\defeq(r_2-r_4,r_1-r_2-r_3+r_4,r_4,r_3-r_4).
  \]
\end{defn}

\begin{corollary}\label{cor:width-3-determined}
  Suppose that the Khovanov homology $\Kh_{\Z}(L)$ satisfies the following
  properties:
  \begin{itemize}
  \item $\Kh^{i,j}_{\Z}(L)$ lies on three adjacent diagonals, say $2i-j=\si,\si+2,\si+4$.
  \item There is no torsion other than $2$-torsion.
  \item There is no torsion on the diagonal $2i-j=\si$.
  \end{itemize}
  Then the homotopy types of the stable spaces $\KhSpace^j(L)$ are
  determined by $\Kh_{\Z}(L)$ and $\St(L)$ as follows: Fix $j$; let
  $i=\frac{j+\si}{2}$; let $\St(i,j)=(x_1,x_2,x_3,x_4)$; then
  $\KhSpace^j(L)$ is stably homotopy equivalent to the wedge sum of
  \[
  (\bigvee^{x_1}\Si^{i-2}\CP^2)\vee(\bigvee^{x_2}\Si^{i-3}(\RP^5/\RP^2))\vee
  (\bigvee^{x_3}\Si^{i-2}(\RP^4/\RP^1))\vee(\bigvee^{x_4}\Si^{i-2}(\RP^2\wedge\RP^2))
  \]
  and a wedge of Moore spaces. In particular, $\KhSpace^j(L)$ is a
  wedge of Moore spaces if and only if $x_1=x_2=x_3=x_4=0$. 
\end{corollary}

\begin{proof}
  The first part is immediate from \Proposition{space-classify}. To wit,
  if one decomposes the quiver
  \[
  \xymatrix{ \Kh^{i,j}_{\F_2}\ar[r]_{\Sq^1}\ar@/^2pc/[rr]^{\Sq^2} &
    \Kh^{i+1,j}_{\F_2} \ar[r]_{\Sq^1}& \Kh^{i+2,j}_{\F_2} }
  \]
  as a direct sum of the nine quivers of
  \Proposition{characterise-quiver}, \Equation{coeff-of-decomposition}
  implies that the number of (X-i) summands will be $x_i$.

The `if' direction of the second part follows from the first part. For
the `only if' direction, observe that the rank of
$\Sq^2\from\Kh_{\F_2}^{i,j}\to\Kh_{\F_2}^{i+2,j}$ is
$x_1+x_2+x_3+x_4$; therefore, if $\KhSpace^j(L)$ is a wedge of Moore
spaces, $\Sq^2=0$, and hence $x_1=x_2=x_3=x_4=0$.
\end{proof}

\section{Computations}\label{sec:computations}
It can be checked from the databases \cite{KAT-kh-knotatlas} that all
prime links up to $11$ crossings satisfy the conditions of
\Corollary{width-3-determined}. Therefore, their homotopy types are
determined by the Khovanov homology $\Kh_{\Z}$ and the function
$\St$ of \Definition{st}. In \Table{computations}, we present the
values of $\St$. To save space, we only list the links $L$ for which
the function $\St(L)$ is not identically $(0,0,0,0)$; and for such links,
we only list tuples $(i,j)$ for which $\St(i,j)\neq (0,0,0,0)$. For
the same reason, we do not mention $\Kh_{\Z}(L)$ in the table; the
Khovanov homology data can easily be extracted from
\cite{KAT-kh-knotatlas}.

After collecting the data for the PD-presentations from
\cite{KAT-kh-knotatlas}, we used several Sage programs for carrying
out this computation. (To get more information about Sage, visit
\url{http://www.sagemath.org/}.)  All the programs and computations
are available at any of the following locations:

\url{http://math.columbia.edu/~sucharit/programs/KhovanovSteenrod/}

\url{https://github.com/sucharit/KhovanovSteenrod}


\begin{remark}\label{rem:mirror-sum}
  Let $\mirror{L}$ denote the mirror of
  $L$. In~\cite[Conjecture~\ref*{KhSp:conj:mirror}]{RS-khovanov} we
  conjecture that $\KhSpace^j(\mirror{L})$ and $\KhSpace^{-j}(L)$ are
  Spanier-Whitehead dual. In particular, the action of $\Sq^i$ on
  $\Kh^{*,j}(L)$ and $\Kh^{*,-j}(\mirror{L})$ should be
  transposes of each other. This conjecture provides some justification
  for the fact that \Table{computations} does not list the results for
  both mirrors of chiral knots.

  For disjoint unions, we conjecture
  in~\cite[Conjecture~\ref*{KhSp:conj:disjoint-union}]{RS-khovanov}
  that $\KhSpace(L_1\amalg L_2)$ is the smash product of
  $\KhSpace(L_1)$ and $\KhSpace(L_2)$. So, \Table{computations} only
  lists non-split links.

  The expected behavior of $\KhSpace$ under connected sums is more
  complicated:
  in~\cite[Conjecture~\ref*{KhSp:conj:unred-con-sum}]{RS-khovanov} we
  conjecture that $\KhSpace(L_1\# L_2)\simeq
  \KhSpace(L_1)\otimes_{\KhSpace(U)}\KhSpace(L_2)$, where $\otimes$
  denotes the tensor product of module spectra. So, like for Khovanov
  homology itself, the Khovanov homotopy type of a connected sum of
  links is not determined by the Khovanov homotopy types of the
  individual links: the module structures are required. Nonetheless,
  we have restricted \Table{computations} to prime links.
\end{remark}

\begin{proof}[Proof of {\Theorem{not-moore}}]
  From \Table{computations} we see that for $T_{3,4}=8_{19}$,
  $\St(2,11)=(0,1,0,0)$. Therefore, by \Corollary{width-3-determined},
  $\KhSpace^{11}(T_{3,4})$ is not a wedge sum of Moore spaces.
\end{proof}

\begin{example}\label{exam:10_145}
  Consider the knot $K=10_{145}$. From \cite{KAT-kh-knotatlas}, we know
  its Khovanov homology:
  \[
  \begin{array}{c|cccccccccc}
    &-9&-8&-7&-6&-5&-4&-3&-2&-1&0\\
\hline
    -3&.&.&.&.&.&.&.&.&.&\Z\\
    -5&.&.&.&.&.&.&.&.&.&\Z\\
    -7&.&.&.&.&.&.&\Z&\Z&.&.\\
    -9&.&.&.&.&.&.&\F_2&\F_2&.&.\\
    -11&.&.&.&.&\Z&\Z^2&\Z&.&.&.\\
    -13&.&.&.&\Z&\F_2&\F_2&.&.&.&.\\
    -15&.&.&.&\Z\oplus\F_2&\Z&.&.&.&.&.\\
    -17&.&\Z&\Z&.&.&.&.&.&.&.\\
    -19&.&\F_2&.&.&.&.&.&.&.&.\\
    -21&\Z&.&.&.&.&.&.&.&.&.
  \end{array}
  \]
  and from \Table{computations}, we know the function $\St(K)$:
  \begin{align*}
  \St(-4,-9)&=(0,0,0,1)\\
  \St(-6,-13)&=(0,0,1,0)\\
  \St(-7,-15)&=(0,1,0,0).
  \end{align*}
  Therefore, via \Corollary{width-3-determined}, we can compute
  Khovanov homotopy types:
  \begin{align*}
    \KhSpace^{-3}(K)&\sim S^0
    &\KhSpace^{-5}(K)&\sim S^0\\
    \KhSpace^{-7}(K)&\sim \Si^{-3}(S^0\vee S^1)
    &\KhSpace^{-9}(K)&\sim \Si^{-6}(\RP^2\wedge\RP^2)\\
    \KhSpace^{-11}(K)&\sim \Si^{-5}(S^0\vee S^1\vee S^1\vee S^2) 
    &\KhSpace^{-13}(K)&\sim \Si^{-8}(\RP^4/\RP^1\vee\Si\RP^2)\\
    \KhSpace^{-15}(K)&\sim \Si^{-10}(\RP^5/\RP^2\vee S^4)
    &\KhSpace^{-17}(K)&\sim \Si^{-8}(S^0\vee S^1)\\
    \KhSpace^{-19}(K)&\sim \Si^{-10}\RP^2
    &\KhSpace^{-21}(K)&\sim \Si^{-9}S^0.
  \end{align*}
\end{example}

\begin{example}
  The Kinoshita-Terasaka knot $K_1=\text{K}11n42$ and its Conway mutant
  $K_2=\text{K}11n34$ have identical Khovanov homology. From
  \Table{computations}, we see that $\St(K_1)=\St(K_2)$. Therefore, by
  \Corollary{width-3-determined}, they have the same Khovanov homotopy type.
\end{example}

The Kinoshita-Terasaka knot and its Conway mutant is an example of a
pair of links that are not distinguished by their Khovanov
homologies. In an earlier version of the paper, we asked:
\begin{question}\label{ques:is-ours-interesting}
  Does there exist a pair of links $L_1$ and $L_2$ with
  $\Kh_{\Z}(L_1)=\Kh_{\Z}(L_2)$, but $\KhSpace(L_1)\not\sim
  \KhSpace(L_2)$?
\end{question}

We provided the following partial answer:

\begin{example}
  The links $L_1=\text{L}11n383$ and $L_2=\text{L}11n393$ have
  isomorphic Khovanov homology in quantum grading $(-3)$:
  $\Kh^{-2,-3}_{\Z}=\Z^3$, $\Kh^{-1,-3}_{\Z}=\Z^3\oplus\F_2^4$,
  $\Kh^{0,-3}_{\Z}=\Z^2$, \cite{KAT-kh-knotatlas}.  However,
  $\St(L_1)(-2,-3)=(0,2,0,0)$ and $\St(L_2)(-2,-3)=(0,1,0,0)$ (\Table{computations});
  therefore, $\KhSpace^{-3}(L_1)$ is not stably homotopy equivalent to
  $\KhSpace^{-3}(L_2)$.
\end{example}

Since the first version of this paper, C.~Seed has independently
computed the $\Sq^1$ and $\Sq^2$ action on the Khovanov homology of
knots, and has answered \Question{is-ours-interesting} in the
affirmative in \cite{See-kh-squares}.

We conclude with an observation and a question. Since all prime links
up to $11$ crossings satisfy the conditions of
\Corollary{width-3-determined}, their Khovanov homotopy types are
wedges of various suspensions of $S^0$, $\RP^2$, $\CP^2$,
$\RP^5/\RP^2$, $\RP^4/\RP^1$ and $\RP^2\wedge\RP^2$; and this wedge
sum decomposition is unique since it is determined by the Khovanov
homology $\Kh_{\Z}$ and the function $\St$. \Example{10_145} already
exhibits all but one of these summands; it does not have a $\CP^2$
summand. A careful look at \Table{computations} reveals that neither
does any other link up to $11$ crossings. The conspicuous absence of
$\CP^2$ naturally leads to the following question.
\begin{question}
Does there exist a link $L$ for which $\KhSpace^j(L)$ contains
$\Si^m\CP^2$ in some\footnote{Wedge sum decompositions are in general
  not unique.} wedge sum decomposition, for some $j,m$?
\end{question}

\tiny
\begin{center}
\begin{longtable}{p{0.1\textwidth}@{}p{0.85\textwidth}}
\caption{}\label{table:computations}\\*
\toprule
$L$&$\St(L)$\\*
\midrule
$8_{19}$ & $(2, 11)\mapsto (0, 1, 0, 0)$\\
$9_{42}$ & $(-2, -1)\mapsto (0, 1, 0, 0)$\\
$10_{124}$ & $(2, 13)\mapsto (0, 1, 0, 0)$, $(5, 19)\mapsto (0, 0, 1, 0)$\\
$10_{128}$ & $(2, 11)\mapsto (0, 1, 0, 0)$\\
$10_{132}$ & $(-5, -9)\mapsto (0, 1, 0, 0)$, $(-4, -7)\mapsto (0, 0, 1, 0)$, $(-2, -3)\mapsto (0, 1, 0, 0)$\\
$10_{136}$ & $(-2, -1)\mapsto (0, 1, 0, 0)$\\
$10_{139}$ & $(2, 13)\mapsto (0, 1, 0, 0)$, $(5, 19)\mapsto (0, 0, 1, 0)$\\
$10_{145}$ & $(-4, -9)\mapsto (0, 0, 0, 1)$, $(-6, -13)\mapsto (0, 0, 1, 0)$, $(-7, -15)\mapsto (0, 1, 0, 0)$\\
$10_{152}$ & $(-7, -19)\mapsto (0, 1, 0, 0)$, $(-4, -13)\mapsto (0, 0, 1, 0)$\\
$10_{153}$ & $(0, 1)\mapsto (0, 1, 0, 0)$, $(1, 3)\mapsto (0, 0, 1, 0)$, $(-2, -3)\mapsto (0, 0, 1, 0)$, $(-3, -5)\mapsto (0, 1, 0, 0)$\\
$10_{154}$ & $(5, 17)\mapsto (0, 0, 1, 0)$, $(2, 11)\mapsto (0, 1, 0, 0)$\\
$10_{161}$ & $(-4, -11)\mapsto (0, 0, 1, 0)$, $(-7, -17)\mapsto (0, 1, 0, 0)$\\
$\text{K}11n6$ & $(0, 1)\mapsto (0, 0, 1, 0)$, $(-3, -5)\mapsto (0, 0, 1, 0)$, $(-4, -7)\mapsto (0, 1, 0, 0)$, $(-1, -1)\mapsto (0, 1, 0, 0)$\\
$\text{K}11n9$ & $(3, 13)\mapsto (0, 1, 0, 0)$, $(5, 17)\mapsto (0, 0, 1, 0)$, $(4, 15)\mapsto (0, 0, 1, 0)$, $(2, 11)\mapsto (0, 1, 0, 0)$, $(0, 7)\mapsto (0, 1, 0, 0)$, $(1, 9)\mapsto (0, 0, 1, 0)$\\
$\text{K}11n12$ & $(2, 7)\mapsto (0, 1, 0, 0)$, $(0, 3)\mapsto (0, 0, 1, 0)$, $(3, 9)\mapsto (0, 0, 1, 0)$\\
$\text{K}11n19$ & $(0, -1)\mapsto (0, 0, 1, 0)$, $(-3, -7)\mapsto (0, 0, 0, 1)$\\
$\text{K}11n20$ & $(0, 1)\mapsto (0, 0, 1, 0)$\\
$\text{K}11n24$ & $(-2, -1)\mapsto (0, 1, 0, 0)$\\
$\text{K}11n27$ & $(2, 11)\mapsto (0, 1, 0, 0)$\\
$\text{K}11n31$ & $(2, 9)\mapsto (0, 0, 0, 1)$, $(3, 11)\mapsto (0, 1, 0, 0)$, $(4, 13)\mapsto (0, 1, 1, 0)$, $(5, 15)\mapsto (0, 0, 1, 0)$, $(0, 5)\mapsto (0, 1, 0, 0)$, $(1, 7)\mapsto (0, 0, 1, 0)$\\
$\text{K}11n34$ & $(0, 1)\mapsto (0, 1, 1, 0)$, $(1, 3)\mapsto (0, 0, 1, 0)$, $(-2, -3)\mapsto (0, 0, 1, 0)$, $(-1, -1)\mapsto (0, 1, 0, 0)$, $(-4, -7)\mapsto (0, 1, 0, 0)$, $(-3, -5)\mapsto (0, 1, 1, 0)$\\
$\text{K}11n38$ & $(-1, 1)\mapsto (0, 1, 0, 0)$, $(0, 3)\mapsto (0, 0, 1, 0)$, $(-4, -5)\mapsto (0, 1, 0, 0)$, $(-3, -3)\mapsto (0, 0, 1, 0)$\\
$\text{K}11n39$ & $(1, 5)\mapsto (0, 0, 1, 0)$, $(0, 3)\mapsto (0, 1, 0, 0)$, $(-3, -3)\mapsto (0, 1, 0, 0)$, $(-2, -1)\mapsto (0, 0, 1, 0)$\\
$\text{K}11n42$ & $(0, 1)\mapsto (0, 1, 1, 0)$, $(1, 3)\mapsto (0, 0, 1, 0)$, $(-2, -3)\mapsto (0, 0, 1, 0)$, $(-1, -1)\mapsto (0, 1, 0, 0)$, $(-4, -7)\mapsto (0, 1, 0, 0)$, $(-3, -5)\mapsto (0, 1, 1, 0)$\\
$\text{K}11n45$ & $(1, 5)\mapsto (0, 0, 1, 0)$, $(0, 3)\mapsto (0, 1, 0, 0)$, $(-3, -3)\mapsto (0, 1, 0, 0)$, $(-2, -1)\mapsto (0, 0, 1, 0)$\\
$\text{K}11n49$ & $(0, 3)\mapsto (0, 0, 1, 0)$, $(-3, -3)\mapsto (0, 0, 1, 0)$, $(-4, -5)\mapsto (0, 1, 0, 0)$, $(-1, 1)\mapsto (0, 1, 0, 0)$\\
$\text{K}11n57$ & $(4, 15)\mapsto (0, 0, 1, 0)$, $(2, 11)\mapsto (0, 1, 0, 0)$, $(1, 9)\mapsto (0, 0, 1, 0)$, $(0, 7)\mapsto (0, 1, 0, 0)$, $(3, 13)\mapsto (0, 1, 0, 0)$\\
$\text{K}11n61$ & $(-1, 3)\mapsto (0, 1, 0, 0)$, $(0, 5)\mapsto (0, 0, 1, 0)$, $(2, 9)\mapsto (0, 1, 0, 0)$\\
$\text{K}11n67$ & $(2, 7)\mapsto (0, 0, 1, 0)$, $(1, 5)\mapsto (0, 1, 0, 0)$, $(-1, 1)\mapsto (0, 0, 1, 0)$, $(-2, -1)\mapsto (0, 1, 0, 0)$\\
$\text{K}11n70$ & $(-2, 1)\mapsto (0, 1, 0, 0)$, $(0, 5)\mapsto (0, 1, 0, 0)$, $(1, 7)\mapsto (0, 0, 1, 0)$\\
$\text{K}11n73$ & $(1, 5)\mapsto (0, 0, 1, 0)$, $(0, 3)\mapsto (0, 1, 0, 0)$, $(-3, -3)\mapsto (0, 1, 0, 0)$, $(-2, -1)\mapsto (0, 0, 1, 0)$\\
$\text{K}11n74$ & $(1, 5)\mapsto (0, 0, 1, 0)$, $(0, 3)\mapsto (0, 1, 0, 0)$, $(-3, -3)\mapsto (0, 1, 0, 0)$, $(-2, -1)\mapsto (0, 0, 1, 0)$\\
$\text{K}11n77$ & $(2, 13)\mapsto (0, 1, 0, 0)$, $(5, 19)\mapsto (0, 0, 1, 0)$\\
$\text{K}11n79$ & $(-2, -1)\mapsto (0, 1, 0, 0)$\\
$\text{K}11n80$ & $(-3, -7)\mapsto (0, 0, 1, 0)$, $(-4, -9)\mapsto (0, 1, 0, 0)$, $(-1, -3)\mapsto (0, 1, 0, 0)$, $(0, -1)\mapsto (0, 0, 1, 0)$\\
$\text{K}11n81$ & $(2, 11)\mapsto (0, 1, 0, 0)$\\
$\text{K}11n88$ & $(2, 11)\mapsto (0, 1, 0, 0)$\\
$\text{K}11n92$ & $(0, 1)\mapsto (0, 0, 1, 0)$\\
$\text{K}11n96$ & $(1, 5)\mapsto (0, 0, 1, 0)$, $(0, 3)\mapsto (0, 1, 0, 0)$, $(-3, -3)\mapsto (0, 1, 0, 0)$, $(-2, -1)\mapsto (0, 1, 1, 0)$\\
$\text{K}11n97$ & $(2, 5)\mapsto (0, 0, 1, 0)$, $(1, 3)\mapsto (0, 1, 0, 0)$, $(-2, -3)\mapsto (0, 1, 0, 0)$, $(-1, -1)\mapsto (0, 0, 1, 0)$\\
$\text{K}11n102$ & $(-5, -9)\mapsto (0, 0, 1, 0)$, $(-6, -11)\mapsto (0, 1, 0, 0)$, $(-2, -3)\mapsto (0, 0, 1, 0)$, $(-3, -5)\mapsto (0, 1, 0, 0)$\\
$\text{K}11n104$ & $(4, 15)\mapsto (0, 0, 1, 0)$, $(2, 11)\mapsto (0, 1, 0, 0)$, $(1, 9)\mapsto (0, 0, 1, 0)$, $(0, 7)\mapsto (0, 1, 0, 0)$, $(3, 13)\mapsto (0, 1, 0, 0)$\\
$\text{K}11n111$ & $(-1, 3)\mapsto (0, 0, 1, 0)$, $(-2, 1)\mapsto (0, 1, 0, 0)$, $(2, 9)\mapsto (0, 0, 1, 0)$, $(1, 7)\mapsto (0, 1, 0, 0)$\\
$\text{K}11n116$ & $(0, 1)\mapsto (0, 0, 1, 0)$, $(-3, -5)\mapsto (0, 0, 1, 0)$, $(-4, -7)\mapsto (0, 1, 0, 0)$, $(-1, -1)\mapsto (0, 1, 0, 0)$\\
$\text{K}11n126$ & $(2, 11)\mapsto (0, 1, 0, 0)$\\
$\text{K}11n133$ & $(-1, 3)\mapsto (0, 1, 0, 0)$, $(0, 5)\mapsto (0, 0, 1, 0)$, $(2, 9)\mapsto (0, 1, 0, 0)$\\
$\text{K}11n135$ & $(4, 13)\mapsto (0, 0, 1, 0)$, $(0, 5)\mapsto (0, 1, 0, 0)$, $(3, 11)\mapsto (0, 1, 0, 0)$, $(1, 7)\mapsto (0, 0, 1, 0)$\\
$\text{K}11n138$ & $(-2, -1)\mapsto (0, 1, 0, 0)$\\
$\text{K}11n143$ & $(2, 7)\mapsto (0, 0, 1, 0)$, $(1, 5)\mapsto (0, 1, 0, 0)$, $(-1, 1)\mapsto (0, 0, 1, 0)$, $(-2, -1)\mapsto (0, 1, 0, 0)$\\
$\text{K}11n145$ & $(1, 5)\mapsto (0, 0, 1, 0)$, $(0, 3)\mapsto (0, 1, 0, 0)$, $(-3, -3)\mapsto (0, 1, 0, 0)$, $(-2, -1)\mapsto (0, 0, 1, 0)$\\
$\text{K}11n151$ & $(-1, 3)\mapsto (0, 0, 1, 0)$, $(-2, 1)\mapsto (0, 1, 0, 0)$, $(2, 9)\mapsto (0, 0, 1, 0)$, $(1, 7)\mapsto (0, 1, 0, 0)$\\
$\text{K}11n152$ & $(-1, 3)\mapsto (0, 0, 1, 0)$, $(-2, 1)\mapsto (0, 1, 0, 0)$, $(2, 9)\mapsto (0, 0, 1, 0)$, $(1, 7)\mapsto (0, 1, 0, 0)$\\
$\text{K}11n183$ & $(5, 17)\mapsto (0, 0, 1, 0)$, $(2, 11)\mapsto (0, 1, 0, 0)$\\
$\text{L}6n1$ & $(0, 3)\mapsto (0, 0, 1, 0)$\\
$\text{L}7n1$ & $(-4, -10)\mapsto (0, 0, 1, 0)$\\
$\text{L}8n2$ & $(-2, -2)\mapsto (0, 1, 0, 0)$\\
$\text{L}8n3$ & $(-4, -11)\mapsto (0, 0, 1, 0)$, $(-6, -15)\mapsto (0, 0, 1, 0)$\\
$\text{L}8n6$ & $(-6, -13)\mapsto (0, 1, 0, 0)$\\
$\text{L}8n7$ & $(0, 4)\mapsto (0, 0, 1, 0)$\\
$\text{L}8n8$ & $(0, 2)\mapsto (0, 0, 1, 0)$, $(-2, -2)\mapsto (0, 1, 0, 0)$\\
$\text{L}9n1$ & $(-4, -10)\mapsto (0, 0, 1, 0)$\\
$\text{L}9n3$ & $(-2, -4)\mapsto (0, 1, 0, 0)$, $(-5, -10)\mapsto (0, 1, 0, 0)$, $(-4, -8)\mapsto (0, 0, 1, 0)$\\
$\text{L}9n4$ & $(-6, -16)\mapsto (0, 0, 1, 0)$, $(-4, -12)\mapsto (0, 0, 1, 0)$, $(-7, -18)\mapsto (0, 1, 0, 0)$\\
$\text{L}9n9$ & $(-7, -16)\mapsto (0, 1, 0, 0)$, $(-4, -10)\mapsto (0, 0, 1, 0)$, $(-6, -14)\mapsto (0, 0, 1, 0)$\\
$\text{L}9n12$ & $(1, 6)\mapsto (0, 0, 1, 0)$, $(-2, 0)\mapsto (0, 1, 0, 0)$\\
$\text{L}9n15$ & $(-4, -12)\mapsto (0, 0, 1, 0)$\\
$\text{L}9n18$ & $(-4, -12)\mapsto (0, 0, 1, 0)$\\
$\text{L}9n21$ & $(0, 1)\mapsto (0, 0, 1, 0)$, $(-2, -3)\mapsto (0, 0, 1, 0)$, $(-3, -5)\mapsto (0, 1, 0, 0)$\\
$\text{L}9n22$ & $(-2, -3)\mapsto (0, 1, 0, 0)$\\
$\text{L}9n23$ & $(-2, -5)\mapsto (0, 0, 1, 0)$, $(-4, -9)\mapsto (0, 0, 1, 0)$\\
$\text{L}9n25$ & $(-2, -3)\mapsto (0, 1, 0, 0)$\\
$\text{L}9n26$ & $(-2, -3)\mapsto (0, 1, 0, 0)$\\
$\text{L}9n27$ & $(0, 1)\mapsto (0, 0, 1, 0)$\\
$\text{L}10n1$ & $(-2, -4)\mapsto (0, 1, 0, 0)$, $(-4, -8)\mapsto (0, 0, 1, 0)$, $(-1, -2)\mapsto (0, 0, 1, 0)$\\
$\text{L}10n3$ & $(-2, -2)\mapsto (0, 1, 0, 0)$\\
$\text{L}10n5$ & $(1, 6)\mapsto (0, 0, 1, 0)$, $(-2, 0)\mapsto (0, 1, 0, 0)$, $(0, 4)\mapsto (0, 1, 0, 0)$\\
$\text{L}10n8$ & $(-2, -2)\mapsto (0, 1, 0, 0)$\\
$\text{L}10n9$ & $(2, 6)\mapsto (0, 0, 1, 0)$, $(-1, 0)\mapsto (0, 0, 1, 0)$, $(0, 2)\mapsto (0, 0, 1, 0)$, $(1, 4)\mapsto (0, 1, 0, 0)$, $(-2, -2)\mapsto (0, 1, 0, 0)$\\
$\text{L}10n10$ & $(-2, -6)\mapsto (0, 0, 1, 0)$, $(-5, -12)\mapsto (0, 0, 1, 0)$, $(-4, -10)\mapsto (0, 0, 1, 0)$, $(-3, -8)\mapsto (0, 1, 0, 0)$, $(-6, -14)\mapsto (0, 1, 0, 0)$\\
$\text{L}10n13$ & $(-4, -10)\mapsto (0, 0, 1, 0)$\\
$\text{L}10n14$ & $(0, 0)\mapsto (0, 0, 1, 0)$, $(-3, -6)\mapsto (0, 0, 1, 0)$, $(-1, -2)\mapsto (0, 1, 0, 0)$, $(-4, -8)\mapsto (0, 1, 0, 0)$\\
$\text{L}10n18$ & $(-3, -4)\mapsto (0, 1, 0, 0)$, $(0, 2)\mapsto (0, 1, 1, 0)$, $(1, 4)\mapsto (0, 0, 1, 0)$, $(-2, -2)\mapsto (0, 0, 1, 0)$\\
$\text{L}10n23$ & $(2, 10)\mapsto (0, 1, 0, 0)$\\
$\text{L}10n24$ & $(0, 2)\mapsto (0, 0, 1, 0)$\\
$\text{L}10n25$ & $(-4, -6)\mapsto (0, 1, 0, 0)$, $(-3, -4)\mapsto (0, 0, 1, 0)$, $(0, 2)\mapsto (0, 0, 1, 0)$, $(-1, 0)\mapsto (0, 1, 0, 0)$, $(-2, -2)\mapsto (0, 1, 0, 0)$\\
$\text{L}10n28$ & $(-2, -4)\mapsto (0, 0, 1, 0)$, $(-3, -6)\mapsto (0, 1, 0, 0)$, $(-6, -12)\mapsto (0, 1, 0, 0)$, $(-5, -10)\mapsto (0, 0, 1, 0)$\\
$\text{L}10n32$ & $(-4, -6)\mapsto (0, 1, 0, 0)$, $(-3, -4)\mapsto (0, 0, 1, 0)$, $(0, 2)\mapsto (0, 0, 1, 0)$, $(-1, 0)\mapsto (0, 1, 0, 0)$\\
$\text{L}10n36$ & $(-3, -4)\mapsto (0, 1, 0, 0)$, $(0, 2)\mapsto (0, 1, 0, 0)$, $(1, 4)\mapsto (0, 0, 1, 0)$, $(-2, -2)\mapsto (0, 0, 1, 0)$\\
$\text{L}10n37$ & $(1, 6)\mapsto (0, 0, 1, 0)$, $(-2, 0)\mapsto (0, 1, 0, 0)$\\
$\text{L}10n39$ & $(2, 10)\mapsto (0, 1, 0, 0)$\\
$\text{L}10n42$ & $(0, 6)\mapsto (0, 1, 0, 0)$, $(2, 10)\mapsto (0, 1, 0, 0)$\\
$\text{L}10n45$ & $(0, 0)\mapsto (0, 0, 1, 0)$\\
$\text{L}10n54$ & $(2, 10)\mapsto (0, 1, 0, 0)$\\
$\text{L}10n56$ & $(0, 0)\mapsto (0, 0, 1, 0)$\\
$\text{L}10n59$ & $(-3, -4)\mapsto (0, 1, 0, 0)$, $(0, 2)\mapsto (0, 1, 0, 0)$, $(1, 4)\mapsto (0, 0, 1, 0)$, $(-2, -2)\mapsto (0, 0, 1, 0)$\\
$\text{L}10n60$ & $(1, 6)\mapsto (0, 0, 1, 0)$, $(-2, 0)\mapsto (0, 1, 0, 0)$\\
$\text{L}10n62$ & $(2, 10)\mapsto (0, 1, 0, 0)$\\
$\text{L}10n66$ & $(0, 1)\mapsto (0, 0, 1, 0)$, $(-3, -5)\mapsto (0, 0, 1, 0)$, $(-4, -7)\mapsto (0, 1, 0, 0)$, $(-2, -3)\mapsto (0, 1, 0, 0)$, $(-1, -1)\mapsto (0, 1, 0, 0)$\\
$\text{L}10n67$ & $(-2, -3)\mapsto (0, 1, 0, 0)$\\
$\text{L}10n68$ & $(-4, -11)\mapsto (0, 0, 1, 0)$, $(-6, -15)\mapsto (0, 0, 1, 0)$\\
$\text{L}10n70$ & $(-2, -5)\mapsto (0, 1, 0, 0)$, $(-4, -9)\mapsto (0, 0, 1, 0)$, $(-6, -13)\mapsto (0, 1, 0, 0)$, $(-5, -11)\mapsto (0, 1, 0, 0)$\\
$\text{L}10n72$ & $(2, 7)\mapsto (0, 0, 1, 0)$, $(0, 3)\mapsto (0, 0, 1, 0)$, $(-1, 1)\mapsto (0, 1, 0, 0)$, $(-2, -1)\mapsto (0, 1, 0, 0)$\\
$\text{L}10n74$ & $(-4, -11)\mapsto (0, 0, 1, 0)$, $(-8, -19)\mapsto (0, 0, 1, 0)$, $(-5, -13)\mapsto (0, 0, 1, 0)$, $(-6, -15)\mapsto (0, 1, 0, 0)$\\
$\text{L}10n77$ & $(-6, -17)\mapsto (0, 0, 1, 0)$, $(-7, -19)\mapsto (0, 1, 0, 0)$, $(-8, -21)\mapsto (0, 0, 1, 0)$, $(-4, -13)\mapsto (0, 0, 1, 0)$\\
$\text{L}10n82$ & $(-6, -11)\mapsto (0, 1, 0, 0)$, $(-3, -5)\mapsto (0, 0, 1, 0)$\\
$\text{L}10n83$ & $(-4, -9)\mapsto (0, 0, 1, 0)$\\
$\text{L}10n84$ & $(-4, -11)\mapsto (0, 0, 1, 0)$, $(-8, -19)\mapsto (0, 1, 0, 0)$, $(-6, -15)\mapsto (0, 0, 1, 0)$, $(-7, -17)\mapsto (0, 1, 0, 0)$\\
$\text{L}10n87$ & $(-2, 1)\mapsto (0, 1, 0, 0)$, $(1, 7)\mapsto (0, 0, 1, 0)$\\
$\text{L}10n88$ & $(0, 1)\mapsto (0, 0, 1, 0)$, $(-2, -3)\mapsto (0, 0, 1, 0)$, $(-3, -5)\mapsto (0, 1, 0, 0)$\\
$\text{L}10n91$ & $(-6, -13)\mapsto (0, 1, 0, 0)$\\
$\text{L}10n93$ & $(-4, -13)\mapsto (0, 0, 1, 0)$\\
$\text{L}10n97$ & $(-2, -4)\mapsto (0, 0, 1, 0)$, $(0, 0)\mapsto (0, 0, 1, 0)$, $(-3, -6)\mapsto (0, 1, 0, 0)$, $(-4, -8)\mapsto (0, 1, 0, 0)$\\
$\text{L}10n98$ & $(0, 2)\mapsto (0, 1, 0, 0)$, $(-2, -2)\mapsto (0, 1, 0, 0)$\\
$\text{L}10n101$ & $(2, 8)\mapsto (0, 0, 1, 0)$, $(4, 12)\mapsto (0, 0, 1, 0)$, $(1, 6)\mapsto (0, 1, 0, 0)$\\
$\text{L}10n102$ & $(-4, -10)\mapsto (0, 0, 1, 0)$, $(-6, -14)\mapsto (0, 3, 0, 0)$\\
$\text{L}10n103$ & $(-4, -8)\mapsto (0, 0, 1, 0)$\\
$\text{L}10n104$ & $(-6, -12)\mapsto (0, 1, 0, 0)$, $(-4, -8)\mapsto (0, 1, 0, 0)$\\
$\text{L}10n106$ & $(-2, -4)\mapsto (0, 0, 1, 0)$, $(-4, -8)\mapsto (0, 0, 1, 0)$\\
$\text{L}10n107$ & $(0, 2)\mapsto (0, 0, 1, 0)$, $(-2, -2)\mapsto (0, 1, 0, 0)$\\
$\text{L}10n108$ & $(-2, -6)\mapsto (0, 0, 1, 0)$, $(-4, -10)\mapsto (0, 0, 2, 0)$, $(-6, -14)\mapsto (0, 0, 1, 0)$\\
$\text{L}10n111$ & $(-4, -6)\mapsto (0, 1, 0, 0)$, $(-2, -2)\mapsto (0, 1, 0, 0)$\\
$\text{L}10n112$ & $(0, 5)\mapsto (0, 0, 1, 0)$\\
$\text{L}10n113$ & $(0, 3)\mapsto (0, 0, 4, 0)$, $(-2, -1)\mapsto (0, 1, 0, 0)$\\
$\text{L}11n1$ & $(-2, -6)\mapsto (0, 0, 1, 0)$, $(-5, -12)\mapsto (0, 0, 1, 0)$, $(-4, -10)\mapsto (0, 0, 1, 0)$, $(-3, -8)\mapsto (0, 1, 0, 0)$, $(-6, -14)\mapsto (0, 1, 0, 0)$\\
$\text{L}11n5$ & $(-4, -6)\mapsto (0, 1, 0, 0)$, $(-3, -4)\mapsto (0, 0, 1, 0)$, $(-1, 0)\mapsto (0, 1, 0, 0)$, $(0, 2)\mapsto (0, 0, 1, 0)$, $(-2, -2)\mapsto (0, 1, 0, 0)$\\
$\text{L}11n6$ & $(0, 2)\mapsto (0, 0, 1, 0)$\\
$\text{L}11n8$ & $(-2, -4)\mapsto (0, 1, 0, 0)$, $(-5, -10)\mapsto (0, 1, 0, 0)$, $(-4, -8)\mapsto (0, 0, 1, 0)$\\
$\text{L}11n9$ & $(0, 0)\mapsto (0, 0, 1, 0)$, $(-2, -4)\mapsto (0, 1, 1, 0)$, $(-1, -2)\mapsto (0, 0, 1, 0)$, $(-3, -6)\mapsto (0, 1, 0, 0)$, $(-5, -10)\mapsto (0, 1, 0, 0)$, $(-4, -8)\mapsto (0, 0, 1, 0)$\\
$\text{L}11n10$ & $(-6, -16)\mapsto (0, 0, 1, 0)$, $(-4, -12)\mapsto (0, 0, 1, 0)$, $(-7, -18)\mapsto (0, 1, 0, 0)$\\
$\text{L}11n12$ & $(-7, -16)\mapsto (0, 1, 0, 0)$, $(-2, -6)\mapsto (0, 1, 0, 0)$, $(-5, -12)\mapsto (0, 1, 0, 0)$, $(-4, -10)\mapsto (0, 0, 1, 0)$, $(-6, -14)\mapsto (0, 0, 1, 0)$\\
$\text{L}11n13$ & $(-4, -10)\mapsto (0, 0, 1, 0)$\\
$\text{L}11n15$ & $(-6, -12)\mapsto (0, 0, 1, 0)$, $(-7, -14)\mapsto (0, 1, 0, 0)$, $(-2, -4)\mapsto (0, 1, 0, 0)$, $(-3, -6)\mapsto (0, 0, 1, 0)$, $(-5, -10)\mapsto (0, 1, 0, 0)$, $(-4, -8)\mapsto (0, 1, 1, 0)$\\
$\text{L}11n16$ & $(-9, -22)\mapsto (0, 1, 0, 0)$, $(-5, -14)\mapsto (0, 0, 1, 0)$, $(-4, -12)\mapsto (0, 0, 1, 0)$, $(-7, -18)\mapsto (0, 1, 0, 0)$, $(-8, -20)\mapsto (0, 0, 1, 0)$, $(-6, -16)\mapsto (0, 1, 1, 0)$\\
$\text{L}11n19$ & $(-9, -22)\mapsto (0, 1, 0, 0)$, $(-5, -14)\mapsto (0, 0, 1, 0)$, $(-4, -12)\mapsto (0, 0, 1, 0)$, $(-7, -18)\mapsto (0, 1, 0, 0)$, $(-8, -20)\mapsto (0, 0, 1, 0)$, $(-6, -16)\mapsto (0, 1, 1, 0)$\\
$\text{L}11n22$ & $(-2, -4)\mapsto (0, 1, 0, 0)$, $(-1, -2)\mapsto (0, 0, 1, 0)$, $(-4, -8)\mapsto (0, 0, 1, 0)$\\
$\text{L}11n24$ & $(2, 6)\mapsto (0, 0, 1, 0)$, $(-1, 0)\mapsto (0, 0, 1, 0)$, $(1, 4)\mapsto (0, 1, 0, 0)$, $(-2, -2)\mapsto (0, 2, 0, 0)$\\
$\text{L}11n27$ & $(-2, -4)\mapsto (0, 1, 0, 0)$, $(-5, -10)\mapsto (0, 1, 0, 0)$, $(-4, -8)\mapsto (0, 0, 1, 0)$\\
$\text{L}11n28$ & $(-7, -14)\mapsto (0, 1, 0, 0)$, $(-3, -6)\mapsto (0, 0, 1, 0)$, $(-6, -12)\mapsto (0, 0, 1, 0)$, $(-4, -8)\mapsto (0, 1, 0, 0)$\\
$\text{L}11n30$ & $(-6, -12)\mapsto (0, 0, 1, 0)$, $(-7, -14)\mapsto (0, 1, 0, 0)$, $(-2, -4)\mapsto (0, 1, 0, 0)$, $(-3, -6)\mapsto (0, 0, 1, 0)$, $(-5, -10)\mapsto (0, 1, 0, 0)$, $(-4, -8)\mapsto (0, 1, 1, 0)$\\
$\text{L}11n33$ & $(-3, -4)\mapsto (0, 1, 0, 0)$, $(2, 6)\mapsto (0, 0, 1, 0)$, $(-1, 0)\mapsto (0, 0, 1, 0)$, $(1, 4)\mapsto (0, 1, 1, 0)$, $(-2, -2)\mapsto (0, 1, 1, 0)$, $(0, 2)\mapsto (0, 1, 1, 0)$\\
$\text{L}11n38$ & $(-4, -10)\mapsto (0, 0, 1, 0)$\\
$\text{L}11n39$ & $(0, 0)\mapsto (0, 0, 1, 0)$, $(-2, -4)\mapsto (0, 1, 0, 0)$, $(-1, -2)\mapsto (0, 1, 1, 0)$, $(-3, -6)\mapsto (0, 0, 1, 0)$, $(-5, -10)\mapsto (0, 1, 0, 0)$, $(-4, -8)\mapsto (0, 1, 1, 0)$\\
$\text{L}11n41$ & $(-4, -10)\mapsto (0, 0, 1, 0)$\\
$\text{L}11n44$ & $(-4, -14)\mapsto (0, 0, 1, 0)$, $(-6, -18)\mapsto (0, 0, 1, 0)$, $(-9, -24)\mapsto (0, 1, 0, 0)$, $(-8, -22)\mapsto (0, 0, 1, 0)$, $(-7, -20)\mapsto (0, 1, 0, 0)$\\
$\text{L}11n48$ & $(-7, -16)\mapsto (0, 1, 0, 0)$, $(-2, -6)\mapsto (0, 1, 0, 0)$, $(-5, -12)\mapsto (0, 1, 0, 0)$, $(-4, -10)\mapsto (0, 0, 1, 0)$, $(-6, -14)\mapsto (0, 0, 1, 0)$\\
$\text{L}11n52$ & $(-2, -2)\mapsto (0, 1, 0, 0)$\\
$\text{L}11n53$ & $(0, 2)\mapsto (0, 0, 1, 0)$\\
$\text{L}11n54$ & $(-2, 0)\mapsto (0, 1, 0, 0)$, $(2, 8)\mapsto (0, 1, 0, 0)$, $(3, 10)\mapsto (0, 0, 1, 0)$, $(-1, 2)\mapsto (0, 1, 0, 0)$, $(1, 6)\mapsto (0, 0, 1, 0)$, $(0, 4)\mapsto (0, 1, 1, 0)$\\
$\text{L}11n57$ & $(-3, -4)\mapsto (0, 1, 0, 0)$, $(0, 2)\mapsto (0, 1, 0, 0)$, $(1, 4)\mapsto (0, 0, 1, 0)$, $(-2, -2)\mapsto (0, 1, 1, 0)$\\
$\text{L}11n59$ & $(1, 6)\mapsto (0, 0, 1, 0)$, $(-2, 0)\mapsto (0, 1, 0, 0)$\\
$\text{L}11n61$ & $(-2, -2)\mapsto (0, 1, 0, 0)$\\
$\text{L}11n62$ & $(-4, -12)\mapsto (0, 0, 1, 0)$, $(-7, -18)\mapsto (0, 1, 0, 0)$\\
$\text{L}11n64$ & $(-7, -16)\mapsto (0, 1, 0, 0)$, $(-4, -10)\mapsto (0, 1, 1, 0)$, $(-3, -8)\mapsto (0, 0, 1, 0)$, $(-6, -14)\mapsto (0, 0, 1, 0)$\\
$\text{L}11n65$ & $(-9, -20)\mapsto (0, 1, 0, 0)$, $(-8, -18)\mapsto (0, 0, 1, 0)$, $(-5, -12)\mapsto (0, 0, 1, 0)$, $(-6, -14)\mapsto (0, 1, 0, 0)$\\
$\text{L}11n68$ & $(-6, -16)\mapsto (0, 0, 1, 0)$, $(-4, -12)\mapsto (0, 0, 1, 0)$, $(-7, -18)\mapsto (0, 1, 0, 0)$, $(-9, -22)\mapsto (0, 1, 0, 0)$, $(-8, -20)\mapsto (0, 0, 1, 0)$\\
$\text{L}11n71$ & $(-7, -16)\mapsto (0, 1, 0, 0)$, $(-5, -12)\mapsto (0, 0, 1, 0)$, $(-4, -10)\mapsto (0, 0, 1, 0)$, $(-8, -18)\mapsto (0, 0, 1, 0)$, $(-9, -20)\mapsto (0, 1, 0, 0)$, $(-6, -14)\mapsto (0, 1, 1, 0)$\\
$\text{L}11n74$ & $(1, 8)\mapsto (0, 0, 1, 0)$, $(-2, 2)\mapsto (0, 1, 0, 0)$\\
$\text{L}11n75$ & $(-2, 0)\mapsto (0, 1, 0, 0)$, $(2, 8)\mapsto (0, 1, 0, 0)$, $(3, 10)\mapsto (0, 0, 1, 0)$, $(-1, 2)\mapsto (0, 1, 0, 0)$, $(1, 6)\mapsto (0, 0, 1, 0)$, $(0, 4)\mapsto (0, 0, 1, 0)$\\
$\text{L}11n87$ & $(-4, -8)\mapsto (0, 0, 1, 0)$, $(-2, -4)\mapsto (0, 1, 0, 0)$, $(0, 0)\mapsto (0, 0, 1, 0)$, $(-1, -2)\mapsto (0, 0, 1, 0)$\\
$\text{L}11n89$ & $(-6, -12)\mapsto (0, 1, 0, 0)$, $(-2, -4)\mapsto (0, 1, 1, 0)$, $(-1, -2)\mapsto (0, 0, 1, 0)$, $(-3, -6)\mapsto (0, 1, 0, 0)$, $(-5, -10)\mapsto (0, 1, 1, 0)$, $(-4, -8)\mapsto (0, 0, 1, 0)$\\
$\text{L}11n91$ & $(-7, -14)\mapsto (0, 1, 0, 0)$, $(-3, -6)\mapsto (0, 0, 1, 0)$, $(-6, -12)\mapsto (0, 0, 1, 0)$, $(-4, -8)\mapsto (0, 1, 0, 0)$\\
$\text{L}11n94$ & $(-2, 0)\mapsto (0, 1, 0, 0)$, $(2, 8)\mapsto (0, 1, 0, 0)$, $(3, 10)\mapsto (0, 0, 1, 0)$, $(-1, 2)\mapsto (0, 1, 0, 0)$, $(1, 6)\mapsto (0, 0, 1, 0)$, $(0, 4)\mapsto (0, 1, 1, 0)$\\
$\text{L}11n95$ & $(-6, -16)\mapsto (0, 0, 1, 0)$, $(-4, -12)\mapsto (0, 0, 1, 0)$, $(-7, -18)\mapsto (0, 1, 0, 0)$, $(-9, -22)\mapsto (0, 1, 0, 0)$, $(-8, -20)\mapsto (0, 0, 1, 0)$\\
$\text{L}11n98$ & $(1, 8)\mapsto (0, 0, 1, 0)$, $(-2, 2)\mapsto (0, 1, 0, 0)$\\
$\text{L}11n99$ & $(-2, 0)\mapsto (0, 1, 0, 0)$, $(2, 8)\mapsto (0, 1, 0, 0)$, $(3, 10)\mapsto (0, 0, 1, 0)$, $(-1, 2)\mapsto (0, 1, 0, 0)$, $(1, 6)\mapsto (0, 0, 1, 0)$, $(0, 4)\mapsto (0, 1, 1, 0)$\\
$\text{L}11n103$ & $(-9, -20)\mapsto (0, 1, 0, 0)$, $(-8, -18)\mapsto (0, 0, 1, 0)$, $(-5, -12)\mapsto (0, 0, 1, 0)$, $(-6, -14)\mapsto (0, 1, 0, 0)$\\
$\text{L}11n106$ & $(-7, -16)\mapsto (0, 1, 0, 0)$, $(-2, -6)\mapsto (0, 0, 1, 0)$, $(-5, -12)\mapsto (0, 0, 1, 0)$, $(-4, -10)\mapsto (0, 0, 1, 0)$, $(-3, -8)\mapsto (0, 1, 0, 0)$, $(-6, -14)\mapsto (0, 1, 1, 0)$\\
$\text{L}11n108$ & $(-2, 0)\mapsto (0, 1, 0, 0)$, $(2, 8)\mapsto (0, 1, 0, 0)$, $(3, 10)\mapsto (0, 0, 1, 0)$, $(-1, 2)\mapsto (0, 1, 0, 0)$, $(1, 6)\mapsto (0, 0, 1, 0)$, $(0, 4)\mapsto (0, 0, 1, 0)$\\
$\text{L}11n109$ & $(-9, -20)\mapsto (0, 1, 0, 0)$, $(-8, -18)\mapsto (0, 0, 1, 0)$, $(-5, -12)\mapsto (0, 0, 1, 0)$, $(-4, -10)\mapsto (0, 0, 1, 0)$, $(-6, -14)\mapsto (0, 1, 0, 0)$\\
$\text{L}11n111$ & $(-7, -16)\mapsto (0, 1, 0, 0)$, $(-2, -6)\mapsto (0, 0, 1, 0)$, $(-5, -12)\mapsto (0, 0, 1, 0)$, $(-4, -10)\mapsto (0, 0, 1, 0)$, $(-3, -8)\mapsto (0, 1, 0, 0)$, $(-6, -14)\mapsto (0, 1, 1, 0)$\\
$\text{L}11n112$ & $(2, 8)\mapsto (0, 0, 1, 0)$, $(-1, 2)\mapsto (0, 0, 1, 0)$, $(1, 6)\mapsto (0, 1, 0, 0)$, $(-2, 0)\mapsto (0, 1, 0, 0)$\\
$\text{L}11n115$ & $(-2, -2)\mapsto (0, 1, 0, 0)$\\
$\text{L}11n120$ & $(-3, -4)\mapsto (0, 1, 1, 0)$, $(-1, 0)\mapsto (0, 1, 0, 0)$, $(1, 4)\mapsto (0, 0, 1, 0)$, $(-2, -2)\mapsto (0, 1, 1, 0)$, $(-4, -6)\mapsto (0, 1, 0, 0)$, $(0, 2)\mapsto (0, 1, 1, 0)$\\
$\text{L}11n121$ & $(1, 6)\mapsto (0, 0, 1, 0)$, $(-2, 0)\mapsto (0, 1, 0, 0)$\\
$\text{L}11n122$ & $(0, 2)\mapsto (0, 1, 0, 0)$, $(1, 4)\mapsto (0, 0, 1, 0)$, $(-2, -2)\mapsto (0, 1, 0, 0)$\\
$\text{L}11n124$ & $(0, 2)\mapsto (0, 0, 1, 0)$\\
$\text{L}11n125$ & $(-4, -8)\mapsto (0, 1, 0, 0)$, $(0, 0)\mapsto (0, 0, 1, 0)$, $(-3, -6)\mapsto (0, 0, 1, 0)$, $(-1, -2)\mapsto (0, 1, 0, 0)$\\
$\text{L}11n127$ & $(-4, -6)\mapsto (0, 1, 0, 0)$, $(-2, -2)\mapsto (0, 1, 0, 0)$\\
$\text{L}11n132$ & $(-7, -16)\mapsto (0, 1, 0, 0)$, $(-4, -10)\mapsto (0, 0, 1, 0)$\\
$\text{L}11n133$ & $(-4, -14)\mapsto (0, 0, 1, 0)$, $(-7, -20)\mapsto (0, 1, 0, 0)$\\
$\text{L}11n139$ & $(3, 8)\mapsto (0, 0, 1, 0)$, $(2, 6)\mapsto (0, 1, 0, 0)$, $(0, 2)\mapsto (0, 0, 1, 0)$\\
$\text{L}11n141$ & $(-4, -12)\mapsto (0, 0, 1, 0)$\\
$\text{L}11n148$ & $(-4, -6)\mapsto (0, 1, 0, 0)$, $(-1, 0)\mapsto (0, 0, 1, 0)$\\
$\text{L}11n159$ & $(0, 6)\mapsto (0, 1, 0, 0)$, $(2, 10)\mapsto (0, 1, 0, 0)$\\
$\text{L}11n162$ & $(-4, -6)\mapsto (0, 1, 0, 0)$, $(-1, 0)\mapsto (0, 0, 1, 0)$\\
$\text{L}11n165$ & $(-6, -16)\mapsto (0, 1, 0, 0)$, $(-5, -14)\mapsto (0, 0, 1, 0)$, $(-4, -12)\mapsto (0, 0, 1, 0)$, $(-9, -22)\mapsto (0, 1, 0, 0)$, $(-8, -20)\mapsto (0, 0, 1, 0)$\\
$\text{L}11n166$ & $(1, 2)\mapsto (0, 0, 1, 0)$, $(-2, -4)\mapsto (0, 0, 1, 0)$, $(0, 0)\mapsto (0, 1, 0, 0)$, $(-3, -6)\mapsto (0, 1, 0, 0)$\\
$\text{L}11n169$ & $(-4, -14)\mapsto (0, 0, 1, 0)$, $(-7, -20)\mapsto (0, 1, 0, 0)$\\
$\text{L}11n176$ & $(-6, -16)\mapsto (0, 0, 1, 0)$, $(-4, -12)\mapsto (0, 0, 1, 0)$\\
$\text{L}11n180$ & $(-2, -4)\mapsto (0, 1, 0, 0)$, $(0, 0)\mapsto (0, 0, 1, 0)$, $(-1, -2)\mapsto (0, 0, 1, 0)$, $(-4, -8)\mapsto (0, 0, 0, 1)$\\
$\text{L}11n181$ & $(-4, -6)\mapsto (0, 1, 0, 0)$, $(-1, 0)\mapsto (0, 0, 1, 0)$\\
$\text{L}11n189$ & $(-4, -6)\mapsto (0, 1, 0, 0)$, $(-1, 0)\mapsto (0, 0, 1, 0)$\\
$\text{L}11n197$ & $(-4, -8)\mapsto (0, 0, 1, 0)$, $(-2, -4)\mapsto (0, 1, 0, 0)$, $(0, 0)\mapsto (0, 0, 1, 0)$, $(-5, -10)\mapsto (0, 1, 0, 0)$, $(-1, -2)\mapsto (0, 0, 1, 0)$\\
$\text{L}11n199$ & $(-6, -16)\mapsto (0, 1, 0, 0)$, $(-5, -14)\mapsto (0, 0, 1, 0)$, $(-4, -12)\mapsto (0, 0, 1, 0)$, $(-8, -20)\mapsto (0, 0, 1, 0)$\\
$\text{L}11n204$ & $(-4, -14)\mapsto (0, 0, 1, 0)$, $(-7, -20)\mapsto (0, 1, 0, 0)$\\
$\text{L}11n218$ & $(3, 8)\mapsto (0, 0, 1, 0)$, $(2, 6)\mapsto (0, 1, 0, 0)$, $(0, 2)\mapsto (0, 0, 1, 0)$\\
$\text{L}11n222$ & $(-4, -14)\mapsto (0, 0, 1, 0)$, $(-7, -20)\mapsto (0, 1, 0, 0)$\\
$\text{L}11n224$ & $(-4, -12)\mapsto (0, 0, 1, 0)$, $(-7, -18)\mapsto (0, 1, 0, 0)$\\
$\text{L}11n228$ & $(1, 2)\mapsto (0, 0, 1, 0)$, $(-2, -4)\mapsto (0, 0, 1, 0)$, $(0, 0)\mapsto (0, 1, 0, 0)$, $(-3, -6)\mapsto (0, 1, 0, 0)$\\
$\text{L}11n232$ & $(0, 2)\mapsto (0, 0, 1, 0)$\\
$\text{L}11n233$ & $(-4, -14)\mapsto (0, 0, 1, 0)$, $(-7, -20)\mapsto (0, 1, 0, 0)$\\
$\text{L}11n235$ & $(-4, -10)\mapsto (0, 1, 0, 0)$, $(-3, -8)\mapsto (0, 0, 1, 0)$, $(-6, -14)\mapsto (0, 0, 0, 1)$\\
$\text{L}11n236$ & $(-4, -12)\mapsto (0, 0, 1, 0)$\\
$\text{L}11n239$ & $(-3, -4)\mapsto (0, 0, 1, 0)$, $(-1, 0)\mapsto (0, 1, 0, 0)$, $(1, 4)\mapsto (0, 0, 1, 0)$, $(-2, -2)\mapsto (0, 1, 0, 0)$, $(-4, -6)\mapsto (0, 1, 0, 0)$, $(0, 2)\mapsto (0, 1, 1, 0)$\\
$\text{L}11n240$ & $(2, 8)\mapsto (0, 0, 1, 0)$, $(-1, 2)\mapsto (0, 0, 1, 0)$, $(1, 6)\mapsto (0, 1, 0, 0)$, $(-2, 0)\mapsto (0, 1, 0, 0)$\\
$\text{L}11n243$ & $(-2, -2)\mapsto (0, 1, 0, 0)$\\
$\text{L}11n244$ & $(0, 0)\mapsto (0, 0, 1, 0)$\\
$\text{L}11n252$ & $(-8, -18)\mapsto (0, 1, 0, 0)$, $(-5, -12)\mapsto (0, 0, 1, 0)$\\
$\text{L}11n253$ & $(-6, -16)\mapsto (0, 1, 0, 0)$, $(-5, -14)\mapsto (0, 0, 1, 0)$, $(-4, -12)\mapsto (0, 0, 1, 0)$, $(-8, -20)\mapsto (0, 0, 0, 1)$\\
$\text{L}11n254$ & $(-4, -14)\mapsto (0, 0, 1, 0)$, $(-7, -20)\mapsto (0, 1, 0, 0)$\\
$\text{L}11n256$ & $(2, 7)\mapsto (0, 1, 0, 0)$, $(-1, 1)\mapsto (0, 1, 0, 0)$, $(1, 5)\mapsto (0, 0, 1, 0)$, $(3, 9)\mapsto (0, 0, 1, 0)$, $(-2, -1)\mapsto (0, 1, 0, 0)$, $(0, 3)\mapsto (0, 1, 1, 0)$\\
$\text{L}11n257$ & $(-2, -5)\mapsto (0, 1, 0, 0)$, $(-4, -9)\mapsto (0, 0, 1, 0)$, $(-5, -11)\mapsto (0, 1, 0, 0)$\\
$\text{L}11n261$ & $(4, 13)\mapsto (0, 1, 0, 0)$, $(5, 15)\mapsto (0, 0, 1, 0)$, $(2, 9)\mapsto (0, 1, 1, 0)$, $(1, 7)\mapsto (0, 1, 0, 0)$\\
$\text{L}11n262$ & $(-4, -11)\mapsto (0, 0, 1, 0)$, $(-6, -15)\mapsto (0, 0, 1, 0)$, $(-7, -17)\mapsto (0, 1, 0, 0)$\\
$\text{L}11n264$ & $(-4, -9)\mapsto (0, 0, 1, 0)$\\
$\text{L}11n266$ & $(-3, -7)\mapsto (0, 1, 0, 0)$, $(0, -1)\mapsto (0, 0, 1, 0)$\\
$\text{L}11n267$ & $(1, 5)\mapsto (0, 0, 1, 0)$, $(0, 3)\mapsto (0, 1, 0, 0)$, $(-2, -1)\mapsto (0, 3, 0, 0)$\\
$\text{L}11n269$ & $(-2, -5)\mapsto (0, 0, 1, 0)$, $(-3, -7)\mapsto (0, 1, 0, 0)$, $(-4, -9)\mapsto (0, 0, 1, 0)$, $(0, -1)\mapsto (0, 0, 1, 0)$, $(-5, -11)\mapsto (0, 1, 0, 0)$\\
$\text{L}11n270$ & $(1, 5)\mapsto (0, 0, 1, 0)$, $(0, 3)\mapsto (0, 1, 0, 0)$, $(-2, -1)\mapsto (0, 1, 0, 0)$\\
$\text{L}11n272$ & $(0, 1)\mapsto (0, 0, 1, 0)$, $(-5, -9)\mapsto (0, 1, 0, 0)$, $(-2, -3)\mapsto (0, 1, 1, 0)$, $(-1, -1)\mapsto (0, 0, 1, 0)$, $(-4, -7)\mapsto (0, 0, 1, 0)$, $(-3, -5)\mapsto (0, 1, 0, 0)$\\
$\text{L}11n273$ & $(-2, -3)\mapsto (0, 1, 0, 0)$\\
$\text{L}11n274$ & $(-2, -5)\mapsto (0, 1, 0, 0)$, $(-3, -7)\mapsto (0, 0, 1, 0)$, $(-4, -9)\mapsto (0, 1, 1, 0)$, $(-6, -13)\mapsto (0, 0, 1, 0)$, $(-1, -3)\mapsto (0, 0, 1, 0)$\\
$\text{L}11n276$ & $(-7, -19)\mapsto (0, 1, 0, 0)$, $(-4, -13)\mapsto (0, 0, 1, 0)$\\
$\text{L}11n278$ & $(0, 1)\mapsto (0, 1, 0, 0)$, $(1, 3)\mapsto (0, 0, 1, 0)$, $(-2, -3)\mapsto (0, 1, 1, 0)$\\
$\text{L}11n280$ & $(-2, -3)\mapsto (0, 1, 0, 0)$\\
$\text{L}11n284$ & $(0, 5)\mapsto (0, 1, 0, 0)$, $(2, 9)\mapsto (0, 1, 0, 0)$\\
$\text{L}11n285$ & $(-2, -5)\mapsto (0, 0, 1, 0)$, $(-4, -9)\mapsto (0, 0, 1, 0)$\\
$\text{L}11n287$ & $(0, 1)\mapsto (0, 0, 1, 0)$, $(-5, -9)\mapsto (0, 1, 0, 0)$, $(-2, -3)\mapsto (0, 2, 0, 0)$, $(-1, -1)\mapsto (0, 1, 0, 0)$, $(-4, -7)\mapsto (0, 1, 1, 0)$, $(-3, -5)\mapsto (0, 0, 1, 0)$\\
$\text{L}11n288$ & $(-4, -11)\mapsto (0, 0, 2, 0)$, $(-2, -7)\mapsto (0, 0, 1, 0)$, $(-5, -13)\mapsto (0, 1, 0, 0)$, $(-6, -15)\mapsto (0, 0, 1, 0)$, $(-7, -17)\mapsto (0, 1, 0, 0)$\\
$\text{L}11n292$ & $(0, 3)\mapsto (0, 1, 0, 0)$, $(-2, -1)\mapsto (0, 1, 0, 0)$\\
$\text{L}11n293$ & $(2, 7)\mapsto (0, 0, 1, 0)$, $(-1, 1)\mapsto (0, 0, 1, 0)$, $(-3, -3)\mapsto (0, 1, 0, 0)$, $(1, 5)\mapsto (0, 1, 1, 0)$, $(-2, -1)\mapsto (0, 1, 1, 0)$, $(0, 3)\mapsto (0, 1, 1, 0)$\\
$\text{L}11n294$ & $(2, 11)\mapsto (0, 2, 0, 0)$, $(0, 7)\mapsto (0, 1, 0, 0)$\\
$\text{L}11n299$ & $(0, 1)\mapsto (0, 0, 1, 0)$, $(-5, -9)\mapsto (0, 1, 0, 0)$, $(-2, -3)\mapsto (0, 2, 0, 0)$, $(-1, -1)\mapsto (0, 1, 1, 0)$, $(-4, -7)\mapsto (0, 2, 1, 0)$, $(-3, -5)\mapsto (0, 0, 1, 0)$\\
$\text{L}11n302$ & $(2, 7)\mapsto (0, 0, 1, 0)$, $(1, 5)\mapsto (0, 0, 1, 0)$, $(0, 3)\mapsto (0, 1, 1, 0)$, $(-3, -3)\mapsto (0, 1, 0, 0)$, $(-2, -1)\mapsto (0, 0, 1, 0)$\\
$\text{L}11n303$ & $(-2, 1)\mapsto (0, 1, 0, 0)$, $(1, 7)\mapsto (0, 0, 1, 0)$\\
$\text{L}11n306$ & $(-4, -11)\mapsto (0, 0, 1, 0)$, $(-6, -15)\mapsto (0, 0, 1, 0)$\\
$\text{L}11n308$ & $(2, 7)\mapsto (0, 1, 0, 0)$, $(-1, 1)\mapsto (0, 1, 0, 0)$, $(3, 9)\mapsto (0, 0, 1, 0)$, $(0, 3)\mapsto (0, 1, 1, 0)$, $(-2, -1)\mapsto (0, 1, 0, 0)$\\
$\text{L}11n311$ & $(-1, 1)\mapsto (0, 1, 0, 0)$, $(-3, -3)\mapsto (0, 0, 1, 0)$, $(1, 5)\mapsto (0, 0, 1, 0)$, $(-4, -5)\mapsto (0, 1, 0, 0)$, $(-2, -1)\mapsto (0, 1, 0, 0)$, $(0, 3)\mapsto (0, 0, 1, 0)$\\
$\text{L}11n313$ & $(-9, -21)\mapsto (0, 1, 0, 0)$, $(-5, -13)\mapsto (0, 0, 1, 0)$, $(-4, -11)\mapsto (0, 0, 1, 0)$, $(-8, -19)\mapsto (0, 1, 1, 0)$, $(-6, -15)\mapsto (0, 1, 1, 0)$, $(-7, -17)\mapsto (0, 1, 0, 0)$\\
$\text{L}11n315$ & $(0, 1)\mapsto (0, 0, 1, 0)$, $(-5, -9)\mapsto (0, 1, 0, 0)$, $(-2, -3)\mapsto (0, 1, 1, 0)$, $(-1, -1)\mapsto (0, 1, 0, 0)$, $(-4, -7)\mapsto (0, 1, 1, 0)$, $(-3, -5)\mapsto (0, 0, 1, 0)$\\
$\text{L}11n320$ & $(-2, -5)\mapsto (0, 1, 1, 0)$, $(-5, -11)\mapsto (0, 0, 1, 0)$, $(-4, -9)\mapsto (0, 0, 2, 0)$, $(-1, -3)\mapsto (0, 0, 1, 0)$, $(-3, -7)\mapsto (0, 1, 0, 0)$, $(-6, -13)\mapsto (0, 1, 0, 0)$\\
$\text{L}11n323$ & $(-5, -13)\mapsto (0, 0, 1, 0)$, $(-4, -11)\mapsto (0, 0, 2, 0)$, $(-3, -9)\mapsto (0, 1, 0, 0)$, $(-6, -15)\mapsto (0, 1, 1, 0)$, $(-7, -17)\mapsto (0, 1, 0, 0)$, $(-2, -7)\mapsto (0, 0, 1, 0)$\\
$\text{L}11n326$ & $(0, 1)\mapsto (0, 0, 1, 0)$, $(-4, -7)\mapsto (0, 0, 1, 0)$, $(-2, -3)\mapsto (0, 1, 0, 0)$, $(-1, -1)\mapsto (0, 0, 1, 0)$\\
$\text{L}11n328$ & $(0, 3)\mapsto (0, 1, 0, 0)$, $(-2, -1)\mapsto (0, 1, 0, 0)$\\
$\text{L}11n329$ & $(-2, -5)\mapsto (0, 0, 1, 0)$, $(-3, -7)\mapsto (0, 1, 0, 0)$, $(-4, -9)\mapsto (0, 0, 1, 0)$, $(0, -1)\mapsto (0, 0, 1, 0)$\\
$\text{L}11n332$ & $(-4, -9)\mapsto (0, 0, 1, 0)$\\
$\text{L}11n333$ & $(-2, -3)\mapsto (0, 1, 0, 0)$\\
$\text{L}11n334$ & $(0, 1)\mapsto (0, 0, 1, 0)$, $(-5, -9)\mapsto (0, 1, 0, 0)$, $(-2, -3)\mapsto (0, 1, 0, 0)$, $(-1, -1)\mapsto (0, 1, 1, 0)$, $(-4, -7)\mapsto (0, 1, 1, 0)$, $(-3, -5)\mapsto (0, 0, 1, 0)$\\
$\text{L}11n336$ & $(0, 3)\mapsto (0, 0, 1, 0)$\\
$\text{L}11n337$ & $(-4, -11)\mapsto (0, 0, 1, 0)$, $(-6, -15)\mapsto (0, 0, 1, 0)$, $(-7, -17)\mapsto (0, 1, 0, 0)$\\
$\text{L}11n338$ & $(-6, -11)\mapsto (0, 1, 0, 0)$, $(-3, -5)\mapsto (0, 0, 1, 0)$\\
$\text{L}11n339$ & $(-3, -7)\mapsto (0, 0, 1, 0)$, $(-4, -9)\mapsto (0, 1, 1, 0)$, $(-6, -13)\mapsto (0, 0, 1, 0)$, $(-7, -15)\mapsto (0, 1, 0, 0)$\\
$\text{L}11n342$ & $(1, 5)\mapsto (0, 0, 1, 0)$, $(-2, -1)\mapsto (0, 3, 0, 0)$\\
$\text{L}11n345$ & $(4, 15)\mapsto (0, 0, 0, 1)$, $(2, 11)\mapsto (0, 2, 0, 0)$\\
$\text{L}11n347$ & $(-3, -7)\mapsto (0, 0, 1, 0)$, $(-4, -9)\mapsto (0, 1, 0, 0)$, $(-6, -13)\mapsto (0, 1, 0, 0)$\\
$\text{L}11n348$ & $(-2, -5)\mapsto (0, 0, 1, 0)$, $(-7, -15)\mapsto (0, 1, 0, 0)$, $(-5, -11)\mapsto (0, 0, 1, 0)$, $(-4, -9)\mapsto (0, 1, 1, 0)$, $(-3, -7)\mapsto (0, 1, 1, 0)$, $(-6, -13)\mapsto (0, 1, 1, 0)$\\
$\text{L}11n350$ & $(-1, 1)\mapsto (0, 1, 0, 0)$, $(-3, -3)\mapsto (0, 0, 1, 0)$, $(1, 5)\mapsto (0, 0, 1, 0)$, $(-4, -5)\mapsto (0, 1, 0, 0)$, $(-2, -1)\mapsto (0, 1, 0, 0)$, $(0, 3)\mapsto (0, 1, 1, 0)$\\
$\text{L}11n352$ & $(-5, -9)\mapsto (0, 0, 1, 0)$, $(-6, -11)\mapsto (0, 1, 0, 0)$, $(-2, -3)\mapsto (0, 0, 1, 0)$, $(-3, -5)\mapsto (0, 1, 0, 0)$\\
$\text{L}11n354$ & $(-2, -5)\mapsto (0, 0, 1, 0)$, $(-7, -15)\mapsto (0, 1, 0, 0)$, $(-5, -11)\mapsto (0, 0, 1, 0)$, $(-4, -9)\mapsto (0, 1, 1, 0)$, $(-3, -7)\mapsto (0, 1, 0, 0)$, $(-6, -13)\mapsto (0, 1, 1, 0)$\\
$\text{L}11n357$ & $(0, 3)\mapsto (0, 1, 0, 0)$, $(-2, -1)\mapsto (0, 1, 0, 0)$\\
$\text{L}11n359$ & $(-1, 1)\mapsto (0, 1, 0, 0)$, $(-3, -3)\mapsto (0, 0, 1, 0)$, $(1, 5)\mapsto (0, 0, 1, 0)$, $(-4, -5)\mapsto (0, 1, 0, 0)$, $(-2, -1)\mapsto (0, 1, 0, 0)$, $(0, 3)\mapsto (0, 0, 1, 0)$\\
$\text{L}11n360$ & $(0, 1)\mapsto (0, 0, 1, 0)$, $(-5, -9)\mapsto (0, 1, 0, 0)$, $(-4, -7)\mapsto (0, 0, 1, 0)$, $(-2, -3)\mapsto (0, 1, 1, 0)$, $(-1, -1)\mapsto (0, 0, 1, 0)$\\
$\text{L}11n363$ & $(-3, -7)\mapsto (0, 0, 1, 0)$, $(-4, -9)\mapsto (0, 1, 0, 0)$, $(-1, -3)\mapsto (0, 1, 0, 0)$, $(0, -1)\mapsto (0, 0, 1, 0)$\\
$\text{L}11n366$ & $(2, 11)\mapsto (0, 1, 0, 0)$\\
$\text{L}11n367$ & $(0, 1)\mapsto (0, 0, 1, 0)$\\
$\text{L}11n368$ & $(-2, -5)\mapsto (0, 0, 1, 0)$, $(-7, -15)\mapsto (0, 1, 0, 0)$, $(-5, -11)\mapsto (0, 0, 1, 0)$, $(-4, -9)\mapsto (0, 0, 1, 0)$, $(-3, -7)\mapsto (0, 1, 0, 0)$, $(-6, -13)\mapsto (0, 1, 1, 0)$\\
$\text{L}11n374$ & $(-4, -9)\mapsto (0, 0, 1, 0)$\\
$\text{L}11n375$ & $(-4, -11)\mapsto (0, 0, 1, 0)$, $(-8, -19)\mapsto (0, 0, 1, 0)$, $(-5, -13)\mapsto (0, 0, 1, 0)$, $(-6, -15)\mapsto (0, 1, 0, 0)$\\
$\text{L}11n376$ & $(-4, -9)\mapsto (0, 0, 1, 0)$\\
$\text{L}11n379$ & $(-2, -5)\mapsto (0, 0, 1, 0)$, $(0, -1)\mapsto (0, 0, 1, 0)$, $(-3, -7)\mapsto (0, 1, 0, 0)$\\
$\text{L}11n380$ & $(-2, -5)\mapsto (0, 1, 0, 0)$, $(-4, -9)\mapsto (0, 0, 1, 0)$, $(-6, -13)\mapsto (0, 1, 0, 0)$, $(-5, -11)\mapsto (0, 1, 0, 0)$\\
$\text{L}11n381$ & $(-5, -9)\mapsto (0, 1, 0, 0)$, $(-4, -7)\mapsto (0, 0, 1, 0)$, $(-2, -3)\mapsto (0, 1, 0, 0)$\\
$\text{L}11n383$ & $(0, 1)\mapsto (0, 0, 1, 0)$, $(-5, -9)\mapsto (0, 1, 0, 0)$, $(-2, -3)\mapsto (0, 2, 0, 0)$, $(-1, -1)\mapsto (0, 1, 0, 0)$, $(-4, -7)\mapsto (0, 1, 1, 0)$, $(-3, -5)\mapsto (0, 0, 1, 0)$\\
$\text{L}11n385$ & $(-5, -13)\mapsto (0, 0, 1, 0)$, $(-4, -11)\mapsto (0, 0, 2, 0)$, $(-3, -9)\mapsto (0, 1, 0, 0)$, $(-6, -15)\mapsto (0, 1, 1, 0)$, $(-7, -17)\mapsto (0, 1, 0, 0)$, $(-2, -7)\mapsto (0, 0, 1, 0)$\\
$\text{L}11n387$ & $(-2, -3)\mapsto (0, 1, 0, 0)$\\
$\text{L}11n390$ & $(0, -1)\mapsto (0, 0, 1, 0)$, $(-3, -7)\mapsto (0, 1, 0, 0)$\\
$\text{L}11n391$ & $(-3, -7)\mapsto (0, 0, 1, 0)$, $(-4, -9)\mapsto (0, 1, 1, 0)$, $(-6, -13)\mapsto (0, 0, 1, 0)$, $(-7, -15)\mapsto (0, 1, 0, 0)$\\
$\text{L}11n393$ & $(0, 1)\mapsto (0, 0, 1, 0)$, $(-3, -5)\mapsto (0, 0, 1, 0)$, $(-4, -7)\mapsto (0, 1, 0, 0)$, $(-2, -3)\mapsto (0, 1, 0, 0)$, $(-1, -1)\mapsto (0, 1, 0, 0)$\\
$\text{L}11n394$ & $(-4, -7)\mapsto (0, 0, 1, 0)$, $(-2, -3)\mapsto (0, 1, 0, 0)$, $(-1, -1)\mapsto (0, 0, 1, 0)$\\
$\text{L}11n395$ & $(1, 5)\mapsto (0, 0, 1, 0)$, $(0, 3)\mapsto (0, 1, 0, 0)$, $(-2, -1)\mapsto (0, 1, 0, 0)$\\
$\text{L}11n396$ & $(-2, -1)\mapsto (0, 1, 0, 0)$\\
$\text{L}11n397$ & $(-1, 1)\mapsto (0, 1, 0, 0)$, $(-3, -3)\mapsto (0, 0, 1, 0)$, $(1, 5)\mapsto (0, 0, 1, 0)$, $(-4, -5)\mapsto (0, 1, 0, 0)$, $(-2, -1)\mapsto (0, 2, 0, 0)$, $(0, 3)\mapsto (0, 0, 1, 0)$\\
$\text{L}11n398$ & $(-4, -11)\mapsto (0, 0, 1, 0)$, $(-6, -15)\mapsto (0, 0, 1, 0)$, $(-7, -17)\mapsto (0, 1, 0, 0)$\\
$\text{L}11n399$ & $(-3, -7)\mapsto (0, 0, 1, 0)$, $(-4, -9)\mapsto (0, 1, 1, 0)$, $(-6, -13)\mapsto (0, 0, 1, 0)$, $(-7, -15)\mapsto (0, 1, 0, 0)$\\
$\text{L}11n400$ & $(-6, -11)\mapsto (0, 1, 0, 0)$, $(-3, -5)\mapsto (0, 0, 1, 0)$\\
$\text{L}11n403$ & $(1, 5)\mapsto (0, 0, 1, 0)$, $(-2, -1)\mapsto (0, 3, 0, 0)$\\
$\text{L}11n404$ & $(-2, -5)\mapsto (0, 0, 1, 0)$, $(-3, -7)\mapsto (0, 1, 0, 0)$, $(0, -1)\mapsto (0, 0, 1, 0)$\\
$\text{L}11n406$ & $(0, 1)\mapsto (0, 0, 1, 0)$\\
$\text{L}11n408$ & $(0, 1)\mapsto (0, 0, 1, 0)$, $(-3, -5)\mapsto (0, 0, 1, 0)$, $(-4, -7)\mapsto (0, 1, 0, 0)$, $(-1, -1)\mapsto (0, 1, 0, 0)$\\
$\text{L}11n411$ & $(-9, -21)\mapsto (0, 1, 0, 0)$, $(-8, -19)\mapsto (0, 0, 1, 0)$, $(-6, -15)\mapsto (0, 1, 0, 0)$\\
$\text{L}11n413$ & $(-1, 1)\mapsto (0, 0, 1, 0)$, $(-4, -5)\mapsto (0, 1, 0, 0)$\\
$\text{L}11n414$ & $(0, 1)\mapsto (0, 0, 1, 0)$, $(-5, -9)\mapsto (0, 1, 0, 0)$, $(-4, -7)\mapsto (0, 0, 1, 0)$, $(-2, -3)\mapsto (0, 1, 1, 0)$, $(-1, -1)\mapsto (0, 0, 1, 0)$\\
$\text{L}11n417$ & $(0, 5)\mapsto (0, 1, 0, 0)$, $(2, 9)\mapsto (0, 1, 0, 0)$\\
$\text{L}11n420$ & $(-4, -11)\mapsto (0, 0, 1, 0)$, $(-2, -7)\mapsto (0, 0, 1, 0)$\\
$\text{L}11n423$ & $(0, 3)\mapsto (0, 1, 0, 0)$, $(-2, -1)\mapsto (0, 1, 0, 0)$\\
$\text{L}11n426$ & $(-4, -5)\mapsto (0, 1, 0, 0)$\\
$\text{L}11n433$ & $(2, 11)\mapsto (0, 1, 0, 0)$, $(0, 7)\mapsto (0, 1, 0, 0)$\\
$\text{L}11n436$ & $(0, 1)\mapsto (0, 0, 1, 0)$\\
$\text{L}11n439$ & $(-3, -4)\mapsto (0, 1, 0, 0)$, $(0, 2)\mapsto (0, 0, 4, 0)$, $(-2, -2)\mapsto (0, 0, 1, 0)$\\
$\text{L}11n440$ & $(1, 6)\mapsto (0, 0, 1, 0)$, $(-2, 0)\mapsto (0, 1, 0, 0)$, $(0, 4)\mapsto (0, 1, 0, 0)$\\
$\text{L}11n441$ & $(-2, -4)\mapsto (0, 1, 0, 0)$\\
$\text{L}11n443$ & $(-3, -4)\mapsto (0, 1, 0, 0)$, $(2, 6)\mapsto (0, 0, 1, 0)$, $(-1, 0)\mapsto (0, 0, 1, 0)$, $(1, 4)\mapsto (0, 1, 0, 0)$, $(-2, -2)\mapsto (0, 1, 1, 0)$, $(0, 2)\mapsto (0, 1, 1, 0)$\\
$\text{L}11n444$ & $(-2, -4)\mapsto (0, 1, 0, 0)$, $(-4, -8)\mapsto (0, 1, 0, 0)$\\
$\text{L}11n445$ & $(0, 4)\mapsto (0, 0, 1, 0)$\\
$\text{L}11n446$ & $(-3, -4)\mapsto (0, 1, 0, 0)$, $(2, 6)\mapsto (0, 0, 1, 0)$, $(-1, 0)\mapsto (0, 0, 1, 0)$, $(1, 4)\mapsto (0, 1, 0, 0)$, $(-2, -2)\mapsto (0, 1, 1, 0)$, $(0, 2)\mapsto (0, 0, 2, 0)$\\
$\text{L}11n447$ & $(-2, -4)\mapsto (0, 1, 0, 0)$\\
$\text{L}11n448$ & $(-4, -6)\mapsto (0, 0, 1, 0)$, $(-1, 0)\mapsto (0, 0, 1, 0)$, $(0, 2)\mapsto (0, 0, 2, 0)$, $(-2, -2)\mapsto (0, 1, 0, 0)$\\
$\text{L}11n449$ & $(-2, -4)\mapsto (0, 1, 0, 0)$\\
$\text{L}11n451$ & $(-2, -4)\mapsto (0, 1, 0, 0)$, $(0, 0)\mapsto (0, 0, 1, 0)$, $(-3, -6)\mapsto (0, 0, 1, 0)$, $(-1, -2)\mapsto (0, 1, 0, 0)$, $(-4, -8)\mapsto (0, 1, 0, 0)$\\
$\text{L}11n452$ & $(0, 2)\mapsto (0, 1, 0, 0)$, $(1, 4)\mapsto (0, 0, 1, 0)$, $(-2, -2)\mapsto (0, 4, 0, 0)$\\
$\text{L}11n453$ & $(-6, -16)\mapsto (0, 0, 2, 0)$, $(-4, -12)\mapsto (0, 0, 1, 0)$, $(-8, -20)\mapsto (0, 0, 1, 0)$\\
$\text{L}11n455$ & $(0, 2)\mapsto (0, 0, 1, 0)$, $(-2, -2)\mapsto (0, 0, 1, 0)$\\
$\text{L}11n456$ & $(-2, -4)\mapsto (0, 0, 1, 0)$, $(0, 0)\mapsto (0, 0, 1, 0)$, $(-3, -6)\mapsto (0, 1, 0, 0)$, $(-4, -8)\mapsto (0, 1, 0, 0)$\\
$\text{L}11n459$ & $(-2, 0)\mapsto (0, 1, 0, 0)$, $(0, 4)\mapsto (0, 1, 0, 0)$\\

\bottomrule
\end{longtable}
\end{center}
\normalsize

\bibliography{Squares}

\end{document}